\documentclass[
final,
UKenglish,
a4paper,
10pt,
bibliography=totoc
]{scrartcl}

\author{Marcel Fenzl\footnote{University of Zurich, Winterthurerstrasse 190, 8057 Z{\"u}rich, Switzerland. \newline Email: \href{mailto:marcel.fenzl@math.uzh.ch}{marcel.fenzl@math.uzh.ch}} \and Gaultier Lambert\footnote{University of Zurich, Winterthurerstrasse 190, 8057 Z{\"u}rich, Switzerland. \newline G.L. research is supported by the SNSF Ambizione grant S-71114-05-01. \newline Email: \href{mailto:gaultier.lambert@math.uzh.ch}{gaultier.lambert@math.uzh.ch}}}
\title{Precise deviations for disk counting statistics of invariant determinantal processes}

\makeatletter
\def\blfootnote{\gdef\@thefnmark{}\@footnotetext}
\makeatother

\newcommand{\subjclass}[1]{\blfootnote{\textup{2010} \textit{Mathematics Subject Classification.} #1}}

\usepackage[left=2.5cm,right=2.5cm,top=3cm]{geometry}

\usepackage{babel}
\usepackage[T1]{fontenc}
\usepackage[utf8]{inputenc}
\usepackage[useregional]{datetime2}

\usepackage{lmodern}
\usepackage{microtype}

\usepackage{fzlmath}

\usepackage{doi}
\usepackage{url}
\usepackage{enumitem}
\setlist[enumerate,1]{label={(\roman*)}}

\usepackage{hyperref}
\hypersetup{
	pdfauthor={Marcel Fenzl and Gaultier Lambert},
	pdftitle={Precise deviations for disk counting statistics of invariant determinantal processes},
	pdfsubject={2010 Mathematics Subject Classification. Primary 60F10, 60F17, 60G55; Secondary 60F05, 60B20, 82B31},
	bookmarksopen={true},
	final
}

\usepackage{amsthm}

\usepackage[capitalize,nameinlink]{cleveref}

\theoremstyle{plain}
\newtheorem{theorem}{Theorem}[section]
\newtheorem{lemma}[theorem]{Lemma}
\newtheorem{corollary}[theorem]{Corollary}
\newtheorem{proposition}[theorem]{Proposition}
\newtheorem{assumption}{Assumptions}

\theoremstyle{definition}
\newtheorem{definition}[theorem]{Definition}

\theoremstyle{remark}
\newtheorem{remark}{Remark}

\DeclareMathOperator{\Betadist}{Beta}
\DeclareMathOperator{\Betaprimedist}{Beta^\prime}

\newcommand*{\LL}{\alpha} 
\newcommand*{\DD}{\rho} 
\newcommand*{\PP}{\mathrm{p}} 
\newcommand*{\RT}{\cR} 
\renewcommand{\i}{\mathrm{i}} 

\begin{document}

\maketitle



\begin{abstract}

We consider two-dimensional determinantal processes which are rotation-invariant and study the fluctuations of the number of points in disks.
Based on the theory of mod-phi convergence, we obtain Berry-Esseen as well as precise moderate to large deviation estimates for these statistics. These results are consistent with the Coulomb gas heuristic from the physics literature. We also obtain functional limit theorems for the stochastic process $(\# D_r)_{r>0}$ when the radius $r$ of the disk $D_r$ is growing in different regimes.
We present several applications to invariant determinantal processes, including the polyanalytic Ginibre ensembles, zeros of the hyperbolic Gaussian analytic function and other hyperbolic models.
As a corollary, we compute the precise asymptotics for the entanglement entropy of (integer) Laughlin states for all Landau levels.

\end{abstract}


\pagestyle{headings}
\pagenumbering{arabic}


\section{Introduction and results}

\subsection{Introduction}\label{sec:Introduction}

Determinantal point processes arise in several contexts in probability theory, such as random matrices, free fermions, zeros of Gaussian analytic functions, domino tilings, etc.\ \cite{TSZ08PPArbitraryDimension, B15DeterminantelPointProcesses}.
In particular, the eigenvalues of a random matrix with i.i.d.\ standard complex Gaussian entries form a determinantal process in $\bC$ which is known as the Ginibre ensemble \cite{G65StatisticalEnsembles}.
This process can also be viewed as a 2-dimensional gas of electrons at inverse temperature $\beta=2$ \cite{S18CoulombInteractions}.
Random systems such as determinantal processes or Coulomb gas, due to their built-in repulsion, exhibit remarkable \emph{hyperuniformity} or \emph{rigidity} properties, that is the fluctuations of the number of points in a given region is smaller than compared to a Poisson process with the same intensity.
In fact, a well-known principle from electrostatics introduced in \cite{MY80ChargeFluctuations} states that the variance of the number of points in a box should grow like the surface area instead of the volume as in case of i.i.d.\ random points.
Based on this physical reasoning, Jancovici, Lebowitz, Manificat \cite{JLM93LargeChangeFluctuations} predicted the large deviation asymptotics for the number of points in large balls.
We aim at rigorously verifying this conjecture for the Ginibre ensemble, as well as to obtain precise deviations for counting statistics.
Our method relies on the determinantal structure of the Ginibre ensemble combined with the theory of mod-phi convergence.
In \cref{sec:Applications}, we discuss other determinantal processes which also fit in our framework, including fermion point processes associated with higher Landau levels, the zero process of the hyperbolic Gaussian analytic function as well as other ensembles which are invariant (in law) under certain groups of Möbius transformations.
Our findings are summarized in \cref{sec:sum}.

It is easier to present our results in the infinite volume case. The (infinite) Ginibre point process is the microscopic scaling limit in the bulk of the Ginibre eigenvalue process.
It turns out that this limit is universal for many different ensembles of random matrices, including normal matrices with a general potential \cite{AHM11FluctuationsEigenvalues} and random matrices with general i.i.d.\ entries \cite{TV15UniversalityNonHermitian, CES19EdgeUniversality}.
Let us denote by $\xi$ this point process (viewed as a discrete random measure on $\bC$). $\xi$ is a determinantal process with kernel $K_0(z,w) = e^{z\bar{w}}$ with respect to the standard Gaussian measure $\frac{\dif \mu }{\dif m}(z) = \frac{1}{\pi }e^{-\abs{z}^2}$, where $m$ denotes the Lebesgue measure.
It follows that this process is translation-invariant on $\bC$ with intensity $1/\pi$.
Since it is also rotation-invariant, it has the following remarkable property which was first obtained by Kostlan \cite{K92SpectraGaussianMatrices}, see \cite{HKPV06DeterminantalProcesses} and \cite{D18PowersGinibre} for generalizations.

\begin{theorem}[\cite{K92SpectraGaussianMatrices}]\label{thm:Kostlan}
	The set of square-modulus of the points of $\xi $ has the same distribution as $\set{\Gamma_1,\Gamma_2,\dotsc}$, where $\Gamma_k$ are independent gamma-distributed random variables with shape $k$ and rate $1$.
\end{theorem}

We are interested in the fluctuations of the number of points in a disk $D_r=\set{\abs{z} < r}$ for a large radius $r>0$.
\cref{thm:Kostlan} shows that this counting statistic is given by $\xi (D_r) = \sum_{k=1}^\infty \ind_{\Gamma_k \le r^2}$ in law, from which one can easily deduce a central limit theorem:
\begin{equation}\label{eq:CLT}
	\pi ^\frac{1}{4} \frac{\xi (D_r) - r^2}{r^\frac{1}{2}} \xrightarrow[r\to \infty ]{\mathrm{d}} \Normaldist_{0,1},
\end{equation}
with $\Normaldist_{0,1}$ being a standard Gaussian random variable, see \cite{MY80ChargeFluctuations}.
The CLT \eqref{eq:CLT} can also be deduced from a general result of Soshnikov \cite{S02GaussianLimitDPP} for the fluctuations of linear statistics of determinantal processes.
In particular, the statistic $\xi (D_r)$ has variance
proportional to the surface area of $D_r$.
As for large deviations of $\xi (D_r)$, it is argued in \cite{JLM93LargeChangeFluctuations} by use of electrostatic arguments that for any $x\in \bR_+$,
\begin{align}\label{JLM93}
	- \log \Pr[\big]{\xi (D_r)-r^2 \ge xr^\gamma} =
	\begin{cases}
		\frac{x^2}{2}r^{2\gamma -1}\bigl(\sqrt{\pi }+o(1)\bigr),&\frac{1}{2}\le \gamma <1,\\
		\frac{x^3}{6}r^{3\gamma -2}\bigl(1+o(1)\bigr),&1<\gamma <2,\\
		\frac{(\gamma -2)x^2}{2}r^{2\gamma} \log r\bigl(1+o(1)\bigr),&2<\gamma.
	\end{cases}
\end{align}

In this paper, we rigorously establish \eqref{JLM93} and investigate the precise deviations of the statistics $\xi (D_r)$ at different scales complementing the above-mentioned results. A similar study is also done for the eigenvalues of large Ginibre random matrices and our findings agree inside the bulk. We also describe the finite-size effects occurring at the edge in \cref{sec:FiniteGinibre}.

Our analysis relies on the theory of mod-phi convergence, a concept which provides a unified framework to study deviations of certain random variables beyond the central limit theorem, see \cite{FMN16ModPhiConvergence}.
In particular, this allows us to makes precise the idea that $\xi (D_r) = \sum_{k=1}^\infty \ind_{\Gamma_k \le r^2}$ behaves like its mean $r^2$ plus a sum of $r$ independent, roughly identically distributed (re-centred) Bernoulli random variables. For a review of mod-phi convergence, we refer to \cref{sec:Background}.
As a result, we obtain the following precise asymptotics for the counting statistics $\xi(D_r)$ of the infinite Ginibre process.

\begin{proposition}\label{thm:GinibrePreciseAsymptotics}
	There exist two functions $I\colon\bR\to\bR_+$ with $I(0)=I'(0)=0$ and $I''>0$ and $\psi\colon\bR\to\bR_+$ with $\psi(0)=1$ such that it holds locally uniformly for $y\in\bR_+$,
	\begin{equation*}
		\Pr[\big]{\xi(D_r) - r^2 \ge r y}
		= e^{-rI(y)} \sqrt{\frac{I''(y)}{2\pi r}} \frac{\psi(I'(y))}{1-e^{-I'(y)}} \bigl(1+O(r^{-1})\bigr).
	\end{equation*}
\end{proposition}

$I$ is usually called the \emph{rate function} and its convex conjugate $\Lambda_0$, as well as $\psi_0$ are given explicitly in the statement of \cref{thm:GinibreTypeModPhi} below with $\LL=0$.
Since our estimates are uniform, \cref{thm:GinibrePreciseAsymptotics} covers the full regime $\gamma\le 1$ in \eqref{JLM93}, including correction terms. Moreover, this implies Berry-Esseen estimates and an extension of the CLT \eqref{eq:CLT} in the optimal window where the Gaussian approximation is valid (see \cref{cor:mod}).
We also verified that our expression for $I$ matches with the prediction from \cite[(2.9a) and (2.9b)]{JLM93LargeChangeFluctuations}.
This rate function as well as the asymptotics \eqref{JLM93} have also been obtained recently in the physics literature using different methods, \cite{LGCCKMS19ComplexGinibre}.
 The large deviation estimates from \eqref{JLM93} are verified in \cref{thm:ldp} below.
In fact, we show that these asymptotics hold for all Ginibre-type ensembles associated with higher Landau levels (see \cref{sec:Ginibre} for a physical motivation of these ensembles). \cref{thm:GinibrePreciseAsymptotics} describes the crossover from moderate to large deviations when $\gamma=1$. We also compute the rate function which appears for the critical value $\gamma=2$.
Some further motivations explaining the emergence of the different regimes and rate functions, as well as the differences between the different Landau levels, are given in \cref{sec:Ginibre}. Importantly, our method also allows us to obtain asymptotics of the entanglement entropy for counting statistics, see \cref{thm:arealaw}.

To put our results in perspective, there is an extensive literature on large deviation estimates for \emph{hyperuniform} two-dimensional point processes to compare with.
Most relevant in our context are the references \cite{BZ98LDCircularLaw, S06LDFermionPointProcess} on eigenvalues of the Ginibre ensemble, \cite{K06Overcrowding, NSV08JLMRandomComplexZeroes} for zeros of Gaussian analytic functions and \cite{ZZ10LDZerosRandomPolynomials, B16LargeDeviationsRandomPolynomials, BZ17UniversalLargeDeviations} for zeros of random polynomials, as well as \cite{ST05RandomComplexZeroes, GN19GaussianComplexZeroesForbiddenRegion} on the so-called hole probabilities of the Gaussian entire function.
We also refer to \cite[Chapter 7]{HKPV09ZerosOfGAF} and \cite{GN18PPHoleEventsLD} for a review of these results and further references.
For the Ginibre ensemble, it is also worth mentioning the rates of convergence to the circular law obtained recently in \cite{GJ18RateCircularLaw}.

We can also investigate the statistics $\bigl(\xi (D_r) \bigr)_{r\in\bR_+}$ as a stochastic process.
To do so, let us use the notation
\begin{equation*}
	X_r(t) = \pi ^\frac{1}{4}\frac{\xi (D_{rt})-(rt)^2}{r^\frac{1}{2}},\qquad t\in\bR_+.
\end{equation*}
The random variable $X_r(t)$ has mean 0 and it has finite variance. We obtain the following central limit theorem.
\begin{proposition}\label{thm:GinibreFCLTFDD}
	The following convergence in finite-dimensional distributions holds as $r\to\infty$:
		\begin{equation*}
			\bigl(X_r(t)\bigr)_{t\in\bR_+} \xrightarrow{\text{fdd}} \bigl(t^\frac{1}{2}N_t\bigr)_{t\in\bR_+},
		\end{equation*}
	where $(N_t)_{t\in \bR_+}$ is a white noise, i.e.\ $N_t$ are i.i.d.\ standard Gaussian random variables.
\end{proposition}

This lack of correlations is interesting, as it is in contrast with what happens if the points where independent. If $\cP$ is a homogeneous Poisson point process with intensity $1/\pi$ on $\bC$, then we verify for $r\to\infty$ that
\begin{equation*}
	\biggl(\frac{\cP(D_{rt}) - (rt)^2}{r}\biggr)_{t\in \bR_+} \xrightarrow{\cS(\bR_+)} (W_t)_{t\in \bR_+},
\end{equation*}
where $(W_t)_{t\in \bR_+}$ is a standard Brownian motion and the convergence holds with respect to the Skorokhod topology\footnote{Recall that a sequence of càdlàg functions $X_n \colon \bR \to \bR$ is said to converge to a càdlàg function $X\colon \bR \to \bR$ in the Skorokhod topology if there exists a sequence of strictly increasing, continuous functions $(w_n)_{n\in\bN}$ such that for all compact sets $B\subseteq \bR$ it holds $\sup_{t\in B} \abs{w_n(t)-t} \to 0$ and $\sup_{t\in B} \abs{X_n(w_n(t)) - X(t)} \to 0$.}.
This difference lies in the \emph{rigidity} of the Ginibre configurations. Indeed, the Coulomb gas heuristic suggests that only the particles lying near the boundary of $D_r$ contribute significantly to the fluctuations of $\xi (D_r)$.
Therefore, the statistics $\xi (D_{r})$ for disks which are macroscopically separated become independent in the limit.
This also suggests that in order to obtain a non-trivial limit, we should let the radius vary in a microscopic way.
We obtain the following functional central limit theorem.

\begin{proposition}\label{thm:GinibreFCLTSkorokhod}
	The following weak convergence holds true with respect to the Skorokhod topology as $r\to\infty$:
	\begin{equation*}
		\Bigl(X_r\Bigl(1+\frac{t}{r}\Bigr)\Bigr)_{t\in\bR} \xrightarrow{\cS(\bR)} (G_t)_{t\in \bR },
	\end{equation*}
	where $(G_t)_{t\in \bR }$ is a stationary, centred Gaussian process with covariance kernel given by
	\begin{equation*}
		\Ex{G_sG_t}
		=e^{-(t-s)^2} - 2\sqrt{\pi }(t-s)\Pr[\big]{\Normaldist_{0,1} \ge \sqrt{2}(t-s)}
		\qquad\text{for }s\le t.
	\end{equation*}
\end{proposition}

The proofs of \cref{thm:GinibrePreciseAsymptotics}, \cref{thm:GinibreFCLTFDD} and \cref{thm:GinibreFCLTSkorokhod} follow from general results presented in \cref{sec:MainResults} which apply to rotation-invariant determinantal processes.
The applications of our general results to the finite and infinite Ginibre ensembles are given in \cref{sec:GinibreProof}.

\begin{remark} \label{rk:GP}
The process $(G_t)_{t\in\bR}$ has the same regularity as Brownian motion. By our general results, we find that its covariance kernel is also given by
	\begin{equation*}
		\Ex{G_sG_t} = \sqrt{\pi } \int_{\bR} \Pr{\Normaldist_{0,1} \ge x-2s}\Pr{\Normaldist_{0,1}\le x-2t} \dif x
		\qquad\text{for }s\le t.
	\end{equation*}
	The formula presented in \cref{thm:GinibreFCLTSkorokhod} follows from an exact computation (see \cref{sec:CovarianceGinibre}).
\end{remark}

\begin{remark}
A classical result asserts that the (asymptotic) fluctuations of the Ginibre eigenvalues are described by an $H^1$-valued Gaussian process \cite{RV07CircularLawAndGFF,RV07ComplexDeterminantalProcesses}.
This means that for any function $f \in H^1 \cap L^1(\mu)$,
\begin{equation} \label{smoothclt}
	\xi(f_\rho) - \int_\bC f_\rho(z) \frac{\dif m(z)}{\pi} \xrightarrow[\rho\to \infty ]{\mathrm{d}} \norm{f}_{H^1(\bC)} \Normaldist_{0,1},
\qquad\text{where }f_\rho(z)= f\Bigl(\frac{z}{\rho}\Bigr)
\end{equation}
and $\norm{f}_{H^1(\bC)}^2 = \frac{1}{4\pi} \int_\bC \abs{\nabla f(z)}^2 \dif m(z)$. If one considers $f = \ind_{\bD} * \phi_\rho$, where $\phi_\rho$ is an approximate $\delta$-function at microscopic scale $\rho^{-1/2}$, then one expects from \eqref{smoothclt} that $ (\xi(D_\rho)-\rho^2) \sim \sqrt{\rho}\Normaldist_{0,1}$ as $\rho\to\infty$. Similarly, by considering the test function $f = \sum_{j=1}^m \alpha_j \ind_{D_{t_j}}* \phi_\rho$ with $0<t_1 <\cdots < t_m$ and $\alpha_1,\dotsc, \alpha_m\in\bR$ for $m\in\bN$, one can informally recover the multi-dimensional central limit theorem given in \cref{thm:GinibreFCLTFDD}.
\end{remark}


\subsection{Background and notation}\label{sec:Background}

\paragraph{Notation.}
Throughout this article, we denote by $\bR_+ = \intco{0,\infty} $, $\bN_0 = \set{0,1,2,\dotsc}$ and the horizontal strip by $\bS = \set{z\in\bC : \abs{\Im{z}} < \pi }$. Let the disk of radius $r>0$ in the complex plane be $D_r = \set{z\in\bC \given \abs{z} \le r}$. For binomial coefficients, we use the convention: $\binom{n}{m} = 0$ if $n<m$. We use $c_\alpha,C_\alpha$ for numerical constants depending on $\alpha$ which may change throughout chains of inequalities.

As in the introduction, $\Normaldist_{0,1}$ stands for a standard Gaussian random variable, and we let $\Phi_0(x) = \Pr{\Normaldist_{0,1}\ge x}$ be its probability tail function.

We denote the cumulant generating function of a centred Bernoulli distribution with parameter $p$ by
\begin{equation} \label{def:kappa}
	\kappa_p(z) = \log \Ex[\big]{e^{z(\Bernoullidist_p-p)}} = \log\bigl(1+p(e^z-1)\bigr)-pz
\end{equation}
for $p\in\intcc{0,1}$ and $z\in\bS$ and by $\dot\kappa_p(z) = \diff{\kappa_p(z)}{p}$, etc.\ its derivatives. For a real-valued random variable $X$, let $\varkappa^{(q)}(X)$ be its $q$-th cumulant.

A function $f\colon\bR\to\bC$ is called absolutely continuous if $f\in W^{1,1}(\bR)$, that is $f\in L^1(\bR)$ and it has a weak derivative $f'\in L^1(\bR)$.
Notice that any function $f\in W^{1,1}(\bR)$ is continuous and satisfies $\abs{f(x)-f(x')} \le \int_{x_0}^\infty \abs{f'(y)}\dif y$ for all $x,x'\ge x_0$. Since this integral converges to 0 as $x_0\to\infty$, this shows that $f(x)$ has a finite limit as $x\to +\infty$. Since $f\in L^1(\bR)$, this limit is 0 and the same holds as $x\to-\infty$. This shows that $f\in C_0(\bR)$ for any $f\in W^{1,1}(\bR)$ and we will use this fact throughout the rest of the paper.

\paragraph{Determinantal point processes.}

Determinantal processes is a class of point processes which has been introduced by Macchi to describe free fermions \cite{M75CoincidenceApproach}. Such processes are characterized by the property that all their correlation functions (with respect to a reference measure $\mu$) are of the form
\begin{equation*}
	\det_{i,j\in\set{1,\dotsc,n}} \bigl[ K(x_i, x_j) \bigr], \qquad n\in\bN.
\end{equation*}
The function $K$ is called the correlation kernel. If it defines a locally trace-class (integral) operator $\mathrm{K}$, then this operator characterizes the law of the process, see \cref{rk:lawuniqueness} below. A short discussion on point processes and their correlation functions is given at the beginning of \cref{sec:ModulusPoints}, and we refer to \cite{S00DeterminantalRandomPointFields,ST03RandomPointFieldsI,J06RandomMatrices,HKPV06DeterminantalProcesses,B15DeterminantelPointProcesses} as introductions to determinantal processes, including several examples. For instance, the eigenvalues of unitary invariant ensembles like the GUE and the CUE and the eigenvalues of normal random matrices are determinantal processes \cite{AGZ09RandomMatrices, AHM11FluctuationsEigenvalues}.
In particular, the correlation kernel of the $n\times n$ Ginibre ensemble is $K_N(z,w) = \sum_{k=0}^{N-1} \frac{z^k\bar{w}^k}{k!}$ with respect to $\frac{\dif \mu }{\dif m}(z) = \frac{1}{\pi }e^{-\abs{z}^2}$ on $\bC$.

In the present article, we are interested in certain two-dimensional processes which are rotation-invariant such as the \emph{(polyanalytic) Ginibre ensembles} whose kernels with respect to the Gaussian measure $\frac{\dif \mu }{\dif m}(z) = \frac{1}{\pi }e^{-\abs{z}^2}$ are given by
\begin{equation*}
	K_\LL(z,w) = L_\LL\bigl(\abs{z-w}^2\bigr)e^{z\bar{w}} , \qquad z,w\in\bC,\ \alpha\in\bN_0,
\end{equation*}
where $L_\LL$ denotes the orthonormal Laguerre polynomial of degree $\LL$. These ensembles generalize the infinite Ginibre point process ($\LL=0$). They describe the thermodynamic limit of two-dimensional free fermions in a uniform magnetic field in the Landau level $\LL$, see \cite{Z06MatrixModels}.

Another example that we consider for comparison purposes is the zeros set of the \emph{hyperbolic Gaussian analytic function}
$F(z)= \sum_{k=0}^\infty a_k z^k$ with $(a_k)_{k\in\bN_0}$ being a sequence of i.i.d.\ standard complex Gaussian random variables.
In \cite{PV05ZerosGaussianPowerSeries} it is shown that $\set{z\in\bD : F(z)=0 }$ defines a simple point process which is determinantal with kernel $K(z,w) = \frac{1}{\pi (1-z\bar{w})^2}$ with respect to the Lebesgue measure $m$ on the unit disk $\bD = \set{z\in\bC: \abs{z} < 1}$.

We focus on rotation-invariant determinantal processes because our analysis relies on the fact that counting statistics for nested disks centred at the origin are so-called \emph{simultaneously observable}, see \cite[Proposition 2.8]{ST03RandomPointFieldsII} or \cite[Proposition 9]{HKPV06DeterminantalProcesses} or \cref{lem:ModulusPoints}.
Namely, if $\xi$ is a rotation-invariant determinantal process in $\bC$ associated with a locally trace-class, self-adjoint operator $\mathrm{K}$, then for any $r_\ell>\dotsb>r_1>0$, it holds
\begin{equation}\label{eq:TS}
	\bigl( \xi(D_{r_k}) \bigr)_{i=1}^\ell
	\overset{\mathrm{d}}{=} \biggl( \sum_{k \in\bN} \ind_{U_k \le \lambda_k(r_i)}\biggr)_{i=1}^\ell ,
\end{equation}
where $(U_k)_{k\in\bN}$ is a sequence of i.i.d.\ random variables which are uniform in $\intcc{0,1}$ and $(\lambda_k(r))_{k\in\bN}$ denotes the eigenvalues of the operator $\mathrm{K}|_{D_r}$.
In particular, it follows from the general theory that $\lambda_k(r)\in\intcc{0,1}$ for all $k\in\bN$ and $r>0$, see \cite[Theorem 3]{S00DeterminantalRandomPointFields} or \cite[Theorem 22]{HKPV06DeterminantalProcesses}.

\paragraph{Mod-phi convergence.}
The concept of mod-phi convergence was first introduced in \cite{KN10ModPoissonConvergence,JKN11ModGaussianConvergence,KN12ModGaussianConvergence} and further developed in the series of works \cite{FMN16ModPhiConvergence,FMN17ModPhiConvergenceIII,FMN19ModPhiConvergenceII} with the goal of providing a unified framework to obtain precise information on a sequence satisfying a central limit theorem, such as Berry-Esseen estimates and precise moderate to large deviation estimates. By now, the theory of mod-phi convergence has been successfully applied in different contexts such as asymptotic combinatorics \cite{FMN18ModGaussianModuliSpaces}, in connection to the Ising model \cite{MN15ModGaussianIsing} and determinants of classical random matrix ensembles \cite{DHR19RandomDeterminants}. Intuitively, the idea is to compare a sequence of random variables $(X_n)_{n\in\bN}$ with a sum of independent, identically distributed random variables at the level of cumulant generating functions to obtain precise tail estimates.
As we are interested in counting statistics, we introduce the concept of mod-phi convergence for discrete random variables following the approach of \cite{FMN16ModPhiConvergence}.

\begin{definition}\label{def:ModPhi}
	The sequence of $\bZ$-valued random variables $(X_n)_{n\in\bN}$ converges in the mod-phi sense with parameters $(t_n)_{n\in\bN}$, normalizing \emph{distribution} $\nu$ on $\bR$ and limiting function $\psi$ (analytic in the strip $\bS = \set{z\in\bC \given \abs{\Im{z}} < \pi }$) if as $t_n\to \infty$, it holds locally uniformly on $\bS$,
	\begin{equation*}
		\frac{\Ex[\big]{e^{zX_n}}}{e^{t_n \Lambda (z)}} \xrightarrow[n\to \infty ]{} \psi (z),
	\end{equation*}
where $\Lambda (z) = \log \int_\bR e^{zx} \dif \nu(x)$ denotes the cumulant generating function of $\nu$.
We also assume that the following technical condition holds: there exists $C, \epsilon>0$ such that for all $x\in\intcc{-\epsilon,\epsilon}$ and all $\delta>0$,
\begin{equation}\label{ass:DecayCGF}
	\max_{y\in\intcc{-\pi ,\pi }\setminus\intoo{-\delta ,\delta }} \abs[\big]{\exp\bigl(\Lambda (x+\i y)-\Lambda (x)\bigr)} \le 1-C\delta ^2.
\end{equation}
\end{definition}

In contrast to the works mentioned above, the \emph{distributions} $\nu$ observed in this article are discrete signed measures, which is a rather surprising fact.
To fix the normalization, we assume that the total mass of~$\nu$ is 1 in which case $\psi(0)=1$.
In \cref{def:ModPhi}, we implicitly assume that $\int_\bR e^{zx} \abs{\nu}(\dif x) <\infty$ for all $z>0$ where $\abs{\nu}$ denotes the variation of $\nu$ and that $\log(\int_\bR e^{zx} \dif \nu(x))$ is well-defined for the principal branch of $\log$ in the strip $\bS$. Then $\Lambda$ is an analytic function in this strip and
\begin{equation*}
	\Lambda'(0) = \lim_{n\to\infty} \frac{\Ex{X_n}}{t_n}
	\qquad\text{and}\qquad
	\Lambda''(0) = \lim_{n\to\infty} \frac{\Var{X_n}}{t_n}.
\end{equation*}

We always assume that $\Lambda''(0)>0$. Then we deduce from the above mod-phi convergence the following precise moderate and large deviation estimates, \cite[Section 3]{FMN16ModPhiConvergence}.
\begin{theorem}\label{thm:ModPhiConclusions}
	Let $(X_n)_{n\in\bN}$ be a sequence of centred real-valued random variables converging in the mod-phi sense of \cref{def:ModPhi}.
	Let $I:\bR\to\intcc{0,\infty}$ be the convex conjugate of $\Lambda $. Then, the following asymptotic holds locally uniformly for $y\in \bR_+$, as $n\to \infty $,
	\begin{equation} \label{eq:Mod}
		\Pr[\big]{X_n \ge t_ny} = e^{-t_n I(y)} \sqrt{\frac{I''(y)}{2\pi t_n}}\frac{\psi(I'(y))}{1-e^{-I'(y)}}\bigl(1+O(t_n^{-1})\bigr).
	\end{equation}
\end{theorem}
$I$ is usually called the \emph{rate function}.
Recall that $I(y) = \sup_{x\in\bR}\set{ xy-\Lambda (x)}$ is convex on $\bR$ and that $I\ge 0$ since $\Lambda(0)=0$ because of our normalization. The RHS of \eqref{eq:Mod} needs to be interpreted as $0$ whenever $I(y)= \infty$.
The main difference with \cite[Theorem 3.2.2]{FMN16ModPhiConvergence} is that we only require that $\Lambda(z)$ exists for $z\in\bS$, and we assume the technical condition \eqref{ass:DecayCGF} (which is sufficient for the result in \cite{FMN16ModPhiConvergence}).
Since we are restricted to $\bZ$-valued random variables, the proof (which is based on Fourier's inversion formula) goes through directly.
In the formulation of \cref{thm:ModPhiConclusions}, we assume that $y>0$, but one can also obtain the left tail by applying the results to the sequence $(-X_n)_{n\in\bN}$.

Note that compared to usual large deviation estimates, the asymptotic \eqref{eq:Mod} holds for
$\Pr{X_n \ge t}$, rather than its logarithm, uniformly for all $t\in\bR$ up to the order $t_n$.
Thus, if we normalize the random variables $X_n$, we deduce the following results.

\begin{corollary} \label{cor:mod}
 Let $\widetilde{X}_n = \frac{X_n}{\sqrt{t_n\Lambda''(0)}}$.
Under the assumptions of \cref{thm:ModPhiConclusions}, it holds
\begin{itemize}
		\item \emph{Berry-Esseen estimate:} There exists a constant $C>0$ such that
		\begin{equation*}
			\sup_{x\in \bR } \abs[\big]{\Pr[\big]{\widetilde{X}_n \ge x} - \Pr[\big]{\Normaldist_{0,1} \ge x}} \le \frac{C}{\sqrt{t_n}}.
		\end{equation*}
		\item \emph{Extended central limit theorem\footnote{In general, the condition $x_n=o(t_n^{1/6})$ for the extended CLT is optimal.}:} It holds for any sequence $x_n=o(t_n^{1/6})$ that
		\begin{equation*}
			\Pr[\big]{\widetilde{X}_n\ge x_n} = \Pr[\big]{\Normaldist_{0,1} \ge x_n}\bigl(1+o(1)\bigr).
		\end{equation*}
		\item \emph{Precise moderate deviations:} If $x_n \ge 0$ is any sequence with $x_n=o(\sqrt{t_n})$, then
		\begin{equation*}
			\Pr[\big]{\widetilde{X}_n \ge x_n} = \frac{\exp\Bigl(-t_n I\Bigl(\frac{x_n}{\sqrt{I''(0)t_n}}\Bigr)\Bigr)}{\sqrt{2\pi }x_n}\bigl(1+o(1)\bigr).
		\end{equation*}
\end{itemize}
\end{corollary}

\subsection{Main results}\label{sec:MainResults}

We take a general viewpoint and consider a rotation-invariant determinantal process $\xi$ on $\bC$ with a Hermitian-symmetric correlation kernel $K$ (with respect to a Radon measure $\mu$).
Let us assume that $K(z,w)=\sum_{k\in\bZ} \sum_{j =1}^{\ell_k} \rho_{k,j}(\abs{z}) z^k \rho_{k,j}(\abs{w})\bar{w}^k$ for $z,w\in\bC$ with $\rho_{k,j} \colon \bR_+ \to \bR$ normalized such that for any $k\in\bZ$: $\int_{\bC} \rho_{k,i} \rho_{k,j} (\abs{z}) \abs{z}^{2k} \dif \mu(z) = \ind_{i=j}$ for all $ i,j\in\set{1, \dotsc, \ell_k}$---this set being empty if $\ell_k=0$.
These conditions imply that the operator $\mathrm{K}$ with (integral) kernel $K$ is a projection.
Let $\mathcal{Y} = \set{ (k,j) \given j\in\set{1,\dotsc, \ell_k}, k\in\bZ }$ and observe that for any $r>0$, the operator $\mathrm{K}|_{D_r}$ with kernel $K$ acting on $L^2(D_r,\mu)$ has eigenvalues
\begin{equation}\label{eq:Eig}
	\lambda_{u}(r)
	=\int_{\abs{z} < r} \rho_{k,j}^2\bigl(\abs{z}\bigr) \abs{z}^{2k} \dif \mu(z) \in \intcc{0,1} , \qquad u=(k,j)\in\mathcal{Y}.
\end{equation}
We also assume that $\operatorname{Tr}\mathrm{K}|_{D_r} = \sum_{u\in \mathcal{Y}} \lambda_{u}(r)<\infty$ for all $r>0$ so that the operator $\mathrm{K}$ is locally trace-class.
Then, we obtain the following lemma which is modelled after \cref{thm:Kostlan} in \cite{K92SpectraGaussianMatrices}, see also \cite[Section 4.1]{PV05ZerosGaussianPowerSeries}.
Note that it immediately implies the identity \eqref{eq:TS}.

\begin{lemma} \label{lem:ModulusPoints}
	Denote by $\set{Z_u}_{u\in \mathcal{Y}}$ the atoms of $\xi$ and let $(X_u)_{u\in \mathcal{Y}}$ be an array of independent random variables with laws $(\lambda_u)_{u\in \mathcal{Y}}$ as in \eqref{eq:Eig}.
	The two point processes $\set{\abs{Z_u}}_{u\in \mathcal{Y}}$ and $\set{X_u}_{u\in \mathcal{Y}}$ on $\bR_+$ have the same distribution.
\end{lemma}
Since it is more general than previous results, the proof of \cref{lem:ModulusPoints} is given in \cref{sec:ModulusPoints} for completeness.
We use this result to study the law of the counting statistic $( \xi(D_r) )_{r\in\bR_+}$ as a stochastic process.
For simplicity, we assume that $\ell_k = \ind_{k \ge-\alpha}$ with $\alpha\in \bN_0$ and denote by $T(r) = \Ex{\xi(D_r)}$. Then, we define the random variables for $k\ge -\alpha$,
\begin{equation} \label{def:Gamma}
	\Gamma_{k+\alpha+1} = T(X_k),
\end{equation}
where the sequence $(X_k)_{k\ge -\alpha}$ is as in \cref{lem:ModulusPoints}.
This transformation corresponds to \emph{unfolding}\footnote{By definition, note that $\Ex{ \sum_{k\in\bN} \ind_{\Gamma_k\le R}} =R$ for all $R>0$, so that the densities of the $\Gamma_k$ form a partition of unity. In particular $(\Gamma_k)_{k\in\bN}$ cannot be any sequence of absolutely continuous positive random variables.} the point process $\xi$ so that the random variable $\Gamma_k$ is \emph{located} around $k$.
If the random variables $\Gamma_k$ have a first moment, this means that $\Ex{\Gamma_k} = k + O(1)$ for large $k$. However, for the hyperbolic ensembles considered in \cref{sec:hyperbolic}, the \emph{radii} $\Gamma_k$ do not have a well-defined mean.
As for the applications to the invariant processes discussed in \cref{sec:Applications}, this naturally takes into account the geometry at hand.
In the following, we study the process
\begin{equation} \label{def:Xi}
	\Xi_R = \sum_{k\in\bN} \ind_{\Gamma_k\le R} - R \overset{\mathrm{d}}{=} \xi(D_r) - \Ex{\xi(D_r)},
	\qquad \text{with } r=T^{-1}(R).
\end{equation}

Our goal is to show that the random variable $\Xi_R $ converges as $R\to\infty$ in the mod-phi sense under the following conditions on the sequence $(\Gamma_k)_{k\in\bN}$.

\begin{assumption} \label{NA}
Assume there exists $\vartheta\in\set{0,1}$ and an increasing continuous function $\Sigma\colon \bR_+ \to \bR_+$ such that $\Sigma_R \to \infty$ as $R\to\infty$ so that it holds for any $k\in\bN$,
\begin{equation} \label{cond}
	\Pr{\Gamma_k \le R} = \Phi\Bigl(\frac{k-\vartheta R}{\Sigma_R}\Bigr) + \frac{1}{\Sigma_R} \Psi\Bigl(\frac{k-\vartheta R}{\Sigma_R}\Bigr) + O\bigl( \theta_{R,k}\bigr),
\end{equation}
	where the errors satisfy $\sum_{k\in\bN} \theta_{R,k} \to 0 $ as $R\to\infty$, $\Phi(x) = \Pr{Z>x}$ for an absolutely continuous real-valued random variable $Z$ with $\Ex{\abs{Z}}<\infty$ and
	$\Psi \in W^{1,1}(\bR)$. If $\vartheta=1$, the function $\Sigma_R = o(R)$ and we further assume that there exists $m>1$ such that both $\Ex{\abs{Z}^m}<\infty$ and $\Sigma_R^{m+1} = o(R^m)$.
\end{assumption}

These assumptions mean that $(k,R)\mapsto \Pr{\Gamma_k \le R}$ has a \emph{smooth profile} up to a small error $ \theta_{R,k}$ and if $k\ge R + \Sigma_R^{1+\epsilon}$, \eqref{cond} is to be interpreted as a tail-bound for the random variables $\Gamma_k$: $\Pr{\Gamma_k \le R} =O\bigl( \theta_{R,k}\bigr)$.
In particular, this implies that only the random variables $\Gamma_k$ for $k$ inside a window centred around $R$ contribute significantly to \eqref{def:Xi}. These assumptions have been tuned to control the asymptotics of the cumulant general function of $\Xi_R$ in order to apply the theory of mod-phi convergence.
In our applications, the parameter $\vartheta$ distinguishes between the \emph{Euclidean} $(\vartheta=1)$ and \emph{hyperbolic} $(\vartheta=0)$ models for which the \emph{radii} $\Gamma_k$ behave differently, reflecting the underlying geometry of the process.
In particular, in case $\vartheta =1$, the asymptotics \eqref{cond} show that the random variables $\frac{\Gamma_k-k}{\Sigma_k}$ is statistically approximated by $Z$ for large $k$ and \cref{lem:A} below provide technical conditions which allow to verify that the \cref{NA} hold for several instances of determinantal processes.
In case $\vartheta=0$, the limiting random variable $Z>0$ so that $\Phi(x)=1$ and $\Psi(x) =0$ for all $x\le 0$.

\begin{remark} \label{rk:BV}
In \cref{NA}, instead of the condition $\Psi \in W^{1,1}(\bR)$, if we assume that we can decompose
$\Psi = \Psi_{AC} + \Psi_{M} $ where $\Psi_{AC} \in W^{1,1}(\bR)$ and there exists an open interval $J\subseteq \bR$ such that $\Psi_M\in L^1(J)$ is monotone and $\Psi_M = 0$ on $\bR\setminus J$, then the results presented below remain true.
This more general hypothesis will be useful when we investigate the properties of the hyperbolic ensembles in \cref{sec:hyperbolic}.
\end{remark}

One can also fit in our framework, the finite-size version of the point process $\xi$ obtained by truncating the correlation kernel, although our assumptions become slightly more technical. For any $N\in\bN_0$, we let $\xi^{(N)}$ be the determinantal process with correlation kernel $K^{(N)}(z,w)=\sum_{k<N-\alpha} \rho_{k,1}(\abs{z}) z^k \rho_{k,1}(\abs{w})\bar{w}^k$ for $z,w\in\bC$, and we define the processes\footnote{By the general theory, the point process $\xi^{(N)}$ has exactly $N$ points. Moreover, according to \eqref{def:Xi}, we have $\Xi_R =\Xi^{(\infty)}_R$.}
\begin{equation} \label{def:XiN}
	\Xi^{(N)}_R = \sum_{k=1}^{N} \bigl(\ind_{\Gamma_k\le R} - \Pr{\Gamma_k\le R} \bigr) \overset{\mathrm{d}}{=} \xi^{(N)}(D_r) - \Ex[\big]{\xi^{(N)}(D_r)},
	\qquad \text{with } r=T^{-1}(R).
\end{equation}

Note that our normalization is such that $ \sum_{k\le N} \Pr{\Gamma_k\le R} <R $ for all $R>0$, this bound being sharp for $R$ in the \emph{bulk} (i.e.~for $R \ll N$).
In the sequel, for $N<\infty$, we choose a sequence $R=R(N) \to\infty$ so that one of the following conditions holds as $N\to\infty$,
\begin{equation} \label{Ncond}
i)\ \Sigma_R\Phi(\tfrac{N-\vartheta R}{\Sigma_R}) \to 0 , \qquad
ii)\ \frac{N-\vartheta R}{\Sigma_R} = a^+ +o(\Sigma_R^{-1}) \text{ with } a^+ \in\bR, \qquad
iii)\ \Sigma_R \bigl(1-\Phi(\tfrac{N-\vartheta R}{\Sigma_R}) \bigr) \to 0.
\end{equation}

Note that in case $i)$, $\tfrac{N-\vartheta R}{\Sigma_R} \to a^+=+\infty$ and this corresponds to the \emph{bulk} in the sense that the statistical behaviour of the random variables $\Xi^{(N)}_R$ and $\Xi_R$ are the same for large $N$.
In case $iii)$, we must have $\tfrac{N-\vartheta R}{\Sigma_R} \to a^+=-\infty$ and it corresponds to the regime where the fluctuations of $\Xi^{(N)}_R$ are asymptotically negligible, so we are not going to emphasize on this case. Finally, $ii)$ is the \emph{edge} regime where interesting finite-size effects occur\footnote{In case $ii)$, for a given $a^+\in\bR$, the condition \eqref{Ncond} is interpreted as an implicit equation for $R(N)$. Then, because $\Sigma$ is increasing, the radius $R(N)$ has a well-defined asymptotics up to order 1 as $N\to\infty$.}.
In the remainder of this section, with $a^+$ as above if $N<\infty$ and $a^+ =+\infty$ if $N=\infty$, we denote
\begin{equation}\label{interval}
	I = \intoo{a^-,a^+}
	\qquad\text{with}\qquad
	a^- =
	\begin{cases}
		0 & \text{if $\vartheta=0$,}\\
		-\infty & \text{if $\vartheta=1$.}
	\end{cases}
\end{equation}

In \cref{sec:ModPhiGeneralProof}, we prove the following general convergence result.
\begin{theorem}\label{thm:ModPhiGeneral}
	Under the \cref{NA} and \eqref{Ncond}, the random variable $\Xi_R^{(N)}$ converges as $R\to\infty$ in the mod-phi sense of \cref{def:ModPhi} with speed $\Sigma_R$ with respect to a cumulant generating function
	\begin{equation} \label{MGF}
		 \Lambda (z) = \int_I \kappa_{\Phi(x)}(z) \dif x
	\end{equation}
	and with limiting function
	\begin{equation} \label{eq:PsiF}
		\psi (z) = \exp \biggl( \int_I \Psi(x)\dot\kappa_{\Phi(x)}(z) \dif x + \frac{\kappa_{\Phi(\tau)}(z)}{2} \biggr).
	\end{equation}
	The functions $\Lambda$ and $\psi$ are analytic in the strip $\bS = \set{z\in\bC \given \abs{\Im{z}} < \pi }$ and $I$ is given by \eqref{interval}.
\end{theorem}

As explained in \cref{sec:Background}, this allows us to deduce precise asymptotic results like Berry-Esseen estimates as well as precise moderate and large deviations for the (normalized) counting statistics $\Xi_R/\sqrt{\Sigma_R \Lambda''(0)}$ of our determinantal process (see \cref{thm:ModPhiConclusions} and \cref{cor:mod}).
Note that under \cref{NA}, by Fubini's theorem, we have
\begin{equation}\label{eq:BoundExpectationZ}
	\int_\bR \Phi(x)\bigl(1-\Phi(x)\bigr) \dif x
	= \int_\bR \biggl( \int_x^\infty \Phi'(v) \dif v \int_{-\infty}^x \Phi'(u) \dif u \biggr) \dif x
	= \int_{u\le v} \Phi'(u) \Phi'(v) (v-u) \dif u \dif v
	\le 2\Ex{\abs{Z}}
\end{equation}
and $\Ex{\abs{Z}} < \infty$.
This shows that the function $\Lambda$ satisfies
\begin{equation} \label{var}
\Lambda'' (0) = \int_I \Var{\ind_{Z\le x}}\dif x = \int_I \Phi (x)\bigl(1-\Phi (x)\bigr) \dif x <\infty.
\end{equation}
A similar computation yields that $\Lambda$ is convex on $\bR$
and that as $R\to\infty$
\begin{equation}\label{eq:VarXi}
	\Var[\big]{\Xi_R^{(N)}}= \Sigma_R \Lambda''(0) + O(1).
\end{equation}
Let us also emphasize that $a^+=\infty$ if $N=\infty$, while $\Lambda''(0)$ depends on our choice of sequence $R(N)$ through \eqref{Ncond} and \eqref{interval} when $N<\infty$.

Our setup also allows us to study $( \xi(D_{r}) )_{r\in\bR_+}$ as a stochastic process which is where we need to distinguish between the case $\vartheta =0,1$.
For instance, in the \emph{Euclidean setting}, we observe non-trivial correlations only in a \emph{microscopic regime}.

\begin{theorem}\label{thm:FCLTPlanar}
	Assume \cref{NA} with $\vartheta=1$, that $\sqrt{\frac{\Sigma_{tR}}{\Sigma_R}} \to g_t$ pointwise as $R\to\infty$ where $g\in C(\bR_+)$ and that if $N<\infty$, $R(N)$ satisfies the conditions \eqref{Ncond} and \eqref{interval}. Then, it holds as $R\to\infty$, 
	\begin{equation*}
		\biggl( \frac{\Xi_{tR}^{(N)}}{ \sqrt{\Sigma_R}} \biggr)_{t\in\bR_+}
		\xrightarrow[]{\text{fdd}} \bigl(\sigma_t W_t\bigr)_{t\in \bR_+},
	\end{equation*}
	where $(W_t)_{t\in \bR_+}$ is a Gaussian white noise and $\sigma_t^2 = g_t^2 \int_{-\infty}^{a^+(t)} \Phi (x)(1-\Phi (x)) \dif x$ with $a^+(t)= \infty$ for $N=\infty$ and $a^+(t) = \lim_{N\to\infty} \frac{N-tR}{\Sigma_{tR}}$ for $N<\infty$ for all $t>0$.
	Moreover, it also holds as $R\to \infty$, 
	\begin{equation*}
		\biggl( \frac{\Xi_{R+t\Sigma_R}^{(N)}}{ \sqrt{\Sigma_R}} \biggr)_{t\in\bR}\xrightarrow[]{\cS(\bR)} (G_t)_{t\in \bR},
	\end{equation*}
	where $(G_t)_{t\in \bR}$ is a centred Gaussian process with kernel
	\begin{equation*}
		\Cov{G_s}{G_t} = \int_{I} \Phi(x-s)\bigl(1-\Phi(x-t)\bigr) \dif x , \qquad \text{for } s\le t.
	\end{equation*}
\end{theorem}
Let us emphasize that the Gaussian process $(G_t)_{t\in \bR}$ is stationary if and only if $I=\bR$, which corresponds to looking at microscopic fluctuations in the bulk.

On the other hand, in the \emph{hyperbolic setting}, as the area of the disk $D_r$ is comparable to the size of its boundary, we obtain non-trivial \emph{macroscopic} fluctuations as the following functional limit theorem shows.
\begin{theorem}\label{thm:FCLTHyperbolic}
	Assume \cref{NA} with $\vartheta=0$ and that $R(N)$ satisfies the conditions \eqref{Ncond} and \eqref{interval} if $N<\infty$. Then, it holds as $R\to\infty$, 
		\begin{equation*}
		\biggl(\frac{\Xi_{tR}^{(N)}}{\sqrt{\Sigma_R}}\biggr)_{t\in\bR_+}
		 \xrightarrow{\cS(\bR_+)} (G_t)_{t\in \bR_+},
	\end{equation*}
	where $(G_t)_{t\in \bR_+}$ is a centred Gaussian process with kernel
	\begin{equation*}
		\Cov{G_s}{G_t} = \int_{I} \Phi\Bigl(\frac{x}{s}\Bigr)\Bigl(1-\Phi\Bigl(\frac{x}{t}\Bigr)\Bigr) \dif x , \qquad \text{for } 0<s\le t.
	\end{equation*}
\end{theorem}

In this case, we see that the Gaussian process $(G_{e^{-t}})_{t\in \bR}$ is stationary if and only if $I=\bR_+$.


\section{Applications}\label{sec:Applications}

In this section we present different determinantal processes which fall in the framework of \cref{sec:MainResults}.
In \cref{sec:FiniteGinibre}, we discuss the finite Ginibre ensembles and present the \emph{edge effect} which occurs because of the \emph{finite} number of particles.
Using our method, we can also consider finite versions of the Ginibre-type ensembles and hyperbolic ensembles discussed below and observe similar boundary effects.
In \cref{sec:Ginibre}, we introduce the Ginibre-type ensembles which are generalizations of the (infinite) Ginibre ensemble coming from the physical description of electrons. Our goal is to compare the deviations for counting statistics, as well as the entanglement entropy, for these different ensembles. These sections expand the results presented in \cref{sec:Introduction}.
In \cref{sec:hyperbolic}, we apply our results to the zeros of the hyperbolic Gaussian analytic function and other hyperbolic ensembles.
Our goal is to put in perspective what happens in different geometric settings by comparing
\emph{hyperbolic} and \emph{Euclidean} models. Finally, our findings are summarized in \cref{sec:sum}.

\subsection{The Ginibre ensemble}\label{sec:FiniteGinibre}

The Ginibre ensemble introduced in \cite{G65StatisticalEnsembles} is the prototypical example of a non-Hermitian random matrix. It consists of an $N\times N$ matrix filled with i.i.d.\ standard complex Gaussian entries and its eigenvalue process, denoted by $\xi ^{(N)}$, is a determinantal process with correlation kernel
\begin{equation} \label{def:KN}
	K_N(z,w) =\sum_{k=0}^{N-1} \frac{(z\bar{w})^k}{k!}
\end{equation}
with respect to the Gaussian measure $\frac{\dif \mu }{\dif m}(z) = \frac{1}{\pi }e^{-\abs{z}^2}$ on $\bC$.
It is well-known that the point process $\xi ^{(N)}$ is invariant by rotation and that it distributes uniformly on the disk $D_{\sqrt{N}}$ for large $N$, this is known as the circular law, see e.g.\ \cite{C15BoltzmannRandomMatrices}.
It is plain that $K_N(z,w) \to e^{z\bar{w}}$ as $N\to\infty$, so that the (infinite) Ginibre ensemble as presented in \cref{sec:Introduction} is the local scaling limit of the Ginibre eigenvalue process.
In fact, \eqref{def:KN} is a finite rank approximation of the Ginibre kernel $K_0$ and as such, by applying \cref{lem:ModulusPoints}, we find that according to \eqref{def:Gamma} with $T(r)=r^2$, $\Gamma_k$ are gamma-distributed random variables with shape $k$ and rate $1$ for $k\in \set{1,\dotsc, N}$ and ``non-existent" otherwise.
In this section, we study the counting statistics $(\xi ^{(N)}(D_{\sqrt{\gamma N}}))_{\gamma\in\intcc{0,1}}$ as $N\to\infty$ and observe that there is an \emph{edge effect} for $\gamma =1$.
We let for $\gamma>0$,
\begin{equation*} 
	\Xi_N(\gamma)
	= \xi ^{(N)}\bigl(D_{\sqrt{\gamma N}}\bigr) - \Ex[\big]{ \xi ^{(N)}\bigl(D_{\sqrt{\gamma N}}\bigr) } \overset{\mathrm{d}}{=} \sum_{k=1}^{N} \bigl( \ind_{\Gamma_{k}\le \gamma N} - \Pr{\Gamma_{k}\le \gamma N} \bigr).
\end{equation*}

We have $ \Ex{ \xi ^{(N)}(D_{\sqrt{\gamma N}}) } = \gamma N +o_C(1)$ uniformly for all $\gamma < 1- C\sqrt{\frac{\log N}{N}} $ if $C>0$ is sufficiently large. However, at the edge, after rescaling $\gamma =1+t/\sqrt{N}$, it holds uniformly for all $|t| \le C\sqrt{\frac{\log N}{N}} $,
$ \Ex{ \xi ^{(N)}(D_{\sqrt{\gamma N}}) } = N\ind_{t<0} + \sqrt{N}f(t) +g(t) + o_C(1)$ for two smooth functions $f$ and $g$. Hence, we cannot use the crude approximation
$\Ex{ \xi ^{(N)}(D_{\sqrt{\gamma N}}) } = \min\set{\gamma,1}N +O(\sqrt{N})$ to study the fluctuations of the process $(\Xi_N(\gamma) )_{\gamma>0}$.

\begin{proposition}
\label{thm:FiniteGinibreModPhi}
	In the bulk, let $I_\infty=\bR$ for $\gamma \in \intoo{0,1}$.
	At the edge, let $I_t = \intoc{-\infty, t}$ for $\gamma = 1- \frac{t}{\sqrt{N}}+\frac{t^2}{2N}$ and $t\in \bR $.
	In both cases, the sequence $(\Xi_N(\gamma))_{N\in\bN}$ converges in the mod-phi sense at speed $\sqrt{\gamma N}$ with cumulant generating function of the form \eqref{MGF} and with limiting function of the form \eqref{eq:PsiF}, where $I=I_t$, $\Phi =\Phi_0$ being the tail distribution function of $\Normaldist_{0,1}$ and $\Psi = \Psi_0 = \frac{2-5x^2}{6\sqrt{2\pi}}e^{-\frac{x^2}{2}}$.
\end{proposition}
The proof of \cref{thm:FiniteGinibreModPhi} is given in \cref{sec:GinibreProof} and it relies on the fact that $\Gamma_k$ are infinitely divisible together with an Edgeworth expansion to obtain the approximation \eqref{cond}.
One could also use the steepest descent method to obtain these asymptotics and it naturally explains why the error function $\Phi_0$ and the correction $\Psi_0$ arise.
By \cref{cor:mod}, this mod-phi convergence implies Berry-Esseen estimates and precise moderate to large deviations for the normalized counting statistics $(\gamma N)^{-\frac14}\Xi_N(\gamma)$ including at the edge of the circular law.

We also obtain functional central limit theorems as in \cref{thm:FCLTPlanar}.

\begin{proposition}\label{thm:FiniteGinibreFCLTFDD}
	It holds as $N\to\infty$ that
	\begin{equation*}
		\bigl(N^{-\frac14}\Xi_N(t) \bigr)_{t\in \bR_+} \xrightarrow{\text{fdd}} \bigl(\sigma_t N_t\bigr)_{t\in\bR_+},
	\end{equation*}
	where $(N_t)_{t\in \bR_+}$ is a Gaussian white noise and $\sigma_t^2=\sqrt{\frac{t}{\pi}}$ if $t<1$, $\sigma_1^2=\frac{1}{2\sqrt{\pi}}$ and $\sigma_t^2=0$ if $t>1$.
\end{proposition}
\begin{proposition}\label{thm:FiniteGinibreFCLTSkorokhod}
	For any $\gamma\in\intoc{0,1}$, it holds as $N\to\infty$ that
	\begin{equation*}
		\bigl( N^{-\frac{1}{4}} \Xi_N\bigl(\gamma+ \tfrac{t}{\sqrt{N}}\bigr)\bigr)_{t\in \bR} \xrightarrow{\cS(\bR )} \bigl( G_t\bigr)_{t\in \bR },
	\end{equation*}
	where $(G_t)_{t\in \bR }$ is a centred Gaussian process with covariance kernel
	\begin{equation*}
		\Cov{G_s}{G_t}
		= \gamma^\frac{1}{4} \int_{I} \Pr{\Normaldist_{0,1} > x-s} \Pr{\Normaldist_{0,1} \le x-t} \dif x
	\end{equation*}
	for $s\le t$, where $I=\bR$ if $\gamma\in (0,1)$ (i.e.\ in the bulk) and $I=\intoc{-\infty,0}$ if $\gamma=1$ (i.e.\ at the edge).
\end{proposition}
Note that as expected, these results are consistent with those presented in \cref{sec:Introduction} for the (infinite) Ginibre ensemble, except for the \emph{edge effects}. In particular, by \cref{rk:GP}, the Gaussian process $(G_t)_{t\in \bR }$ observed in the bulk is the same as in \cref{thm:GinibreFCLTSkorokhod}.

\subsection{Ginibre-type ensembles} \label{sec:Ginibre}

These point processes, also known as polyanalytic Ginibre ensembles, are generalizations of the (infinite) Ginibre ensemble which arises from the physics of two-dimensional (spinless) electrons subject to a constant (perpendicular) magnetic field.
This problem is for instance relevant to the description of the quantum Hall effect, see \cite{E13QuantumHallEffects,R19LaughlinFunction} and reference therein.
As explained in \cite{HH13PolyanalyticGinibre, S15GinibreType, APRT17WeylHeisenbergEnsemble}, this system is described by the \emph{Landau Hamiltonian} whose spectrum consists of eigenvalues with infinite multiplicity;
the corresponding eigenspaces $\cH_\LL$ for $\LL\in\bN_0$ are called \emph{Landau levels} and they can be identified with $\LL$-analytic Bargmann-Fock spaces\footnote{
That is, we have the spectral decomposition $L^2(\mu) \cong \bigoplus_{\LL=0}^{\infty} \cH_\LL$ with $\frac{\dif \mu }{\dif m}(z) = \frac{1}{\pi }e^{-\abs{z}^2}$ on $\bC$, $\cH_0 = \operatorname{span} \set{ z^j \given j \in \bN_0 }$ in $L^2(\mu)$ and
 $ \cH_\LL = \set{ e^{\abs{z}^2} \partial_z^{\LL}(\phi e^{-\abs{z}^2}) \given \phi\in \cH_0} $ for $\LL\in\bN$. Hence, it holds $\overline{\partial_z}^{\LL+1}(\phi) = 0$ for any $\phi\in \cH_\LL$ which motivates the name polyanalytic Ginibre ensemble introduced in \cite{HH13PolyanalyticGinibre}.}.
In particular, the ground state $\LL=0$ of the \emph{Landau Hamiltonian} corresponds to the Hilbert space $ \cH_0$ of all entire functions in $L^2(\mu)$ whose reproducing kernel is the Ginibre kernel $K_0(z,w)= e^{z\bar{w}}$ for $z,w\in\bC$.
In general, the Hilbert space $\cH_\LL$ has a reproducing kernel
\begin{equation} \label{Kalpha}
	K_\LL(z,w) = \sum_{k=-\LL}^{\infty} \phi_{\LL,k}(z) \overline{\phi_{\LL,k}(w)}
	= L_\LL\bigl(\abs{z-w}^2\bigr)e^{z\bar{w}} , \qquad z,w \in\bC,
\end{equation}
where $L_\LL$ denotes the orthonormal Laguerre polynomial of degree $\LL\in\bN_0$.
The Ginibre-type ensembles, denoted by $\xi_\LL$, refers to the determinantal point process on $\bC$ with correlation kernel $K_\LL$. They describe the thermodynamic limit of an infinite system of electrons which are confined in the $\LL$-th Landau level.
Since $L_\LL(0)=1$ for all $\LL\in\bN_0$, these ensembles are homogeneous with intensity $1/\pi$.
Moreover, it is easy to check from the second expression in \eqref{Kalpha} that they are also translation and rotation invariance on $\bC$.

As for the first expression in \eqref{Kalpha}, the modes are given by $\phi_{\LL,k}(z) = L_{\LL}^{(k)}(\abs{z}^2) z^k$,
where $L_{\LL}^{(k)}$ are the orthonormal generalized Laguerre polynomials\footnote{
It is straightforward to check that $\partial_z^{\LL}(z^j e^{-\abs{z}^2}) = \widetilde{L}_{\LL}^{(j-\LL)}(\abs{z}^2) z^{j-\LL} e^{-\abs{z}^2} $ for any $j, \LL \in\bN_0$ where $ \widetilde{L}_{\LL}^{(\beta)}(x) = x^{-\beta} e^{x} \frac{\dif^\LL}{\dif x^\LL} (x^{\beta+\LL} e^{-x})$ is a polynomial of degree $\LL$ for any $\beta\in\bR$. Hence, if we normalize $L_{\LL}^{(\beta)} = c_{\alpha,\beta} \widetilde{L}_{\LL}^{(\beta)}$ with some constant $c_{\alpha,\beta}>0$ appropriately, it holds $\int_{\bR_+} L_{\LL}^{(\beta)}(x)^2 x^\beta e^{-x} dx =1$ for $\beta \ge -\LL$. This shows that for any $\LL\in\bN_0$, $(\phi_{\LL,k} )_{k=-\LL}^{\infty}$ is an orthonormal basis of the Hilbert space $\cH_\LL$.
It turns out that for $\beta=0$, $L_{\LL}^{(0)} = L_{\LL}$ for all $\LL\in\bN_0$ are the classical Laguerre polynomials.}
of degree $\LL\in \bN_0$. This correspondence is not entirely obvious and can be found e.g.~in \cite[Section 2]{HH13PolyanalyticGinibre} or \cite[Section 2]{S15GinibreType}.

Following the convention from \cref{sec:MainResults}, we let $R = T(r) = r^2$ and observe that the square-modulus of the points of $\xi_\alpha$ have law given by
\begin{equation} \label{lawGamma}
	\Pr[\big]{\Gamma_{k}^{(\LL)} \le R} = \int_0^R L_\LL^{(k-\LL-1)}(x)^2 x^{k-\LL-1} e^{-x} \dif x , \qquad k\in\bN.
\end{equation}
For the Ginibre ensemble ($\LL=0$) we recover that $\Gamma_{k}^{(0)}$ are gamma-distributed random variables as in \cref{thm:Kostlan}.
In \cref{lem:GinibreTypeBE}, we show that these random variables have the following asymptotic property: If we let $\widetilde\Gamma_{k}^{(\LL)} = \frac{\Gamma_{k-\LL}^{(\LL)}-k}{\sqrt{k}} $ so that $\Ex{\widetilde\Gamma_{k}^{(\LL)}} = 0$, then it holds for any $\beta>0$, as $k\to\infty$,
\begin{equation} \label{GFexpansion}
	\Ex[\Big]{e^{i x\widetilde\Gamma_{k}^{(\LL)}}}
	= \Ex[\Big]{e^{i x Z_\LL}} \Bigl(1-\frac{\i x^3}{3\sqrt{k}}\Bigr) + O_\beta\Bigl(\frac{1}{k(1+x^2)^\beta}\Bigr).
\end{equation}
For $\LL\in\bN_0$, let $h_\LL(x) = H_\LL(x) e^{-\frac{x^2}{4}}/(2\pi)^{-\frac{1}{4}}$ be the Hermite functions, that is $H_\LL$ are Hermite polynomials of degree $\LL$ which are orthonormal with respect to the weight $e^{-\frac{x^2}{2}}/\sqrt{2\pi}$ on $\bR$.
Then, the random variables $Z_{\LL}$ which appear on the RHS of \eqref{GFexpansion} have probability density function $h_\LL^2$:
\begin{equation} \label{phialpha}
\Phi_\LL(x) = \Pr{ Z_\LL >x} = \int_x^{\infty} h_\LL^2(t) \dif t.
\end{equation}
Note that $\Phi_0(x) = \Pr{\Normaldist_{0,1} > x}$ corresponds to the standard \emph{error function}. For $\alpha\in\bN_0$, $Z_\alpha$ can be interpreted as the position of a quantum particle confined by a harmonic trap in the $\LL$-th excited state.
This connection between the point processes associated with different Landau levels and the harmonic oscillator is rather surprising to us.
Note also that for $\LL\ge 1$, unlike gamma random variables, $\Gamma_{k}^{(\LL)}$ are not infinitely divisible, since otherwise $\widetilde\Gamma_{k}^{(\LL)}$ would become asymptotically Gaussian for large $k$, see \eqref{GFexpansion}.

We define for $\alpha \in \bN_0$ and $R\in\bR_+$,
\begin{equation} \label{Xialpha}
	\Xi_R^{(\LL)}
	= \sum_{k\in \bN} \ind_{\Gamma_{k}^{(\LL)}\le R} - R
	\overset{\mathrm{d}}{=} \xi_\LL(D_r) - \Ex[\big]{ \xi_\LL(D_r)}
	\qquad \text{with } r=\sqrt{R}.
\end{equation}
From \cref{thm:ModPhiGeneral} and the asymptotics \eqref{GFexpansion}, we can infer that the random variables $\Xi_R^{(\LL)}$ converge as $R\to\infty$ in the mod-phi sense with cumulant generating functions of the form \eqref{MGF} which are associated to the tail functions $\Phi_\LL$ given by \eqref{phialpha}. We prove the following result in \cref{sec:GinibreTypeProof}.

\begin{proposition}
\label{thm:GinibreTypeModPhi}
For any $\LL \in \bN_0$, the random variable $\Xi_R^{(\LL)}$ converges in the mod-phi sense at speed $\sqrt{R}$ with the cumulant generating function
	\begin{equation*}
		\Lambda_\LL (z) = \int_\bR \log\Bigl(1+\Phi_\LL(x)(e^z-1)\Bigr)-\Phi_\LL(x)z \dif x ,
	\end{equation*}
	and limiting function
	$\displaystyle \psi_\alpha (z) = \exp\biggl(\int_\bR \Bigl(\frac{e^z-1}{1+\Phi_\LL(x)(e^z-1)} - z\Bigr) \Psi_\LL(x) \dif x \biggr)$
	where $\Psi_{\LL} = \frac{x^2}{2}\Phi_\LL' + \frac{1}{3}\Phi_\LL'''$.
\end{proposition}
By \cref{cor:mod}, this implies an (extended) central limit theorem together with a Berry-Esseen bound and precise moderate deviations for the counting statistics $\xi_\LL(D_{r})$ (after appropriate normalization).

\begin{remark} \label{rk:CE19}
	In connection with our results, counting statistics for fermions in the lowest Landau level $\LL=0$ of a magnetic Laplacian have been recently considered in \cite{CE19EntanglementEntropy} in a more general geometric setting.
	These ensembles correspond to determinantal processes on a K\"ahler manifold $M$ and the law of the statistic for the number of points $N_A$ in a subset $A \subseteq M$ can be described in terms of the eigenvalues of a Toeplitz operator $T_A$.
	By semi-classical methods, if $A$ has a smooth boundary, the authors obtain two-term Weyl asymptotics for the operator $T_A$ and discuss applications to computing the \emph{entanglement entropy}. They show that the corrections are universal and described in terms of the \emph{error function} $\Phi_0$ from \eqref{phialpha}.
	This also allows us to obtain asymptotic expansions of the cumulants $\varkappa^{(q)}(N_A)$ for a general set $A$, see \cite[Theorem 1.6]{CE19EntanglementEntropy}.
	In particular, there results also apply to the Ginibre ensemble in which case, with our conventions, one has for $A\subseteq \bD$ and any $q\in\bN$, as $N\to\infty$,
	\begin{equation} \label{CE19}
	\begin{aligned}
		\Ex[\big]{\xi ^{(N)} (\sqrt{N}A)} &= N\frac{m(A)}{\pi } + O(1),\\
		\varkappa^{(2q)}\bigl(\xi ^{(N)}(\sqrt{N}A)\bigr) &= \sqrt{N} \frac{\Vol(\partial A)}{2\pi } \int_{-\infty }^\infty \varkappa^{(2q)}\bigl(\ind_{U \le \Phi_0(t)}\bigr) \dif t + O\bigl(N^{-\frac{1}{2}}\bigr),\\
		\varkappa^{(2q+1)}\bigl(\xi^{(N)} (\sqrt{N}A)\bigr) &= O(1),
	\end{aligned}
	\end{equation}
	where $U$ is a random variable uniform in $\intcc{0,1}$.
	If $A$ is a disk, these results are consistent with the mod-phi convergence from \cref{thm:FiniteGinibreModPhi}\footnote{There is a typo in the asymptotics of \cite[Theorem 1.6]{CE19EntanglementEntropy} for odd cumulants. We have shown instead that $\varkappa ^{(2q+1)}(\xi (\sqrt{N}A))$ is of order 1 as $N\to\infty$ for any $q\in\bN$.}. However, since the eigenvalues of the operator are explicit in the case of a disk, our results are more precise. We also obtain for any $\gamma\in(0,1)$,
	\begin{equation*}
		\varkappa^{(2q+1)}\bigl(\xi (D_{\gamma N})\bigr) = -\frac{2}{3} \int_{-\infty }^\infty t \varkappa^{(2q+1)}\bigl(\ind_{U \le \Phi_0 (t)}\bigr) \dif t + O\bigl(N^{-\frac{1}{2}}\bigr).
	\end{equation*}
	In fact, our method gives in principle all order asymptotic expansions of the cumulants in terms of $\Phi_0$, see \cref{sec:GinibreProof}.
	Let us also point out that using the asymptotics \eqref{CE19} to deduce mod-phi convergence for counting statistics in a general set $A\subseteq \bD$ seems out of reach as the estimates from \cite{CE19EntanglementEntropy} lack control on the growth of the cumulants.
\end{remark}

Entanglement is a crucial property of quantum system which for instance plays a key role in quantum communication and information theory, \cite{HHHH09QuantumEntanglement}.
The question of quantifying (for large $N$) the \emph{entanglement entropy} for basic quantum system is an ongoing problem in (quantum) statistical physics, see e.g.\ \cite{LSS14RenyiEntenglementEntropies} and reference therein.
This has been achieved recently for different ensembles of free fermions using semi-classical methods based on the work of Widom, see e.g.\ the contribution of Gioev \cite{G06SzegoLimitTheorem}, Sobolev and co-authors \cite{HLS11FermionicEntanglement,S17QuasiClasicalAsymptotics} or \cite{CE19EntanglementEntropy}, as well as, for further references, \cite{E13UniversalityTrappedFermions, LMS18Entanglement} in the physics literature.
In \cref{thm:arealaw}, we are interested in investigating the precise asymptotics for the so-called (bipartite) entanglement entropies, denoted $S_\LL^\beta(D_r)$ of a disk $D_r$, for (integer) Laughlin states.
By exploiting the connection with the Ginibre $\alpha$-type process, we give a probabilistic proof of the precise asymptotics of the entanglement entropy when the radius $r$ is large. Our results are consistent with the well-known \emph{area law} for fermionic ensembles and to be compared with \cite[Theorem 1.8]{CE19EntanglementEntropy} in the context of \cref{rk:CE19}. Moreover, this confirms that the entanglement entropies also depend non-trivially on the index $\LL$ of the Landau levels.

\begin{theorem} \label{thm:arealaw}
	For any $\LL\in\bN_0$ and $\beta>\frac{1}{2}$, it holds as $r\to\infty$,
	\begin{equation*}
		S_\LL^\beta(D_r)
		= \operatorname{Tr}\bigl[f_\beta(\mathrm{K}_\LL|_{D_{r}})\bigr]
		= r^\frac{1}{4} \int_\bR f_\beta\bigl(\Phi_\LL(x)\bigr) \dif x + o(1) ,
	\end{equation*}
	where $f_\beta(x) = \frac{\log(x^\beta+(1-x)^\beta)}{(1-\beta)}$ if $\beta\neq 1$ and $f_1(x) = \lim\limits_{\beta\to1} f_\beta(x) = -x \log(x) - (1-x)\log(1-x)$ for $x\in\intcc{0,1}$.
\end{theorem}

\cref{thm:GinibreTypeModPhi} describes both typical and moderate deviations for counting statistics $	\Xi_R^{(\LL)}$ of the Ginibre-type ensembles. These fluctuations, which corresponds to
the regime $\gamma \in \intcc{\frac{1}{2}, 1}$ in the JLM predictions \eqref{JLM93}, are governed by the harmonic oscillator tail probability functions $\Phi_\alpha$ depending on which Landau level is considered.
Let us now present our results for large deviations, i.e.~the case where $\gamma>1$ in \eqref{JLM93}.
In contrast with moderate deviations, this regime is governed by the exponential tail of the random variables $\Gamma_{k}^{(\LL)}$, see \eqref{lawGamma}.
This explains why large deviations are \emph{universal} for these ensemble.
To present our results, we rely on the following general conditions.

\begin{assumption}\label{ass:LD}
 Let $\Theta:\bR_+ \to\bR_+$ be a function such that $\frac{\Theta_R}{\sqrt{R \log R}} \to \infty$ as $R\to\infty$---this condition sets apart the large deviation regime.
We further assume that $\Theta$ is a regularly varying function, so that
$\gamma = 2 \lim\limits_{R\to\infty} \tfrac{\log \Theta_R}{\log R} $ exists\footnote{This normalization comes from \eqref{JLM93} and the fact that according to \eqref{Xialpha}, $r=\sqrt{R}$. Our convention is that $\gamma=\infty$ if $\Theta_R$ grows faster than any power of $R$.}. We define for $R>0$ and $t \in\bR_+$,
\begin{equation} \label{J}
	v_R =
	\begin{cases}
		\frac{\Theta_R^2}{R}&\text{if }\frac{\Theta_R}{R} \to 0,\\
		\Theta_R&\text{if }\frac{\Theta_R}{R} \to 1 , \\
		\Theta_R\log \Theta_R&\text{if }\frac{\Theta_R}{R} \to \infty
	\end{cases}
	\quad\text{and}\quad
	J_\gamma(t) =
	\begin{cases}
		\frac{t^2}{2} &\text{if }\frac{\Theta_R}{R} \to 0,\\
		(1+t)\log(1+t)-t&\text{if }\frac{\Theta_R}{R} \to 1,\\
		t (1-\frac{2}{\gamma})&\text{if } \frac{\Theta_R}{R} \to \infty.
	\end{cases}
\end{equation}
\end{assumption}

Like \eqref{JLM93}, \eqref{J} distinguishes between three different large deviation regimes.
The following Lemma shows that these regimes are distinguished by the tail behaviour of a gamma random variable as the shape $k\to\infty$.

\begin{lemma} \label{thm:gammaJ}
For any $\alpha\in\bN_0$, under the \cref{ass:LD}, it holds
\begin{equation*}
- \lim_{R\to\infty}\frac{1}{v_R} \log\Pr[\big]{\Gamma^{(\alpha)}_{R+ t\Theta_R} \le R} = J_\gamma(t).
\end{equation*}
\end{lemma}

We skip the proof of \cref{thm:gammaJ} as it follows easily from the bounds of \cref{lem:GinibreBounds} below when $\alpha=0$.
In order to extend the result to $\alpha\ge1$, one can compare the tails of different $\Gamma^{(\alpha)}_k$ random variables using the estimates from \cref{lem:GinibreTypeBounds} below.
That being said, let us observe that for $\Theta_R = x R^{\frac\gamma2}$ with $x>0$ and $\gamma>1$, the \emph{rate function} on the RHS of \eqref{JLM93} is given by
$\int_0^x J_\gamma(t) \dif t$.
Then, using the representation \eqref{Xialpha}, it is possible to obtain the following large deviation estimates for counting statistics of the Ginibre-type ensembles.
The proof is the same for all regimes and it is given in \cref{sec:JLMProof}.

\begin{theorem} \label{thm:ldp}
For any $\alpha\in\bN_0$ and any $x>0$, under the \cref{ass:LD}, it holds
\[
\lim_{R\to\infty}\frac{-1}{v_{r^2}\Theta_{r^2}} \log \Pr[\big]{\xi_\alpha (D_r)\ge r^2 + x\Theta_{r^2}}
=\lim_{R\to\infty}\frac{-1}{v_R\Theta_R}\log\Pr[\big]{\Xi_R^{(\LL)} \ge x\Theta_R}
= \int_0^x J_\gamma(t) \dif t.
\]
\end{theorem}

This proves \eqref{JLM93}, including the \emph{transition} which occurs when $
\gamma=2$.
This particular regime has already been treated by Shirai, and we recover the rate function $I(x+1) = \int_0^x J_2( t) \dif t= \frac{1}{4} (2 (x+1)^2 \log (x+1)- x(3x+2))$
from \cite[Theorem 1.1]{S06LDFermionPointProcess}.
The same proof can be adapted to compute the leading asymptotics of $\Pr{\Xi_R^{(\LL)} \le x\Theta_R}$ for $x<0$ for $\gamma\le 2$ and it should be noticed that a trivial cut-off occurs when $\gamma=2$ and implies that $J_2(x) = \infty$ for $x \le -1$.

Let us conclude this section by stating our functional central limit theorems for counting statistics. We define the processes for $R>0$ and $t\in\bR_+$,
\begin{equation*}
	\widehat\Xi_R^{(\LL)}(t) = R^{-\frac{1}{4}} \Xi_{tR}^{(\LL)}.
\end{equation*}

By applying \cref{thm:FCLTPlanar}, we obtain the following results which describe the global, respectively microscopic, fluctuations of the processes $\widehat\Xi_R^{(\LL)}$.
These results generalize \cref{thm:GinibreFCLTFDD} and \cref{thm:GinibreFCLTSkorokhod} given in the introduction in the context of the Ginibre ensemble ($\LL=0$) to Ginibre-type ensembles associated with higher Landau levels.
Let us emphasize again that the moderate deviations rate function in \eqref{eq:Mod} as well as the correlation structure of the underlying Gaussian process (at the microscopic level) depend non-trivially on the index $\LL$.
\begin{proposition}
\label{thm:GinibreTypeFCLT}
	For any fixed $\LL \in \bN_0$, it holds as $R\to\infty$ that
	\begin{equation*}
		\bigl(\widehat\Xi_R^{(\LL)}(t)\bigr)_{t\in\bR_+} \xrightarrow{\text{fdd}} \bigl(\sigma_t N_t\bigr)_{t\in\bR_+},
	\end{equation*}
	where $(N_t)_{t\in \bR_+}$ is a Gaussian white noise and $\sigma_t^2 = \sqrt{t}\int_\bR \Phi_\LL(x)(1-\Phi_\LL(x)) \dif x$. Moreover, as $R\to\infty$,
	\begin{equation*}
		\Bigl(\widehat\Xi_R^{(\LL)}\Bigl(1+\frac{t}{R}\Bigr)\Bigr)_{t\in \bR_+} \xrightarrow{\cS(\bR )} \bigl(G_t^{(\LL)}\bigr)_{t\in \bR },
	\end{equation*}
	where, for $\Phi_\LL$ is as in \eqref{phialpha}, $(G_t^{(\LL)})_{t\in \bR }$ is a centred Gaussian process with kernel
	\begin{equation*}
		\Cov[\big]{G_s^{(\LL)}}{G_t^{(\LL)}} = \int_\bR \Phi_\LL(x-s)\bigl(1-\Phi_\LL(x-t)\bigr) \dif x,
		\qquad\text{for } s\le t.
	\end{equation*}
\end{proposition}

As a final remark, let us briefly explain what happens if one considers the superposition of several Landau levels.

\begin{remark}
A model which might be physically more relevant than considering \emph{pure Landau levels} is to consider a state with all levels of degree $\le \alpha$ being filled.
This corresponds to a determinantal process $\xi_{\le \alpha}$ on $\bC$ with correlation kernel
\begin{equation} \label{Kalpha2}
K_{\le \alpha}(z,w) = \sum_{\beta\le \alpha} K_{\beta}(z,w) = \sqrt{\alpha+1} L_\LL^{(1)}\bigl(\abs{z-w}^2\bigr)e^{z\bar{w}} , \qquad z,w \in\bC.
\end{equation}
$L_\LL^{(1)}$ is a generalized Laguerre polynomial of degree $\alpha\in\bN$.
Observe that $\xi_{\le \alpha}$ does not correspond to the superposition of independent copies of $(\xi_\beta)_{\beta \le \alpha}$, but according to \eqref{Kalpha2} and \cref{lem:ModulusPoints},
the radii of the point process $\xi_{\le \alpha}$ have the same law as the collection of random variables $(\Gamma_{k}^{(\beta)})_{k\in\bN, \beta\le \alpha}$.
In particular, these random variables do not fit within \cref{NA}. However, according to \eqref{def:Xi} with $T(r) = (\alpha+1) r^2$,
$\Xi_R^{(\le \alpha)} = \sum_{k\in\bN, \beta \le \alpha} \ind_{\Gamma_{k}^{(\beta)} \le R} - (\alpha+1) R = \Xi_R^{(0)} + \cdots + \Xi_R^{(\alpha)}$ and the process $( \Xi_R^{(\beta)})_{\beta\le \alpha}$ are independent.
Using this decomposition and the results from \cref{sec:MainResults}, we obtain that for any $\alpha\in\bN_0$, the random variable $ \Xi_R^{(\le \alpha)}$ converges as $R\to\infty$ in the mod-phi sense of \cref{def:ModPhi} with speed $\sqrt{R}$, cumulant generating function $	\Lambda_{\le\alpha} $ and limiting function $\psi_{\le\alpha}$ given respectively by
\[
\Lambda_{\le\alpha} = \Lambda_0+\dotsb + \Lambda_\alpha
\qquad\text{and}\qquad
\psi_{\le\alpha} = \psi_0\dotsm\psi_\alpha.
\]
Similarly, we obtain functional limit theorems for the process $ \Xi_R^{(\le \alpha)}$ which are analogous to \cref{thm:GinibreTypeFCLT}. Specifically in the microscopic regime, we observe a centred Gaussian process $G^{(\le\alpha)}$ which is just the superposition of independent copies of
$(G^{(\beta)})_{\beta\le \alpha}$.
\end{remark}

\subsection{Hyperbolic ensembles and the zero set of hyperbolic Gaussian analytic function}
\label{sec:hyperbolic}

The Ginibre ensemble is often compared to the zero set $\zeta_1$ of the
\emph{hyperbolic Gaussian analytic function}
$F(z)= \sum_{k=0}^\infty a_k z^k$ with $(a_k)_{n\in \bN_0}$ being a sequence of i.i.d.\ standard complex Gaussian random variables.
The point process $\zeta_1$ turns out to be determinantal and its correlation kernel $K_1$ is the Bergman kernel for the unit disk $\bD = \set{z\in\bC : \abs{z} < 1 }$.
In particular, $\zeta_1$ is invariant (in law) under all linear fractional transformations of $\bD$ which preserve $\nu$, i.e.\ the group $\operatorname{SU}(1,1)$, and its intensity is the invariant measure $\dif\nu (z) = \frac{\dif m(z)}{\pi(1-\abs{z}^2)^2}$ of the \emph{Poincar\'e disk}\footnote{The Poincar\'e disk is a model of 2-dimensional hyperbolic geometry on $\bD$ with metric tensor $\dif\mathrm{s}^2 = \frac{\dif z\dif \overline{z}}{2\pi(1-\abs{z}^2)^2}$. This motivates the name hyperbolic Gaussian analytic function.}. This point process also describes the outliers of random perturbations of certain (infinite) Toeplitz matrices, see \cite{BC16OutlierEigenvalues,BZ19OutliersRandomPermutations}.

It turns out that there is a 1-parameter family of determinantal processes on $\bD$ which generalizes the zero set of the hyperbolic GAF that we call \emph{hyperbolic ensembles} as in \cite[Chapter 3]{K06ZerosRandomAnalyticFunctions}.
These point processes are also invariant under the action of $\operatorname{SU}(1,1)$ and they are associated with Bargmann-Fock spaces of analytic functions. Namely, for any $\DD>0$, the reproducing kernel of the space $\cH_0 \cap L^2\bigl(\bD, \frac{\DD \dif m(z)}{\pi(1-\abs{z}^2)^{1-\DD}}\bigr)$ is
\begin{equation*}
	K_\DD(z,w) = \frac{1}{(1-z\bar{w})^{\DD+1}} = \sum_{k\in\bN_0} \binom{k+\DD}{\DD} z^k\bar{w}^k
\end{equation*}
and we let $\zeta_\DD$ be the determinantal process on $\bD$ with kernel $ K_\DD$ with respect to
$\frac{\DD \dif m(z)}{\pi(1-\abs{z}^2)^{1-\DD}}$.
By \cite[Theorem 4]{K09RandomAnalyticFunctions}, when $\DD \in \bN$, $\zeta_\DD$ corresponds to the zero set of matrix-valued GAF with i.i.d.\ coefficients from the $\DD\times\DD$ Ginibre ensembles.
These hyperbolic ensembles have been studied in the thesis \cite{K06ZerosRandomAnalyticFunctions} and a central limit theorem for smooth linear statistics has been established in \cite{RV07ComplexDeterminantalProcesses}.
Let us record that the process $\zeta_\rho$ has intensity $\dif\nu_\rho (z) = \frac{\DD\dif m(z)}{\pi(1-\abs{z}^2)^2}$
and the parameter $\rho$ is related to the Gauss curvature of the corresponding hyperbolic disk:
$\mathcal{K} = - \nu_\rho^{-1} \Delta \log \nu_\rho= -2\pi/\rho $ where $\Delta = \partial_z\partial_{\overline{z}}$ denotes the (complex) Laplacian.
In particular as the curvature tends to 0, it holds that the push-forward
$(\sqrt{\rho}z)_\# \zeta_\rho$ converges in distribution to the Ginibre ensemble $\xi_0$ as as $\rho \to\infty$.
Since these processes are rotation-invariant, we can study the distributions of their counting statistics in growing (hyperbolic) disks using the formalism from \cref{sec:MainResults}.
We let $T(r) = \frac{\DD r^2}{1-r^2}$ for $r\in\intco{0,1}$ and after \emph{unfolding} the process according to \eqref{def:Gamma}, the modulus of the points have laws given by
\begin{equation} \label{GammaHyper}
	\Pr[\big]{\Gamma_{k}^{(\DD)} \le R}
	= \frac{\Gamma(k+\DD)}{\Gamma(k)\Gamma(\DD)} \int_0^{\frac{R}{\DD}} x^{k-1} (1+x)^{-k-\DD} \dif x, \qquad k\in\bN , \rho>0, R\in\bR_+,
\end{equation}
that is $\Gamma_{k}^{(\DD)}$ are $\DD\Betaprimedist(k,\DD)$-distributed\footnote{Recall that a random variable is $\DD\Betaprimedist(k,\DD)$-distributed if it is the image by $T$ of a $\Betadist(k,\DD)$ random variable with parameters $k,\DD>0$.}.
Then, according to \eqref{def:Xi}, we let for $\DD>0$
\begin{equation*}
	\Xi_R^{(\DD)} = \sum_{k\in\bN} \ind_{\Gamma_k^{(\DD)}\le R} - R \overset{\mathrm{d}}{=} \zeta_\DD(D_r) - \Ex[\big]{\zeta_\DD(D_r)},
	\qquad \text{with } r=\sqrt{\tfrac{R}{\DD+R}}.
\end{equation*}

In \cref{sec:HyperbolicProof} we verify that the random variables $(\Gamma_{k}^{(\DD)})_{k\in\bN}$ satisfy \cref{NA} with $\vartheta =0$, $\Sigma_R=R$ and $\Phi_\DD(x) = \Pr{Y_\rho\le x}$ where $Y_\DD$ is a gamma-distributed random variable with shape $\DD$ and rate $\DD$. Hence, by applying \cref{thm:ModPhiGeneral}, we obtain the following mod-phi convergence result.
\begin{proposition}\label{thm:HyperbolicModPhi}
	For any $\DD> 0$, the random variables $\Xi_R^{(\DD)}$ converges in the mod-phi sense at speed $R$ with cumulant generating function of the form \eqref{MGF} with $I=\bR_+$ and $\Phi(x)=\Phi_\DD(x) = \Pr{Y_\DD>x}$ for $x>0$.
	The limiting function is of the form \eqref{eq:PsiF} with $\Psi(x) = \Psi_\DD(x) = \frac{\rho^{\rho} }{2\Gamma(\DD)} \bigl( \rho x +\rho-1 \bigr) x^{\rho-1} e^{-\rho x}\ind_{x> 0}$.
\end{proposition}

Hence, by \cref{thm:ModPhiConclusions}, we obtain precise moderate and large deviations for counting statistics of these hyperbolic ensembles.
Note that in \cite{FMN16ModPhiConvergence}, the authors show that $\Xi_R^{(1)}/ R^{\frac{1}{3}}$ converges in the mod-Gaussian sense with speed $R^\frac{1}{3}$ and limiting function $\psi (z) = \exp(\frac{z^3}{36})$. It turns out that this mod-Gaussian convergence is a consequence of our \cref{thm:HyperbolicModPhi}.

By applying \cref{thm:FCLTHyperbolic}, we also obtain a functional central limit theorem for counting statistics in hyperbolic disks.

\begin{proposition}
\label{thm:HyperbolicFCLT}
	For any $\DD> 0$, it holds as $R\to\infty$ that
	\begin{equation*}
		\bigl(R^{-\frac{1}{2}} \Xi_{tR}^{(\DD)}\bigr)_{t\in \bR_+} \xrightarrow{\cS(\bR_+)} \bigl(G_t^{(\DD)}\bigr)_{t\in\bR_+}
	\end{equation*}
	where $G^{(\DD)}$ is a centred Gaussian process with covariance kernel given by
	\begin{equation*}
		\Cov[\big]{G_s^{(\DD)}}{G_t^{(\DD)}}
		= \int_{\bR_+} \Pr[\Big]{Y_\rho > \frac{x}{s}} \Pr[\Big]{Y_\rho \le \frac{x}{t}} \dif x
		\qquad\text{for }0<s\le t.
	\end{equation*}
\end{proposition}
\begin{remark}
	For $\DD=1$, we have $\Phi_1(x) = e^{-x}$ which immediately implies the explicit closed formula $\Cov{G_s^{(1)}}{G_t^{(1)}} = \frac{2s^2}{(s+t)}$ for $0<s\le t$.
\end{remark}

%

\subsection{Summary and conclusions} \label{sec:sum}

The following table summarizes our main findings for the invariant determinantal processes considered in \cref{sec:MainResults}.
\\[5pt]
\begin{tabular}{c|c|c|c|c|c}
 & Kernel & $\frac{\dif \nu}{\dif m}(z)$ & $\Gamma_k$ & $\Sigma_r$ & $Z$\\
\hline
Ginibre ensemble & $\frac{1}{\pi} e^{z\bar{w}-\frac{\abs{z}^2+\abs{w}^2}{2}}$ & $\frac{1}{\pi}$ & $\Gammadist(k,1)$ & $r$ &$\Normaldist(0,1)$ \\
Ginibre $\alpha$-type & $\frac{L_\LL(\abs{z-w}^2)}{\pi} e^{z\bar{w}-\frac{\abs{z}^2+\abs{w}^2}{2}}$ & $\frac{1}{\pi}$ & $\Gammadist_\LL(k,1)$ & $r$ &$\Normaldist_\LL(0,1)$\\
Hyperbolic GAF & $\frac{1}{\pi(1-z\bar{w})^{2}}$ & $\frac{1}{\pi (1-\abs{z}^2)^2}$ & $\Betaprimedist(k,1)$ & $\frac{\DD r^2}{1-r^2}$ &$\Expdist(1)$ \\
Hyperbolic ensembles & $\frac{\DD(1-\abs{z}^2)^{\DD-1}(1-\abs{w}^2)^{\DD-1}}{\pi(1-z\bar{w})^{\DD+1}}$ & $\frac{\DD}{\pi (1-\abs{z}^2)^2}$ & $\DD\Betaprimedist(k,\DD)$ & $\frac{\DD r^2}{1-r^2}$ & $\Gammadist(\DD,\DD)$
\end{tabular}\\[5pt]

The first two columns collect the kernel of the respective model (with respect to the Lebesgue measure) as well as the density of the first intensity measure. These densities correspond to the invariant measure on the respective space, i.e.\ on the complex plane or the hyperbolic disk.

For all models, we have established that the point count statistic in a disk of radius $r$ converges as $r\to\infty$, in the mod-$\phi$ sense at speed $\Sigma_r$ with respect to a cumulant generating function
	\begin{equation} \label{MGF2}
		 \Lambda (z) = \int_{\bR} \bigl( \log\bigl(1+\Phi(x)(e^z-1)\bigr)-\Phi(x)z \bigr) \dif x ,
	\end{equation}
where $\Phi(x) = \Pr{Z >x}$ is the probability tail function of the random variable $Z$ depicted in the last column of the table, see \cref{thm:ModPhiGeneral}.
For all the above ensembles, the speed $\Sigma_r$ is proportional to the surface (with respect to the invariance measure $\frac{\dif \nu}{\dif m}$) of the boundary of the disk of radius $r>0$.
This corroborates the \emph{area law} coming from the Coulomb gas heuristic.
Moreover, by \cref{cor:mod}, the function \eqref{MGF2} characterizes the typical fluctuations and moderate deviations of the counting statistics.
Our finding shows that there is a \emph{universal} structure underlying these asymptotic fluctuations which is governed by the random variable $Z$.
These mod-phi convergence results are obtained from the fact that the modulus of the points of the point process, after an unfolding transformation, are independent random variables $\Gamma_k$ with explicit distribution as given in the third column of the table.

In the case of the Ginibre $\alpha$-type ensembles, we have also described completely the large deviation regime for the point count statistic, see \cref{thm:ldp}, as well as the precise asymptotics for the entanglement entropy, see \cref{thm:arealaw}.

Finally, our analysis also allows us to obtain functional limit theorems for point count statistics in the form \cref{thm:FCLTPlanar} or \cref{thm:FCLTHyperbolic}. It turns out that with the correct scaling (which depends on the geometry), the underlying correlations are again completely governed by the probability distribution $\Phi$ of the random variable $Z$ found in the last column of the table.

\paragraph{Acknowledgement.}
G. L. wishes to thank Luis Daniel Abreu, Jonas Jalowy and Gregory Schehr for useful comments on the first version of this paper and for pointing out references.


\section{Proofs of main results} \label{sec:Pmain}

In the sequel, we rely on the notation from \cref{sec:Background} and write $\Xi_R = \Xi_R^{(N)}$ according to \eqref{def:XiN}, omitting the dependency on $N$.
Note that this is consistent with \eqref{def:Xi} in case $N=\infty$.
First in \cref{sec:lemA}, we provide some general (technical) conditions under which our main hypothesis (\cref{NA} with $\vartheta=1$) hold.
In \cref{sec:proofA}, we show that these conditions arise naturally in the \emph{Euclidean setting} and explain how to verify them (based on Fourier analysis).
In \cref{sec:ModPhiGeneralProof}, we prove our mod-phi result by computing the asymptotics of the cumulant generating function of the random variable $\Xi_R$.
Then, in \cref{sec:FCLT}, we give the proofs of our \emph{macroscopic} and \emph{microscopic} central limit theorems. The arguments are based on the convergence of the joint cumulants of $\Xi_R$ and can be mostly reduced to covariance computations.
It should be emphasized that these proofs rely strongly on the fact that $\Xi_R$ is a sum of independent Bernoulli random variables and our methods apply to general such sums; they do not necessarily need to come from counting statistics of determinantal processes.

\subsection{On \texorpdfstring{\cref{NA}}{Assumptions \ref*{NA}}} \label{sec:lemA}
Several proofs in \cref{sec:Pmain} rely on comparing sums to integrals.
To keep track on the errors in these \emph{Riemann sum approximations}, we use the following Euler-Maclaurin formula: For integers $L_0<L_1$ and any absolutely continuous function $f\colon\bR\to\bC$, it holds
\begin{equation}\label{eq:EulerMaclaurin}
	\sum_{k=L_0+1}^{L_1} f(k) = \int_{L_0}^{L_1} f(u)\dif u + \frac{f(L_1)-f(L_0)}{2}+ \int_{L_0}^{L_1} \Bigl(\fracpart{u}-\frac{1}{2}\Bigr) f'(u) \dif u,
\end{equation}
where $f'$ stands for the weak derivative of $f$ and $\fracpart{u}\in\intco{0,1}$ denotes the fractional part of $u\in\bR$.
Formula \eqref{eq:EulerMaclaurin} follows from a direct integration by parts which is justified by using the absolute continuity of $f$. In particular, we are also allowed to choose $L_0=-\infty$ and $L_1=\infty$ in which case $f(\pm\infty) = \lim_{L\to\pm\infty} f(L) = 0$.
Because $f(x) \le \int_{-\infty}^x f'(u) \dif u$ by absolute continuity, it holds that $\abs{f(x)} \le \norm{f'}_{L^1}$ which shows for all $-\infty\le L_0<L_1\le\infty$, as $R\to\infty$,
\begin{equation}\label{eq:EulerMaclaurinReduced}
	\sum_{k=L_0+1}^{L_1} f\Bigl(\frac{k}{R}\Bigr)
	= \int_{L_0}^{L_1} f\Bigl(\frac{u}{R}\Bigr) \dif u + O\bigl(\norm{f'}_{L^1}\bigr).
\end{equation}

The next lemma shows that if the random variables $\widetilde\Gamma_k = \frac{\Gamma_k-k}{\Sigma_k}$ are statistically approximated by $Z$ for large $k\in\bN$, then under the some technical conditions on the smoothness and tails of $Z$, \cref{NA} are satisfied with $\vartheta=1$.
We have not tried to optimize these conditions because we believe that they suffice for applications.
Note that the normalization of $\widetilde\Gamma_k$ is rather natural as it corresponds to the CLT scaling and it is expected in the \emph{Euclidean setting}.

\begin{lemma} \label{lem:A}
	Let $\Phi(x) = \Pr{Z>x}$ be as in \cref{NA}.
	Assume that there exists $0<\epsilon<\frac{1}{10}$ such that $\Ex{\abs{Z}^{m_\epsilon}}<\infty$ with $m_\epsilon > 3+\frac{1}{2\epsilon}$ and it holds as $R\to\infty$,
	\begin{align}
\notag		\Pr{\Gamma_k > R} &\le \eta_{R,k}
		&&\text{for } k<R-R^{\frac{1}{2}+\epsilon},\\
\label{Eexp}		\Pr{\Gamma_k \le R}
		&= \Phi\Bigl(\tfrac{k-R}{\sqrt{k}}\Bigr) + \frac{1}{\sqrt{k}} \Upsilon\Bigl(\tfrac{k-R}{\sqrt{k}}\Bigr) + o\bigl( \eta_{R,k}\bigr)
		&&\text{for } |k-R| \le R^{\frac{1}{2}+\epsilon} \\
\notag		\Pr{\Gamma_k \le R} &\le \eta_{R,k}
		&&\text{for }k>R+R^{\frac{1}{2}+\epsilon},
	\end{align}
	where the errors $\eta_{R,k}$ satisfy $\sum_{k=1}^\infty \eta_{R,k} \to 0$ as $R\to\infty$.
	We further assume that
	\begin{itemize}
		\item $\Phi\in\cC^2(\bR)$ and there exists $\delta>0$ such that $x\mapsto x^{3+\delta}\Phi''(x)$ is uniformly bounded.
		\item $\Upsilon \in W^{1,1}(\bR)$ is $\alpha$-Hölder continuous for $\alpha>2\epsilon$.
	\end{itemize}
	Then, $(\Gamma_k)_{k\in\bN}$ satisfy \cref{NA} with $\Sigma_R=\sqrt{R}$, $\Phi(x)=\Phi(x)$ and $\Psi(x) = \frac{x^2}{2}\Phi'(x) + \Upsilon(x)$ for all $x\in\bR$.
\end{lemma}
\begin{remark}
	If $\frac{\Gamma_k-k}{\Sigma_k}$ converges to $Z$ in distribution with $\Sigma_R = o(\sqrt{R})$, the above lemma still holds true with the only change that $\Psi(x) = \Upsilon(x)$ for all $x\in\bR$ and $\epsilon$ small enough. In case $\Sigma_R \gg \sqrt{R}$, there is an extra contribution appearing which is of larger order than $\frac{1}{\sqrt{k}}\Upsilon(\frac{k-R}{\sqrt{k}})$.
\end{remark}

\begin{proof}
	Denote
	\begin{equation*} 
		\theta_{R,k} =
		\begin{cases}
			\eta_{R,k} + 1-\Phi\bigl(\frac{k-R}{\sqrt{R}}\bigr) + \frac{1}{\sqrt{R}} \abs[\big]{\bigl(\frac{k-R}{\sqrt{R}}\bigr)^2\Phi'\bigl(\frac{k-R}{\sqrt{R}}\bigr)} + \frac{1}{\sqrt{R}} \abs[\big]{\Upsilon\bigl(\frac{k-R}{\sqrt{R}}\bigr)}&\text{if }k<R-R^{\frac{1}{2}+\epsilon},\\
			\eta_{R,k} + \abs[\big]{\Phi\bigl(\frac{k-R}{\sqrt{k}}\bigr) - \Phi\bigl(\frac{k-R}{\sqrt{R}}\bigr) - \frac{(k-R)^2}{R^{\frac{3}{2}}}\Phi'\bigl(\frac{k-R}{\sqrt{R}}\bigr)} + \abs[\big]{\frac{1}{\sqrt{k}}\Psi\bigl(\frac{k-R}{\sqrt{k}}\bigr) - \frac{1}{\sqrt{R}}\Psi\bigl(\frac{k-R}{\sqrt{R}}\bigr)}&\text{if } | k- R| \le R^{\frac{1}{2}+\epsilon},\\
			\eta_{R,k} + \Phi\bigl(\frac{k-R}{\sqrt{R}}\bigr) + \frac{1}{\sqrt{R}} \abs[\big]{\bigl(\frac{k-R}{\sqrt{R}}\bigr)^2\Phi'\bigl(\frac{k-R}{\sqrt{R}}\bigr)} + \frac{1}{\sqrt{R}} \abs[\big]{\Upsilon\bigl(\frac{k-R}{\sqrt{R}}\bigr)}&\text{if }k>R+R^{\frac{1}{2}+\epsilon}
		\end{cases}
	\end{equation*}
	and notice that for this choice we have \eqref{cond}. We need to show that $\sum_{k\in\bN} \theta_{R,k} \to 0$ as $R\to\infty$ which is done below. The additional conditions on $\Phi$ in \cref{NA} are satisfied since $\frac{\Sigma_R^{m_\epsilon+1}}{R^{m_\epsilon}} = R^{\frac{1-m_\epsilon}{2}} \to 0$ as $R\to\infty$.
	Because $\Ex{\abs{Z}^2}<\infty$, the function $x\mapsto x^2\Phi'(x)$ is in $L^1(\bR)$. Because $x\mapsto x^{3+\delta}\Phi''(x)$ is uniformly bounded and continuous, $x\mapsto x^2\Phi''(x)$ is also in $L^1$. This implies that $x\mapsto x^2\Phi'(x)$ is absolutely continuous. As $\Upsilon\in W^{1,1}(\bR)$ by assumption, we conclude that $\Psi$ satisfies the additional conditions in \cref{NA}.

	Let us start by showing that $\sum_{k>R+R^{\frac{1}{2}+\epsilon}} \theta_{R,k} \to 0$ as $R\to\infty$.
	First, notice that for all $x>0$ it holds
	\begin{equation*}
		\abs{x}^k \Phi(x)
		= \abs{x}^k \int_x^\infty \abs{\Phi'(t)} \dif t
		\le \Ex[\big]{\abs{Z}^k}.
	\end{equation*}
	Hence, the moment condition $\Ex{\abs{Z}^{m_\epsilon}}<\infty$ implies that $\Phi(x) \le C_\epsilon x^{-m_\epsilon}$ for some constant $C_\epsilon>0$ and all $x\in\bR_+$.
	Therefore, we have that
	\begin{equation}\label{eq:CLTABdd1}
		\sum_{k=R+R^{\frac{1}{2}+\epsilon}}^\infty \Phi\Bigl(\frac{k-R}{\sqrt{R}}\Bigr)
		\le C_\epsilon R^{\frac{m_\epsilon}{2}} \sum_{k=R^{\frac{1}{2}+\epsilon}}^\infty k^{-m_\epsilon}
		\le C_\epsilon R^\frac{m_\epsilon}{2} \int_{R^{\frac{1}{2}+\epsilon}-1}^\infty x^{-m_\epsilon} \dif x
		\le C_\epsilon R^{\frac{1}{2}+\epsilon-\epsilon m_\epsilon}
	\end{equation}
	and the right-hand side converges to zero as soon as $m_\epsilon>1+\frac{1}{2\epsilon}$. Because $\Phi''$ is bounded, $\Phi'$ is Lipschitz continuous. Using that $\Ex{\abs{Z}^{m_\epsilon}} < \infty$, this implies that $x\mapsto \abs{x^{m_\epsilon} \Phi'(x)}$ is uniformly bounded and therefore there exists a constant $C_\epsilon>0$ such that $\abs{\Phi'(x)} \le C_\epsilon \abs{x}^{-m_\epsilon}$ and all $x\in\bR$. This shows that
	\begin{equation}\label{eq:CLTABdd2}
		\sum_{k=R+R^{\frac{1}{2}+\epsilon}}^\infty \abs[\bigg]{\Bigl(\frac{k-R}{\sqrt{R}}\Bigr)^2 \Phi'\Bigl(\frac{k-R}{\sqrt{R}}\Bigr)}
		\le C_\epsilon R^{\frac{m_\epsilon}{2}-1} \sum_{k=R^{\frac{1}{2}+\epsilon}}^\infty k^{-m_\epsilon+2}
		\le C_\epsilon R^{\frac{1}{2}+3\epsilon-\epsilon m_\epsilon}
	\end{equation}
	which converges to zero for $m_\epsilon>3+\frac{1}{2\epsilon}$. To control the sum corresponding to $\Upsilon$, we apply Euler-Maclaurin's formula \eqref{eq:EulerMaclaurinReduced} to get
	\begin{equation}\label{eq:CLTABdd3}
		\sum_{k=R+R^{\frac{1}{2}+\epsilon}}^\infty \frac{1}{\sqrt{R}} \abs[\Big]{\Upsilon\Bigl(\frac{k-R}{\sqrt{R}}\Bigr)}
		= \int_{R+R^{\frac{1}{2}+\epsilon}}^\infty \frac{1}{\sqrt{R}} \abs[\Big]{\Upsilon\Bigl(\frac{x-R}{\sqrt{R}}\Bigr)} \dif x + O(R^{-\frac{1}{2}})
		= \int_{R^\epsilon}^\infty \abs{\Upsilon(y)} \dif y + O(R^{-\frac{1}{2}}),
	\end{equation}
	where we used the change of variables $y = \frac{x-R}{\sqrt{R}}$. The right-hand side converges to zero as $R\to\infty$ since $\Upsilon \in L^1(\bR)$. Combining \eqref{eq:CLTABdd1}, \eqref{eq:CLTABdd2} and \eqref{eq:CLTABdd3} with the assumption that $\sum_{k\in\bN} \eta_{R,k} \to 0$ as $R\to\infty$ yields that $\sum_{k>R+R^{\frac{1}{2}+\epsilon}} \theta_{R,k} \to 0$ as $R\to\infty$.

	To control $\sum_{k=1}^{R-R^{\frac{1}{2}+\epsilon}} \theta_{R,k}$, the same arguments as for \eqref{eq:CLTABdd2} and \eqref{eq:CLTABdd3} can be applied to show that $\sum_{k=1}^{R-R^{\frac{1}{2}+\epsilon}} \abs{(\frac{k-R}{\sqrt{R}})^2 \Phi'(\frac{k-R}{\sqrt{R}})} + \frac{1}{\sqrt{R}} \abs{\Upsilon(\frac{k-R}{\sqrt{R}})} \to 0$ as $R\to\infty$. The analogue statement for \eqref{eq:CLTABdd1} follows from the bound
	\begin{equation*}
		\abs{x}^k (1-\Phi(x))
		= \abs{x}^k \int_{-\infty}^x \abs{\Phi'(t)} \dif t
		\le \Ex[\big]{\abs{Z}^k}
	\end{equation*}
	for all $x<0$.

	Let us now treat the case corresponding to $k\in\intcc{R-R^{\frac{1}{2}+\epsilon},R+R^{\frac{1}{2}+\epsilon}}$. A Taylor expansion argument shows that there exists $x\in\intcc{\frac{k-R}{\sqrt{k}},\frac{k-R}{\sqrt{R}}}$ such that
	\begin{equation*}
		\Phi\Bigl(\frac{k-R}{\sqrt{k}}\Bigr)
		= \Phi\Bigl(\frac{k-R}{\sqrt{R}}\Bigr) + \Phi'\Bigl(\frac{k-R}{\sqrt{R}}\Bigr) \Bigl(\frac{k-R}{\sqrt{R}}-\frac{k-R}{\sqrt{k}}\Bigr) + \frac{1}{2}\Phi''(x)\Bigl(\frac{k-R}{\sqrt{R}}-\frac{k-R}{\sqrt{k}}\Bigr)^2.
	\end{equation*}
	To control the second and third term, note that it holds uniformly for all $k\in\intcc{R-R^{\frac{1}{2}+\epsilon},R+R^{\frac{1}{2}+\epsilon}}$ as $R\to\infty$ that
	\begin{equation}\label{eq:ExpansionSqrt}
		\frac{1}{\sqrt{k}}
		= \frac{1}{\sqrt{R}} - \frac{k-R}{2R^\frac{3}{2}} + O\biggl(\frac{(k-R)^2}{R^\frac{5}{2}}\biggr)
		= \frac{1}{\sqrt{R}} - \frac{k-R}{2R^\frac{3}{2}} + O\bigl(R^{-\frac{3}{2}+2\epsilon}\bigr).
	\end{equation}
	Because $\Phi'$ and $\Phi''$ are uniformly bounded, we obtain that
	\begin{equation}\label{eq:ExpansionPhi}
		\Phi\Bigl(\frac{k-R}{\sqrt{k}}\Bigr)
		= \Phi\Bigl(\frac{k-R}{\sqrt{R}}\Bigr) + \frac{(k-R)^2}{2R^\frac{3}{2}} \Phi'\Bigl(\frac{k-R}{\sqrt{R}}\Bigr) + O\bigl(R^{-1+4\epsilon}\bigr).
	\end{equation}
	Next, by combining the $\alpha$-Hölder continuity of $\Upsilon$ with \eqref{eq:ExpansionSqrt} there exists a constant $C>0$ such that for all $k\in\intcc{R-R^{\frac{1}{2}+\epsilon},R+R^{\frac{1}{2}+\epsilon}}$ it holds
	\begin{equation}\label{eq:ExpansionPsi}
		\abs[\Big]{\frac{1}{\sqrt{k}} \Upsilon\Bigl(\frac{k-R}{\sqrt{k}}\Bigr) - \frac{1}{\sqrt{R}} \Upsilon\Bigl(\frac{k-R}{\sqrt{R}}\Bigr)}
		\le \frac{C}{\sqrt{R}} \abs[\Big]{\frac{k-R}{\sqrt{k}} - \frac{k-R}{\sqrt{k}}}^\alpha + \frac{\norm{\Upsilon}_\infty \abs{k-R}}{R^\frac{3}{2}}
		\le C R^{-\frac{1+\alpha}{2}+2\epsilon\alpha}
	\end{equation}
	Since the error terms in \eqref{eq:ExpansionPhi} and \eqref{eq:ExpansionPsi} are $o(R^{-\frac{1}{2}-\epsilon})$ for $\epsilon<\frac{1}{10}$ and $\alpha\ge 2\epsilon$, we conclude that $\sum_{k=R-R^{\frac{1}{2}+\epsilon}}^{R+R^{\frac{1}{2}+\epsilon}} \theta_{R,k} \to 0$ as $R\to\infty$.
\end{proof}

\subsection{Proof of \texorpdfstring{\cref{thm:ModPhiGeneral}}{Theorem \ref*{thm:ModPhiGeneral}}}
\label{sec:ModPhiGeneralProof}

\Cref{thm:ModPhiGeneral} is a statement on mod-phi convergence for sums of independent Bernoulli random variables.
The general strategy of the proof consists in showing that the cumulant generating function of the random variable $\Xi_R$ can be approximated by $ \Sigma_R \Lambda + \psi$ using \eqref{cond} and a Riemann sum argument.

Recall that $\kappa_p$ denotes the cumulant generating function of a Bernoulli random variable with parameter $p\in[0,1]$, \eqref{def:kappa}.
We begin by giving uniform bounds on $\kappa_p$ and its derivatives.
\begin{proposition}\label{lem:BoundCGFBernoulli}
	The functions $(p,z)\mapsto \kappa_p(z)$, $(p,z)\mapsto \dot\kappa_p(z)$ and $(p,z)\mapsto \ddot\kappa_p(z)$ are analytic in $(0,1) \times \bS$.	Moreover, for any compact set $A\subseteq \bS$, they are bounded uniformly in $\intcc{0,1}\times A$, i.e.\ there exists a constant $C_A>0$ such that
	\begin{align*}
		\sup_{z\in A}\sup_{p\in\intcc{0,1}} \set[\big]{ \abs[\big]{\kappa_p(z)}, \abs[\big]{\dot\kappa_p(z)}, \abs[\big]{\ddot\kappa_p(z)} } &\le C_A.
	\end{align*}
	In particular, for any $p\in\intcc{0,1}$ it holds
	\begin{equation*}
		\sup_{z\in A} \abs{\kappa_p(z)} \le 2C_Ap(1-p).
	\end{equation*}
\end{proposition}
We omit the proof which follows from the definition of $\kappa_p$ and the computation of its derivatives.
In particular, for the second bound, note that $\kappa_p(0) = \kappa_p(1)=0$ for any $p\in[0,1]$.

Recall \eqref{def:Xi} and \eqref{def:XiN} and that $\Gamma_k$ are independent non-negative random variables.
Let $\lambda_k = \Pr{\Gamma_k\le R}$ for $k\in\bN$ (we omit the dependency on $R$) so that for $z\in\bS$
\begin{equation*}
	\log \Ex[\big]{e^{z\Xi_R}} = \sum_{k=1}^N \kappa_{\lambda_k}(z).
\end{equation*}
Recall that in the infinite case we have $N=\infty$ and in the finite $N$ case we choose $R$ according to \eqref{interval}.
We now apply the asymptotic expansion \eqref{cond}.
\begin{proposition}\label{lem:TaylorExpansion}
	Let $A\subseteq\bS$ be any compact set. As $R\to\infty$, uniformly for all $z\in A$ it holds
	\begin{equation} \label{eq:TaylorExpansion}
		\sum_{k=1}^N \kappa_{\lambda_k}(z)
		= \sum_{k=1}^N \kappa_{\Phi\bigl(\frac{k-\vartheta R}{\Sigma_R}\bigr)}(z) + \frac{1}{\Sigma_R} \sum_{k=1}^N \Psi\Bigl(\tfrac{k-\vartheta R}{\Sigma_R}\Bigr) \dot\kappa_{\Phi\bigl(\frac{k-\vartheta R}{\Sigma_R}\bigr)}(z) + o_A(1).
	\end{equation}
\end{proposition}
\begin{proof}
	By Taylor expansion of $p\mapsto \kappa_p(z)$, it holds for any $z\in \bS$ that
	\begin{align*}
		\sum_{k=1}^N \kappa_{\lambda_k}(z)
		= \sum_{k=1}^N \kappa_{\Phi (x_k)}(z) + \sum_{k=1}^N \bigl(\lambda_k-\Phi (x_k)\bigr) \dot\kappa_{\Phi (x_k)}(z)
		+ \sum_{k=1}^N \frac{\bigl(\lambda_k-\Phi (x_k)\bigr)^2}{2} \ddot\kappa_{p_k}(z)
	\end{align*}
	with $x_k=\frac{k-\vartheta R}{\Sigma_R}$ and $p_k\in \intcc{\lambda_k, \Phi (x_k)}$ for all $k$. By \cref{lem:BoundCGFBernoulli}, $\dot\kappa_p$ is bounded locally uniformly, so that by the expansion \eqref{cond} it holds
	\begin{align*}
		\sum_{k=1}^N \bigl(\lambda_k-\Phi (x_k)\bigr) \dot\kappa_{\Phi (x_k)}(z)
		= \frac{1}{\Sigma_R} \sum_{k=1}^N \Psi(x_k) \dot\kappa_{\Phi (x_k)}(z) + o_A(1)
	\end{align*}
	locally uniformly for $z\in\bS$. Since $\Psi$ is uniformly bounded on $\bR$, we also have
	\begin{equation*}
		\sum_{k=1}^N \frac{\bigl(\lambda_k-\Phi (x_k)\bigr)^2}{2} \ddot\kappa_{p_k}(z)
		= o_A(1)
	\end{equation*}
	locally uniformly for $z\in\bS$. This proves the claim.
\end{proof}
Let us obtain the asymptotics of the terms on the RHS of \eqref{eq:TaylorExpansion}. We use Euler-Maclaurin's formula \eqref{eq:EulerMaclaurin} to control the error terms in the Riemann sum approximation.
\begin{proposition}\label{lem:ModPhiEM1}
	Let $A\subseteq\bS$ be any compact set. As $R\to\infty$, uniformly for all $z\in A$ it holds
	\begin{equation*}
		\sum_{k=1}^N \kappa_{\Phi \bigl(\frac{k-\vartheta R}{\Sigma_R}\bigr)}(z)
		= \Sigma_R \int_I \kappa_{\Phi (x)}(z) \dif x + \frac{\kappa_{\Phi(a^+)}(z)}{2} + o_A(1).
	\end{equation*}
\end{proposition}
\begin{proof}
	We apply Euler-Maclaurin's formula \eqref{eq:EulerMaclaurin} to $f(u) = \kappa_{\Phi (\frac{u-\vartheta R}{\Sigma_R})}(z)$ to obtain
	\begin{equation*}
		\sum_{k=1}^N \kappa_{\Phi \bigl(\frac{k-\vartheta R}{\Sigma_R}\bigr)}(z)
		= \int_0^N f(u)\dif u + \frac{f(N)+f(0)}{2} + \int_0^N \Bigl(\fracpart{u}-\frac{1}{2}\Bigr) f'(u) \dif u.
	\end{equation*}
	In the above formula, $f'\in L^1(\bR)$ is the weak derivative of $f$ which exists since $\Phi$ is absolutely continuous by assumption and $\kappa$ is analytic. In the following, we treat all three summands separately.

	First by a change of variable $x = \frac{u-\vartheta R}{\Sigma_R}$ we get
	\begin{equation*}
		 \int_0^N f(u)\dif u
		 = \Sigma_R \int_{-\frac{\vartheta R}{\Sigma_R}}^{\frac{N-\vartheta R}{\Sigma_R}} \kappa_{\Phi (x)}(z) \dif x.
	\end{equation*}
	If $\vartheta=1$, then since $\Ex{\abs{Z}^m}<\infty$ a similar argument as used in \eqref{eq:BoundExpectationZ} we have that $\abs{x}^m\Phi(x)(1-\Phi(x)) \le 2\Ex{\abs{Z}^m} < \infty$ $x\in\bR$. \Cref{lem:BoundCGFBernoulli} then yields that
	\begin{equation*}
		\abs[\bigg]{\Sigma_R \int_{-\infty}^{-\frac{R}{\Sigma_R}} \kappa_{\Phi(x)}(z) \dif x}
		\le 2C_A\Sigma_R \int_{-\infty}^{-\frac{R}{\Sigma_R}} \Phi(x)\bigl(1-\Phi(x)\bigr) \dif x
		\le 2C_A\Ex{\abs{Z}^m} \frac{\Sigma_R^{m+1}}{R^m}
		= o_A(1)
	\end{equation*}
	uniformly for $z\in A$.
	In case $N<\infty$, notice for the upper bound that by \cref{lem:BoundCGFBernoulli} and by the assumptions on the speed of convergence in \eqref{interval} it holds
	\begin{equation*}
		\abs[\bigg]{\Sigma_R \int_\frac{N-\theta R}{\Sigma_R}^{a^+} \kappa_{\Phi(x)}(z) \dif x}
		\le 2C_A \Sigma_R \int_\frac{N-\theta R}{\Sigma_R}^{a^+} \Phi(x)\bigl(1-\Phi(x)\bigr) \dif x
		= o_A(1).
	\end{equation*}
	Hence, we conclude that uniformly for $z\in A$ it holds
	\begin{equation*}
		\Sigma_R \int_{-\frac{\vartheta R}{\Sigma_R}}^\frac{N-\vartheta R}{\Sigma_R} \kappa_{\Phi(x)}(z) \dif x
		= \Sigma_R \int_I \kappa_{\Phi(x)}(z) \dif x + o_A(1).
	\end{equation*}

	For the second term, we have $\Phi(-\frac{\vartheta R}{\Sigma_R}) \to \Phi(a^-)$ and $\Phi (\frac{N-\vartheta R}{\Sigma_R}) \to \Phi(a^+)$ because $\Phi$ is the probability tail function of an absolutely continuous random variable. Moreover, \eqref{cond} evaluated at $k=1$ shows that $\Phi(a^-)=1$. Since $p \mapsto \kappa_p(z)$ is continuous (uniformly for all $z\in A$ by \cref{lem:BoundCGFBernoulli}), this implies that $f(0) \to \kappa_{\Phi(a^-)}(z) = \kappa_1(z) = 0$ and $f(N) \to \kappa_{\Phi(a^+)}(z)$ as $R\to \infty$.
	Hence, we showed that uniformly for all $z\in A$, as $R\to\infty$, it holds
	\begin{equation*}
		\sum_{k=1}^N \kappa_{\Phi \bigl(\frac{k-\vartheta R}{\Sigma_R}\bigr)}(z)
		= \Sigma_R \int_{I} \kappa_{\Phi (x)}(z) \dif x + \frac{\kappa_{\Phi(a^+)}(z)}{2} +
		 \int_0^N \Bigl(\fracpart{u}-\frac{1}{2}\Bigr)f'(u) \dif u + o_A(1).
	\end{equation*}

	Finally, let us investigate the last term
	\begin{equation*}
		\int_0^N \Bigl(\fracpart{u}-\frac{1}{2}\Bigr)f'(u) \dif u
		= \int_{-\frac{\vartheta R}{\Sigma_R}}^\frac{N-\vartheta R}{\Sigma_R} \Bigl(\fracpart[\big]{\Sigma_R x + \vartheta R}-\frac{1}{2}\Bigr) \dot\kappa_{\Phi(x)}(z) \Phi'(x) \dif x,
	\end{equation*}
	where we again used the change of variables $x=\frac{u-\vartheta R}{\Sigma_R}$. The idea is to apply another integration by parts to conclude that the term is vanishing. Since $\Phi'\in L^1(\bR)$ is not necessarily absolutely continuous this is not possible, and we approximate it first by a more regular function. Fix $\epsilon>0$ and let $g$ be a smooth function with compact support such that $\norm{\Phi' - g}_{L^1} \le \epsilon$. Then,
	\begin{align*}
		\MoveEqLeft \int_{-\frac{\vartheta R}{\Sigma_R}}^\frac{N-\vartheta R}{\Sigma_R} \Bigl(\fracpart[\big]{\Sigma_R x+\vartheta R}-\frac{1}{2}\Bigr) \dot\kappa_{\Phi(x)}(z) g(x) \dif x\\
		&= \frac{1}{12\Sigma_R} \biggl(\dot\kappa_{\Phi(\frac{N-\vartheta R}{\Sigma_R})}(z) g\Bigl(\frac{N-\vartheta R}{\Sigma_R}\Bigr) - \dot\kappa_{\Phi(-\frac{\vartheta R}{\Sigma_R})}(z) g\Bigl(-\frac{\vartheta R}{\Sigma_R}\Bigr)\biggr)\\
		&\qquad - \frac{1}{2\Sigma_R} \int_{-\frac{\vartheta R}{\Sigma_R}}^\frac{N-\vartheta R}{\Sigma_R} \Bigl(\fracpart[\big]{\Sigma_Rx+\vartheta R}^2 - \fracpart[\big]{\Sigma_R x+\vartheta R} - \frac{1}{6}\Bigr) \Bigl(\ddot\kappa_{\Phi(x)}(z) \Phi'(x)g(x) + \dot\kappa_{\Phi(x)}(z) g'(x)\Bigr) \dif x
	\end{align*}
	which clearly converges to zero uniformly for all $z\in A$ as $g$ has compact support and all other involved functions are (locally uniformly) bounded. Finally, \cref{lem:BoundCGFBernoulli} implies that, uniformly for $z\in A$,
	\begin{align*}
		\MoveEqLeft \abs[\bigg]{\int_{-\frac{\vartheta R}{\Sigma_R}}^\frac{N-\vartheta R}{\Sigma_R} \Bigl(\fracpart[\big]{\Sigma_R x+\vartheta R}-\frac{1}{2}\Bigr) \dot\kappa_{\Phi(x)}(z) g(x) \dif x - \int_{-\frac{\vartheta R}{\Sigma_R}}^\frac{N-\vartheta R}{\Sigma_R} \Bigl(\fracpart[\big]{\Sigma_R x+\vartheta R}-\frac{1}{2}\Bigr) \dot\kappa_{\Phi(x)}(z) \Phi'(x) \dif x}\\
		&\le \frac{C_A}{2} \norm[\big]{g-\Phi'}_{L^1}
		\le \frac{C_A}{2} \epsilon.
	\end{align*}
	This proves the local uniform convergence to zero for the last term in the Euler-Maclaurin formula and hence concludes the proof.
\end{proof}
\begin{proposition}\label{lem:ModPhiEM2}
	Let $A\subseteq\bS$ be any compact set. As $R\to\infty$, uniformly for all $z\in A$ it holds
	\begin{align*}
		\frac{1}{\Sigma_R} \sum_{k=1}^N \Psi\Bigl(\frac{k-\vartheta R}{\Sigma_R}\Bigr) \dot\kappa_{\Phi\bigl(\frac{k-\vartheta R}{\Sigma_R}\bigr)}(z)
		= \int_I \Psi(x) \dot\kappa_{\Phi(x)}(z) \dif x + o_A(1).
	\end{align*}
\end{proposition}
\begin{proof}
	The proof follows as the one for \cref{lem:ModPhiEM1}. Recall that by \cref{lem:BoundCGFBernoulli} $\dot\kappa_{\Phi(x)}(z)$ and $\ddot\kappa_{\Phi(x)}(z)$ are uniformly bounded for all $z\in A$ and all $x\in\bR$. Hence, by applying the Euler-Maclaurin formula \eqref{eq:EulerMaclaurinReduced} with $f(u) = \Psi(\frac{u-\vartheta R}{\Sigma_R}) \dot\kappa_{\Phi(\frac{u-\vartheta R}{\Sigma_R})}(z)$ we get
	\begin{equation*}
		\frac{1}{\Sigma_R} \sum_{k=1}^N \Psi\Bigl(\frac{k-\vartheta R}{\Sigma_R}\Bigr) \dot\kappa_{\Phi\bigl(\frac{k-\vartheta R}{\Sigma_R}\bigr)}(z)
		= \frac{1}{\Sigma_R}\int_0^N f(u) \dif u + O_A\Bigl(\frac{\norm{f'}_{L^1}}{\Sigma_R}\Bigr)
		= \frac{1}{\Sigma_R}\int_0^N f(u) \dif u + o_A(1)
	\end{equation*}
	uniformly for $z\in A$.
	For the first term, the change of variables $x = \frac{u-\vartheta R}{\Sigma_R}$ implies
	\begin{equation*}
		\frac{1}{\Sigma_R} \int_0^N f(u) \dif u
		= \int_{-\frac{\vartheta R}{\Sigma_R}}^\frac{N-\vartheta R}{\Sigma_R} \Psi(x) \dot\kappa_{\Phi(x)}(z) \dif x.
	\end{equation*}
	Because $\dot\kappa_p$ is locally uniformly bounded, $\Psi \in L^1(\bR)$ and $\intoo{-\frac{\vartheta R}{\Sigma_R},\frac{N-\vartheta R}{\Sigma_R}} \to I$ as $R\to\infty$ by \cref{NA}, we conclude using \cref{lem:BoundCGFBernoulli} that
	\begin{equation*}
		\int_{-\frac{\vartheta R}{\Sigma_R}}^\frac{N-\vartheta R}{\Sigma_R} \Psi(x) \dot\kappa_{\Phi(x)}(z) \dif x
		= \int_{I} \Psi(x) \dot\kappa_{\Phi(x)}(z) \dif x + o_A(1)
	\end{equation*}
	uniformly for all $z\in A$. This proves the claimed uniform asymptotic expansion.
\end{proof}
\begin{proof}[Proof of \cref{thm:ModPhiGeneral}]
	The theorem follows immediately by combining \cref{lem:TaylorExpansion}, \cref{lem:ModPhiEM1} and \cref{lem:ModPhiEM2}. Analyticity of $\Lambda$ and $\psi$ follow from the analyticity of $z\mapsto\kappa_p(z)$ for all $p\in\intcc{0,1}$ together with the locally uniform bounds on $\kappa$ and $\dot\kappa$ from \cref{lem:BoundCGFBernoulli} by Morera's theorem: Since $\sup_{z\in A} \int_I \abs{\dot\kappa_{\Phi(x)}(z)} \dif x \le 2C_A \int_I \Phi(x)(1-\Phi(x)) \dif x < \infty$ for all compact sets $A\subseteq \bS$ (see \eqref{eq:BoundExpectationZ}), we can apply Fubini's theorem to get
	\begin{equation*}
		\oint_\gamma \Lambda(z) \dif z = \int_I \oint_\gamma \kappa_{\Phi(x)}(z) \dif z \dif x = 0
	\end{equation*}
	for all closed paths $\gamma$. This shows that $\Lambda$ is analytic. The same proof applies to show the analyticity of $\psi$.
	It only remains to prove that the function $\Lambda$ indeed satisfy \eqref{ass:DecayCGF}. A direct calculation shows for any $z\in\bR$ and $y\in\intcc{-\pi,\pi}$,
	\begin{align*}
		\abs[\Big]{\exp\bigl(\Lambda (x+iy)-\Lambda (x)\bigr)}
		&= \exp\Bigl(\int_I \log\abs[\big]{1+\Phi (u)\bigl(e^{x+iy}-1\bigr)} - \log\abs[\big]{1+\Phi (u)\bigl(e^x-1\bigr)} \dif u\Bigr)\\
		&= \exp\biggl(\frac{1}{2} \int_I \log\biggl(1-\frac{2\Phi (u)\bigl(1-\Phi (u)\bigr)e^x\bigl(1-\cos y\bigr)}{\bigl(1+\Phi (u)(e^x-1)\bigr)^2} \biggr)\dif u\biggr)\\
		&\le \exp\Bigl(-\bigl(1-\cos y\bigr) e^{-|x|} \int_I \Phi (u)\bigl(1-\Phi (u)\bigr) \dif u\Bigr)\\
		&\le \exp\Bigl(- \frac{ y^2 e^{-|x|}}{12} \int_{I} \Phi (u)\bigl(1-\Phi (u)\bigr) \dif u\Bigr),
	\end{align*}
	where we used that $1- \cos(y) \ge y^2/12$ for all $y\in [-\pi,\pi]$. Since $\int_{I} \Phi (u)\bigl(1-\Phi (u)\bigr) \dif u < \infty$ this proves the condition \eqref{ass:DecayCGF}.
\end{proof}
\begin{remark}
	When working under the more general assumptions from \cref{rk:BV}, we cannot apply the Euler-Maclaurin formula to $\Psi_M$ any more. Instead, monotonicity allows us to apply a Riemann sum argument directly.
\end{remark}


\subsection{Proofs of \texorpdfstring{\cref{thm:FCLTPlanar} and \cref{thm:FCLTHyperbolic}}{Theorem \ref*{thm:FCLTPlanar} and Theorem \ref*{thm:FCLTHyperbolic}}}
\label{sec:FCLT}

Let $J\subseteq\bR$ be a compact interval and (with a slight abuse of notation) $\RT_R=\RT\colon J \to\bR_+$ be a function increasing in $J$ such that $\lim_{R\to\infty} \RT(t) = \infty$ for all $t\in J$; we use $\RT(t) = tR$ and $\RT(t) = R+t\Sigma_R$ later. Let us denote
\begin{equation} \label{Xihat}
	\widehat{\Xi}_R(t)
	= \frac{\Xi_{\RT(t)}^{(N)}}{ \sqrt{\Sigma_{R}}}
	= \frac{\sum_{k=1}^N \ind_{\Gamma_k \le \RT(t)} - \Pr{\Gamma_k \le \RT(t)}}{ \sqrt{\Sigma_{R}}}.
\end{equation}

Note that we omitted the dependency on $N$. The processes $\widehat{\Xi}_R$ are asymptotically normalized. If $N<\infty$, we again choose a sequence $R=R(N)$ which satisfies one of the conditions \eqref{Ncond} and the interval $I$ is given by~\eqref{interval}.

In the context of \cref{thm:FCLTPlanar} and \cref{thm:FCLTHyperbolic}, $\widehat{\Xi}_R(t) $ represents the fluctuations of the number of points in disks of radius $T^{-1}(\RT(t))$; see \eqref{def:XiN}.
The convergence of $(\widehat{\Xi}_R)_{R\in\bR_+}$ as a stochastic process relies on a general central limit theorem for sums of independent Bernoulli random variables which follows basically just from the convergence of covariances.

Within this section we denote $\lambda_{k,t} = \Pr{\Gamma_k \le \RT(t)}$ for $k\in\bN$ and $t\in J$ and rely on the property: $\Var{\ind_{\Gamma_k \le \RT(t)}} = \lambda_{k,t}(1-\lambda_{k,t}) $. Let us first state a multivariate central limit theorem for the processes~$\widehat{\Xi}_R$.
The argument bears similarity with \cite[Lemma 2]{S02GaussianLimitDPP}.
\begin{lemma}\label{lem:FCLTFDD}
	Assume that there exists a function $c\colon J\times J \to\bR $ such that for all $s,t\in J$
	\begin{equation*}
		\lim_{R\to\infty} \Cov[\big]{\widehat{\Xi}_R(s)}{\widehat{\Xi}_R(t)} = c(s,t).
	\end{equation*}
	Then, it holds as $R\to\infty$ that
	\begin{equation*}
		\bigl(\widehat{\Xi}_R(t)\bigr)_{t\in J} \xrightarrow[]{\text{fdd}} \bigl(G_t\bigr)_{t\in J},
	\end{equation*}
	where $(G_t)_{t\in J}$ is a centred Gaussian process with covariance kernel $\Cov{G_s}{G_t} = c(s,t)$ for all $s,t\in J$.
\end{lemma}
\begin{proof}
	Fix $n\in\bN$. We show for any sequence $t_1 \le \dotsb \le t_n$, $t_i\in J$, and for any choice $a_1,\dotsc,a_n\in\bR$ that as $R\to\infty$ it holds
	\begin{equation*}
		a_1\widehat{\Xi}_R(t_1) + \dotsb + a_n\widehat{\Xi}_R(t_n) \xrightarrow{\mathrm{d}} a_1G_{t_1} + \dotsb + a_nG_{t_n}.
	\end{equation*}
	Then, the result follows by the Cramèr-Wold device. Clearly, the mean of both sides is equal to zero. The convergence of the variance follows from the assumed convergence of covariances by bilinearity.
	Finally, we show that the higher order cumulants converge to zero. First, by additivity of cumulants on independent random variables, notice for $q\in\bN$, $q\ge 3$, that
	\begin{equation*}
		\varkappa^{(q)}\bigl(a_1\widehat{\Xi}_R(t_1) + \dotsb + a_n\widehat{\Xi}_R(t_n)\bigr)
		= \sum_{k=1}^N \varkappa^{(q)}\biggl(\sum_{i=1}^n \frac{a_i}{\sqrt{\Sigma_R}} \bigl(\ind_{\Gamma_k\le \RT(t_i)} - \lambda_{k,t_i}\bigr)\biggr).
	\end{equation*}
	Denote by $\cQ_q$ the set of all set partitions of $\set{1,\dotsc,q}$. The moment-cumulant formula then implies
	\begin{align*}
		\MoveEqLeft \abs[\bigg]{ \varkappa^{(q)}\biggl(\sum_{i=1}^n \frac{a_i}{\sqrt{\Sigma_R}} \bigl(\ind_{\Gamma_k\le \RT(t_i)} - \lambda_{k,t_i}\bigr)\biggr) }\\
		&= \abs[\Bigg]{ \sum_{\Pi\in\cQ_q} (-1)^{\abs{\Pi}-1} \bigl(\abs{\Pi}-1\bigr)! \prod_{\pi\in\Pi} \Ex[\bigg]{\biggl(\sum_{i=1}^n \frac{a_i}{\sqrt{\Sigma_R}} \bigl(\ind_{\Gamma_k\le \RT(t_i)} - \lambda_{k,t_i}\bigr)\biggr)^{\abs{\pi}}} }\\
		&\le \sum_{\Pi\in\cQ_q} \bigl(\abs{\Pi}-1\bigr)! \biggl(\sum_{i=1}^n \frac{\abs{a_i}}{\sqrt{\Sigma_R}}\biggr)^{q-2} \Var[\bigg]{\sum_{i=1}^n \frac{a_i}{\sqrt{\Sigma_R}} \bigl(\ind_{\Gamma_k\le \RT(t_i)} - \lambda_{k,t_i}\bigr)},
	\end{align*}
	where in the first inequality we used that all blocks $\pi\in\Pi$ have size at least two (since the involved random variables are centred) and that $\sum_{i=1}^n \frac{a_i}{\sqrt{\Sigma_R}} (\ind_{\Gamma_k\le \RT(t_i)} - \lambda_{k,t_i})$ is bounded by $\sum_{i=1}^n \frac{\abs{a_i}}{\sqrt{\Sigma_R}}$. Next, notice that $\sum_{\Pi\in\cQ_q}(\abs{\Pi}-1)! \le 2^q q!$ because the left-hand side can be bounded by the number of ordered set partitions. Also notice that $\Cov{\ind_{\Gamma_k \le \RT(t_i)}}{\ind_{\Gamma_k \le \RT(t_j)}} \le \min\set{\Var{\ind_{\Gamma_k \le \RT(t_i)}},\Var{\ind_{\Gamma_k \le \RT(t_j)}}}$ so that
	\begin{equation*}
		\Var[\bigg]{\sum_{i=1}^n \frac{a_i}{\sqrt{\Sigma_R}} \bigl(\ind_{\Gamma_k\le \RT(t_i)} - \lambda_{k,t_i}\bigr)}
		\le n\sum_{i=1}^n \Var[\Big]{ \frac{a_i}{\sqrt{\Sigma_R}}\bigl( \ind_{\Gamma_k \le \RT(t_i)} - \lambda_{k,t_i} \bigr) }.
	\end{equation*}
	Therefore, we conclude
	\begin{equation*}
		\limsup_{R\to\infty} \abs[\Big]{\varkappa^{(q)}\bigl(a_1\widehat{\Xi}_R(t_1) + \dotsb + a_n\widehat{\Xi}_R(t_n)\bigr)}
		\le \limsup_{R\to\infty} 2^q q! \biggl(\sum_{i=1}^n \frac{\abs{a_i}}{\sqrt{\Sigma_R}}\biggr)^{q-2} n\sum_{i=1}^n a_i^2 \Var[\big]{\widehat{\Xi}_R(t_i)}
		= 0.
	\end{equation*}
	This proves the convergence in finite-dimensional distributions.
\end{proof}
The convergence in finite-dimensional distributions can be upgraded to convergence in the Skorokhod topology under a mild continuity condition on the expectation $t\mapsto \RT(t)$.
\begin{proposition}\label{thm:FCLTSkorokhod}
	Under the assumptions of \cref{lem:FCLTFDD} and further assuming that there exist constants $\epsilon>0$, $C>0$ and $R_0\in\bR_+$ such that for all $R\ge R_0$ and $s,t\in J$ with $s\le t \le s+1$,
	\begin{equation}\label{eq:MeanCondition}
		\RT(t) - \RT(s) \le C \Sigma_R(t-s)^{\frac{1}{2}+\epsilon},
	\end{equation}
	then, as $R\to\infty$, the following convergence in the Skorokhod topology holds:
	\begin{equation*}
		\bigl(\widehat{\Xi}_R(t)\bigr)_{t\in J} \xrightarrow{\cS(J)} \bigl(G_t\bigr)_{t\in J}.
	\end{equation*}
\end{proposition}
\begin{proof}
	By \cref{lem:FCLTFDD}, $(\widehat{\Xi}_R(t))_{t\in J}$ converges in finite-dimensional distributions to $(G(t))_{t\in J}$. Thus, it remains to show tightness. We prove this by using the following Kolmogorov criterion \cite[Theorem VI.4.1]{JS03StochasticProcesses}: Fix $t_0\in J$.
	\begin{enumerate}
		\item\label{itm:Tight1} The collection of random variables $(\widehat{\Xi}_R(t_0))_{R\in\bR_+}$ is tight.
		\item\label{itm:Tight2} It holds $\lim\limits_{t\to t_0} \limsup\limits_{R\to\infty} \Pr[\big]{\abs[\big]{\widehat{\Xi}_R(t) - \widehat{\Xi}_R(t_0)} > \epsilon} = 0$ for all $\epsilon>0$.
		\item\label{itm:Tight3} There exist constants $C,R_0\in\bR_+$ and $\beta>1$ such that
	\begin{equation*}
		\Ex[\Big]{\abs[\big]{\widehat{\Xi}_R(t)-\widehat{\Xi}_R(s)}^2 \abs[\big]{\widehat{\Xi}_R(u)-\widehat{\Xi}_R(t)}^2}
		\le C(u-s)^\beta
	\end{equation*}
	for all $s\le t\le u$, $s,t,u\in J$ and all $R\ge R_0$.
	\end{enumerate}

	Concerning \ref{itm:Tight1}, the claim follows immediately from the convergence in distribution of $\widehat{\Xi}_R(t_0)$ for $R\to\infty$. For \ref{itm:Tight2} notice that, for $R>0$, the condition \eqref{eq:MeanCondition} on the means implies
	\begin{align*}
		\Var[\big]{\widehat{\Xi}_R(t) - \widehat{\Xi}_R(t_0)}
		&= \frac{1}{\Sigma_R} \sum_{k=1}^N (\lambda_{k,t} - \lambda_{k,t_0})\bigl(1- (\lambda_{k,t} - \lambda_{k,t_0})\bigl)\\
		&\le \frac{1}{\Sigma_R} \sum_{k=1}^N (\lambda_{k,t} - \lambda_{k,t_0})
		= \frac{\RT(t)-\RT(t_0)}{\Sigma_R}\\
		&\le C(t-t_0)^{\frac{1}{2}+\epsilon}.
	\end{align*}
	Next, apply Chebyshev's inequality and the above variance bound to obtain
	\begin{equation*}
		\limsup_{R\to\infty} \Pr[\big]{\abs[\big]{\widehat{\Xi}_R(t) - \widehat{\Xi}_R(t_0)} > \delta}
		\le \limsup_{R\to\infty} \frac{\Var[\big]{\widehat{\Xi}_R(t) - \widehat{\Xi}_R(t_0)}}{\delta^2}
		\le \frac{C(t-t_0)^{\frac{1}{2}+\epsilon}}{\delta^2}.
	\end{equation*}
	The last term clearly converges to zero as $t\to t_0$.

	Lastly, we show that \ref{itm:Tight3} holds. Notice that
	\begin{align*}
		\MoveEqLeft \Ex[\Big]{\abs[\big]{\widehat{\Xi}_R(t)-\widehat{\Xi}_R(s)}^2 \abs[\big]{\widehat{\Xi}_R(u)-\widehat{\Xi}_R(t)}^2}\\
		&= \Var[\big]{\widehat{\Xi}_R(t)-\widehat{\Xi}_R(s)}\Var[\big]{\widehat{\Xi}_R(u)-\widehat{\Xi}_R(t)} + 2\Cov[\big]{\widehat{\Xi}_R(t)-\widehat{\Xi}_R(s)}{\widehat{\Xi}_R(u)-\widehat{\Xi}_R(t)}^2\\
		&\qquad + \varkappa\bigl(\widehat{\Xi}_R(t)-\widehat{\Xi}_R(s),\widehat{\Xi}_R(t)-\widehat{\Xi}_R(s),\widehat{\Xi}_R(u)-\widehat{\Xi}_R(t),\widehat{\Xi}_R(u)-\widehat{\Xi}_R(t)\bigr),
	\end{align*}
	where the last term denotes the joint cumulant (see \cite[Proposition 6.16]{JLR00RandomGraphs}). By the above variance bound, the first two summands are bounded by
	\begin{align*}
		\Var[\big]{\widehat{\Xi}_R(t)-\widehat{\Xi}_R(s)}\Var[\big]{\widehat{\Xi}_R(u)-\widehat{\Xi}_R(t)}
		&\le C^2(t-s)^{\frac{1}{2}+\epsilon} (u-t)^{\frac{1}{2}+\epsilon}
		\shortintertext{and}
		2\Cov[\big]{\widehat{\Xi}_R(t)-\widehat{\Xi}_R(s)}{\widehat{\Xi}_R(u)-\widehat{\Xi}_R(t)}^2
		&\le 2C^2(t-s)^{\frac{1}{2}+\epsilon} (u-t)^{\frac{1}{2}+\epsilon}.
	\end{align*}
	As $(u-t)^{\frac{1}{2}+\epsilon} (t-s)^{\frac{1}{2}+\epsilon} \le \frac{(u-s)^{1+2\epsilon}}{4}$, it only remains to bound the joint cumulant term:
	\begin{align*}
		\MoveEqLeft \varkappa \bigl(\widehat{\Xi}_R(t)-\widehat{\Xi}_R(s),\widehat{\Xi}_R(t)-\widehat{\Xi}_R(s),\widehat{\Xi}_R(u)-\widehat{\Xi}_R(t),\widehat{\Xi}_R(u)-\widehat{\Xi}_R(t)\bigr)\\
		&= \frac{1}{\Sigma_R^2} \sum_{k=1}^N \varkappa\bigl(\ind_{\RT(s) < \Gamma_k \le \RT(t)},\ind_{\RT(s) < \Gamma_k \le \RT(t)},\ind_{\RT(t) < \Gamma_k \le \RT(u)},\ind_{\RT(t) < \Gamma_k \le \RT(u)}\bigr)\\
		&= \frac{1}{\Sigma_R^2} \sum_{k=1}^N \Bigl( -6(\lambda_{k,t} - \lambda_{k,s})^2(\lambda_{k,u} - \lambda_{k,t})^2 + 2(\lambda_{k,t} - \lambda_{k,s})^2(\lambda_{k,u} - \lambda_{k,t})\\
		&\qquad\qquad\qquad + 2(\lambda_{k,t} - \lambda_{k,s})(\lambda_{k,u} - \lambda_{k,t})^2 - (\lambda_{k,t} - \lambda_{k,s})(\lambda_{k,u} - \lambda_{k,t}) \Bigr).
	\end{align*}
	Let us distinguish two cases.

	\emph{Case 1: $\frac{1}{2C\Sigma_R}\le (u-s)^{\frac{1}{2}+\epsilon}$.} Then, the above bound on the variance yields
	\begin{align*}
		\MoveEqLeft \varkappa \bigl(\widehat{\Xi}_R(t)-\widehat{\Xi}_R(s),\widehat{\Xi}_R(t)-\widehat{\Xi}_R(s),\widehat{\Xi}_R(u)-\widehat{\Xi}_R(t),\widehat{\Xi}_R(u)-\widehat{\Xi}_R(t)\bigr)\\
		&\le \frac{2C(u-s)^{\frac{1}{2}+\epsilon}}{\Sigma_R} \sum_{k=1}^N (\lambda_{k,t} - \lambda_{k,s})(\lambda_{k,u} - \lambda_{k,t})\\
		&= -2C(u-s)^{\frac{1}{2}+\epsilon} \Cov[\big]{\widehat{\Xi}_R(t)-\widehat{\Xi}_R(s)}{\widehat{\Xi}_R(u)-\widehat{\Xi}_R(t)}\\
		&\le 2C(u-s)^{\frac{1}{2}+\epsilon} \sqrt{\Var[\big]{\widehat{\Xi}_R(t)-\widehat{\Xi}_R(s)}\Var[\big]{\widehat{\Xi}_R(u)-\widehat{\Xi}_R(t)}}\\
		&\le C^2(u-s)^{1+2\epsilon}.
	\end{align*}
	This proves the claim in the first case.

	\emph{Case 2: $(u-s)^{\frac{1}{2}+\epsilon}<\frac{1}{2C\Sigma_R}$.} Then,
	\begin{align*}
		\lambda_{k,u} - \lambda_{k,s}
		\le \sum_{k=1}^N \bigl(\lambda_{k,u} - \lambda_{k,s}\bigr)
		= \RT(u) - \RT(s)
		\le C\Sigma_R (u-s)^{\frac{1}{2}+\epsilon}
		\le \frac{1}{2}
	\end{align*}
	which in particular implies that
	\begin{align*}
		\MoveEqLeft -6(\lambda_{k,t} - \lambda_{k,s})^2(\lambda_{k,u} - \lambda_{k,t})^2 + 2(\lambda_{k,t} - \lambda_{k,s})^2(\lambda_{k,u} - \lambda_{k,t})\\
		&\qquad\qquad + 2(\lambda_{k,t} - \lambda_{k,s})(\lambda_{k,u} - \lambda_{k,t})^2 - (\lambda_{k,t} - \lambda_{k,s})(\lambda_{k,u} - \lambda_{k,t})\\
		&\le -(\lambda_{k,t} - \lambda_{k,s})(\lambda_{k,u} - \lambda_{k,t}) \bigl(1-2(\lambda_{k,t} - \lambda_{k,s})\bigr) \bigl(1-2(\lambda_{k,u} - \lambda_{k,t})\bigr)
		\le 0.
	\end{align*}
	Hence, $\varkappa \bigl(\widehat{\Xi}_R(t)-\widehat{\Xi}_R(s),\widehat{\Xi}_R(t)-\widehat{\Xi}_R(s),\widehat{\Xi}_R(u)-\widehat{\Xi}_R(t),\widehat{\Xi}_R(u)-\widehat{\Xi}_R(t)\bigr) \le 0$ in this case and the claim follows.
\end{proof}
\begin{lemma}\label{lem:CovarianceAsymptotic}
	Under \cref{NA}, it holds for any $s,t\in J$ with $s<t$ as $R\to\infty$ that
	\begin{equation*}
		\Cov[\big]{\Xi_{\RT(s)}^{(N)}}{\Xi_{\RT(t)}^{(N)}}
		= \int_0^N \Phi\Bigl(\frac{x-\vartheta\RT(s)}{\Sigma_{\RT(s)}}\Bigr) \biggl( 1 - \Phi\Bigl(\frac{x-\vartheta\RT(t)}{\Sigma_{\RT(t)}}\Bigr) \biggr) \dif x + O(1).
	\end{equation*}
\end{lemma}
\begin{proof}
	In this proof we treat the case of infinite and finite point processes simultaneously as presented in \eqref{def:XiN}. Recall that in the infinite case, we have $N=\infty$. Let $s,t\in J$ with $s\le t$. By \eqref{cond} we obtain
	\begin{align*}
		\MoveEqLeft[1] \Cov[\big]{\Xi_{\RT(s)}}{\Xi_{\RT(t)}}
		= \sum_{k=1}^N \Pr{\Gamma_k \le \RT(s)} \Pr{\Gamma_k > \RT(t)}\\
		&= \sum_{k=1}^N \Bigl(\Phi\bigl(\tfrac{k-\vartheta\RT(s)}{\Sigma_{\RT(s)}}\bigr) + \tfrac{1}{\Sigma_{\RT(s)}} \Psi\bigl(\tfrac{k-\vartheta\RT(s)}{\Sigma_{\RT(s)}}\bigr) + O\bigl( \theta_{R,k}\bigr)\Bigr)
		\Bigl(1- \Phi\bigl(\tfrac{k-\vartheta\RT(t)}{\Sigma_{\RT(t)}}\bigr) - \tfrac{1}{\Sigma_{\RT(t)}} \Psi\bigl(\tfrac{k-\vartheta\RT(t)}{\Sigma_{\RT(t)}}\bigr) + O\bigl( \theta_{R,k}\bigr)\Bigr).
	\end{align*}
	Since $\Phi$ and $\Psi$ are bounded and the error terms $\theta_{R,k}$ satisfy $\sum_{k=1}^N \theta_{R,k} \to 0$ as $R\to\infty$, we obtain
	\begin{equation}\label{eq:CovExpansion1}
	\begin{split}
		\Cov[\big]{\Xi_{\RT(s)}}{\Xi_{\RT(t)}}
		&= \sum_{k=1}^N \biggl(\Phi\Bigl(\frac{k-\vartheta\RT(s)}{\Sigma_{\RT(s)}}\Bigr) + \frac{1}{\Sigma_{\RT(s)}} \Psi\Bigl(\frac{k-\vartheta\RT(s)}{\Sigma_{\RT(s)}}\Bigr)\biggr)\\
		&\qquad\qquad \biggl(1- \Phi\Bigl(\frac{k-\vartheta\RT(t)}{\Sigma_{\RT(t)}}\Bigr) - \frac{1}{\Sigma_{\RT(t)}} \Psi\Bigl(\frac{k-\vartheta\RT(t)}{\Sigma_{\RT(t)}}\Bigr)\biggr) + o(1).
	\end{split}
	\end{equation}
	Recall that, by \cref{NA}, $\Psi\in W^{1,1}(\bR)$. Hence, we can apply the Euler-Maclaurin formula \eqref{eq:EulerMaclaurinReduced} such that it holds
	\begin{equation*}
		\sum_{k=1}^N \frac{1}{\Sigma_{\RT(t)}} \Psi\Bigl(\frac{k-\vartheta\RT(t)}{\Sigma_{\RT(t)}}\Bigr)
		= \int_0^N \frac{1}{\Sigma_{\RT(t)}} \Psi\Bigl(\frac{x-\vartheta\RT(t)}{\Sigma_{\RT(t)}}\Bigr) \dif x + O\Bigl(\frac{1}{\Sigma_{\RT(t)}}\Bigr).
	\end{equation*}
	By a change of variables and using that $\Psi\in L^1(\bR)$, this shows that
	\begin{equation*}
		\sum_{k=1}^N \frac{1}{\Sigma_{\RT(t)}} \Psi\Bigl(\frac{k-\vartheta\RT(t)}{\Sigma_{\RT(t)}}\Bigr)
		\le \norm{\Psi}_{L^1} + O\Bigl(\frac{1}{\Sigma_{\RT(t)}}\Bigr)
	\end{equation*}
	is uniformly bounded. Combining this with \eqref{eq:CovExpansion1} and using that $\Phi$ and $1-\Phi$ are bounded we obtain
	\begin{equation*}
		\Cov[\big]{\Xi_{\RT(s)}}{\Xi_{\RT(t)}}
		= \sum_{k=1}^N \Phi\Bigl(\frac{k-\vartheta\RT(s)}{\Sigma_{\RT(s)}}\Bigr) \biggl(1- \Phi\Bigl(\frac{k-\vartheta\RT(t)}{\Sigma_{\RT(t)}}\Bigr)\biggr) + O(1).
	\end{equation*}
	Since $\Phi$ is absolutely continuous, applying again \eqref{eq:EulerMaclaurinReduced} implies that
	\begin{equation*}
		\Cov[\big]{\Xi_{\RT(s)}}{\Xi_{\RT(t)}}
		= \int_0^N \Phi\Bigl(\frac{x-\vartheta\RT(s)}{\Sigma_{\RT(s)}}\Bigr) \biggl( 1 - \Phi\Bigl(\frac{x-\vartheta\RT(t)}{\Sigma_{\RT(t)}}\Bigr) \biggr) \dif x + O(1).\qedhere
	\end{equation*}
\end{proof}
\begin{remark}
	When working under the more general assumptions from \cref{rk:BV}, we do use a direct Riemann sum approximation instead of the Euler-Maclaurin formula for $\Psi_M$ which leads to the same results.
\end{remark}

We now combine \cref{lem:FCLTFDD} and \cref{thm:FCLTSkorokhod} with \cref{lem:CovarianceAsymptotic} to conclude \cref{thm:FCLTPlanar}
\begin{proof}[New proof of {\cref{thm:FCLTPlanar}}]
	Recall that in \cref{thm:FCLTPlanar} we have $\vartheta=1$ and $\Sigma_R = o(R)$. We treat the two cases $\RT(t) = tR$ and $\RT(t) = R+t\Sigma_R$ separately.

	Let us consider first the macroscopic scaling $\RT(t)=tR$ in which we choose $J=\bR_+$. By a change of variables, \cref{lem:CovarianceAsymptotic} implies that
	\begin{equation}\label{eq:CovarianceAsymptotic2}
		\Cov[\Big]{\frac{\Xi_{sR}}{\sqrt{\Sigma_R}}}{\frac{\Xi_{tR}}{\sqrt{\Sigma_R}}}
		= \frac{\Sigma_{tR}}{\Sigma_R} \int_\bR \ind_{\intcc[\big]{-\frac{tR}{\Sigma_{tR}},\frac{N-tR}{\Sigma_{tR}}}}(x) \Phi\Bigl(\frac{\Sigma_{tR}x+tR-sR}{\Sigma_{sR}}\Bigr) \Bigl( 1 - \Phi(x) \Bigr) \dif x + o(1).
	\end{equation}
	Notice that since $R\mapsto\Sigma_R$ is non-decreasing and $\Phi$ is non-increasing, we can bound the integrand by $\Phi(x)\ind_{x\ge 0} + (1-\Phi(x))\ind_{x<0}$. Recall that $\Phi(x) = \Pr{Z>x}$ with $\Ex{\abs{Z}}<\infty$. As in \eqref{eq:BoundExpectationZ} this yields that $x\mapsto \Phi(x)\ind_{x\ge 0} + (1-\Phi(x))\ind_{x<0}$ is integrable, and we can apply a dominated convergence theorem in \eqref{eq:CovarianceAsymptotic2}. This shows
	\begin{equation*}
		\lim_{R\to\infty} \Cov[\Big]{\frac{\Xi_{sR}}{\sqrt{\Sigma_R}}}{\frac{\Xi_{tR}}{\sqrt{\Sigma_R}}}
		= g_t^2 \int_{-\infty}^{a^+(t)} \Phi(x) \bigl( 1 - \Phi(x) \bigr)\ind_{s=t} \dif x + o(1),
	\end{equation*}
	where for $N<\infty$ that $a^+=\lim_{N\to\infty} \frac{N-tR}{\Sigma_{tR}}$ and $a^+=\infty$ for $N=\infty$. By \cref{lem:FCLTFDD} this implies the convergence in finite dimensional distributions.

	Consider now the microscopic scaling $\RT(t) = R+t\Sigma_R$. We choose $J=\bR$ in \eqref{Xihat}. By a change of variables, \cref{lem:CovarianceAsymptotic} implies that
	\begin{equation}\label{eq:CovarianceAsymptotic3}
	\begin{split}
		\MoveEqLeft \Cov[\Big]{\frac{\Xi_{R+s\Sigma_R}}{\sqrt{\Sigma_R}}}{\frac{\Xi_{R+t\Sigma_R}}{\sqrt{\Sigma_R}}}\\
		&= \frac{\Sigma_{R+t\Sigma_R}}{\Sigma_R} \int_\bR \ind_{\intcc[\big]{-\frac{R+t\Sigma_R}{\Sigma_{R+t\Sigma_R}},\frac{N-R-t\Sigma_R}{\Sigma_{R+t\Sigma_R}}}}(x) \Phi\Bigl(\frac{\Sigma_{R+t\Sigma_R}x+t\Sigma_R-s\Sigma_R}{\Sigma_{R+s\Sigma_R}}\Bigr) \Bigl( 1 - \Phi(x) \Bigr) \dif x + o(1).
	\end{split}
	\end{equation}
	Using the same bound as in case $\RT(t)=tR$ shows that we can apply the dominated convergence theorem in \eqref{eq:CovarianceAsymptotic3}. Observe that for any $\epsilon>0$, it holds for $R$ sufficiently large,
	\begin{equation*}
		\frac{\Sigma_{R(1-\epsilon)}}{\Sigma_R} < \frac{\Sigma_{R+t\Sigma_R}}{\Sigma_R} < \frac{\Sigma_{R(1+\epsilon)}}{\Sigma_R}.
	\end{equation*}
	Since $\frac{\Sigma_{R(1\pm\epsilon)}}{\Sigma_R} \to g(1\pm \epsilon)$ as $R\to \infty$ and $g$ is continuous at $1$, this shows that $\frac{\Sigma_{R+t\Sigma_R}}{\Sigma_R} \to 1$. Combining this with \eqref{eq:CovarianceAsymptotic3} yields
	\begin{equation*}
		\lim_{R\to\infty} \Cov[\bigg]{\frac{\Xi_{R+s\Sigma_R}}{\sqrt{\Sigma_R}}}{\frac{\Xi_{R+t\Sigma_R}}{\sqrt{\Sigma_R}}}
		= \int_{I-t} \Phi(x+t-s) \bigl(1-\Phi(x)\bigr) \dif x
		= \int_I \Phi(x-s) \bigl(1-\Phi(x-t)\bigr) \dif x,
	\end{equation*}
	where $I$ is defined according to \eqref{interval}.
	The condition on the mean \eqref{eq:MeanCondition} is clearly satisfied for our choice of $\RT(t)=R+t\Sigma_R$ with $\epsilon=\frac{1}{2}$. Thus, we conclude the claimed convergence in the Skorokhod topology from \cref{thm:FCLTSkorokhod}.
\end{proof}

We now combine \cref{lem:FCLTFDD} with \cref{lem:CovarianceAsymptotic} to conclude \cref{thm:FCLTHyperbolic}.
\begin{proof}[Proof of {\cref{thm:FCLTHyperbolic}}]
	In this case $\vartheta = 0$, and we choose $\RT(t)= tR$, $t_0 =1$ and $J=\bR_+$ in \eqref{Xihat}. Then, for all $0<s\le t$, by a change of variables, \cref{lem:CovarianceAsymptotic} implies that
	\begin{equation*}
		\Cov[\Big]{\frac{\Xi_{sR}}{\sqrt{\Sigma_R}}}{\frac{\Xi_{tR}}{\sqrt{\Sigma_R}}}
		= \frac{\Sigma_{tR}}{\Sigma_R} \int_0^\frac{N}{\Sigma_{tR}} \Phi\Bigl(\frac{\Sigma_{tR}x}{\Sigma_{sR}}\Bigr) \Bigl( 1 - \Phi(x) \Bigr) \dif x + o(1).
	\end{equation*}
	Recall that in the hyperbolic setting $\Sigma_R \sim R$. The same bound as given for \eqref{eq:CovarianceAsymptotic2} shows that we can apply the dominated convergence theorem to get
	\begin{equation*}
		\lim_{R\to\infty} \Cov[\Big]{\frac{\Xi_{sR}}{\sqrt{\Sigma_R}}}{\frac{\Xi_{tR}}{\sqrt{\Sigma_R}}}
		= t \int_{I/t} \Phi\Bigl(\frac{tx}{s}\Bigr) \Bigl(1-\Phi(x)\Bigr) \dif x
		= \int_I \Phi\Bigl(\frac{x}{s}\Bigr) \Bigl(1-\Phi\Bigl(\frac{x}{t}\Bigr)\Bigr) \dif t.
	\end{equation*}
	The condition on the mean \eqref{eq:MeanCondition} is clearly satisfied for $\RT(t) = tR$ and $\Sigma_R \sim R$ with $\epsilon=\frac{1}{2}$. Thus, the claimed convergence in the Skorokhod topology follows from \cref{thm:FCLTSkorokhod}.
\end{proof}


\section{Applications to different models} \label{sec:proofA}

In this section, we show that the models presented in \cref{sec:Applications} satisfy \cref{NA}, so that we can apply the results of \cref{sec:MainResults}.


\subsection{Proofs of the results in \texorpdfstring{\cref{sec:Introduction} and \cref{sec:FiniteGinibre}}{Section \ref*{sec:Introduction} and Section \ref*{sec:FiniteGinibre}}}
\label{sec:GinibreProof}

Recall for the Ginibre ensemble that $T(r) = \Ex{\xi(D_r)} = r^2$ and that the random variables $\Gamma_k$ from \eqref{def:Gamma} with $\alpha=0$ are gamma-distributed with shape~$k$ and rate~$1$. The proof relies on \cref{lem:A} and it is split in two parts: First we provide elementary tail-bounds (\cref{lem:GinibreBounds}) for the random variables $\Gamma_k$. Then, using the identity $\Gamma_k \overset{\mathrm{d}}{=} \sum_{n=1}^k Z_n$ with independent exponentially-distributed random variables $Z_n$, we deduce the asymptotics \eqref{cond} from an Edgeworth expansion.

\begin{lemma}\label{lem:GinibreBounds}
	For all $k \ge R$, it holds
	\begin{equation*}
		e^{-R} \frac{R^{k}}{k!}
		\le \Pr[\big]{\Gamma_k \le R}
		\le e^{k-R} \Bigl(\frac{R}{k}\Bigr)^k
		\le 3\sqrt{ k} e^{-R} \frac{R^k}{k!}.
	\end{equation*}
	Similarly, for all $1\le k\le R$ it holds
	\begin{equation*}
		e^{-R} \frac{R^{k-1}}{(k-1)!}
		\le \Pr[\big]{\Gamma_k > R}
		\le e^{k-R} \Bigl(\frac{R}{k}\Bigr)^k
		\le 3\sqrt{k} e^{-R} \frac{R^k}{k!}.
	\end{equation*}
\end{lemma}
\begin{proof}
	We only prove the result for $k\ge R$. The bounds for $k\in[1,R]$ follow in a similar way. For the lower bound we use the following relation between the Poisson and the gamma distribution: For any $k\in\bN$ and any $x\in\bR_+$, by integration by parts it holds
	\begin{equation}\label{eq:GammaPoisson}
		\Pr{\Gamma_k \le x}
		= \int_0^x \frac{y^{k-1}}{(k-1)!}e^{-y} \dif y
		= 1-e^{-x} \sum_{\ell=0}^{k-1} \frac{x^\ell}{\ell!}
		= \Pr[\big]{\Poissondist(x) \ge k}.
	\end{equation}
	Hence, we have for any $k\in \bN$,
	\begin{equation*}
		\Pr[\big]{\Gamma_k \le R} = \Pr[\big]{\Poissondist(R) \ge k} \ge \Pr[\big]{\Poissondist(R) = k} = e^{-R} \frac{R^{k}}{k!}.
	\end{equation*}
	For the upper bound, if $k>R$, by Markov's inequality with $x=\frac{k}{R}-1 >0$,
	\begin{equation*}
		\Pr[\big]{\Gamma_k \le R} \le \Ex[\big]{e^{-x\Gamma_k}} e^{xR} = e^{xR} (1+x)^{-k} = e^{k-R} \Bigl(\frac{k}{R}\Bigr)^{-k}.
	\end{equation*}
	By continuity of $\Pr{\Gamma_k \le R}$ for $k>0$, this bound holds true when $k=R$ as well.
For the last step, we use the bound $k! \le 3\sqrt{k} k^k e^{-k}$.
\end{proof}

We now discuss the Edgeworth expansion. Consider a sequence $(X_n)_{n\in\bN}$ of real-valued random variables together with a reference random variables $X$ and denote by $\varkappa_n^{(k)}$ and $\varkappa^{(k)}$ the respective $k$-th cumulant.
Informally, if $\varkappa_n^{(k)}$ is close to $\varkappa^{(k)}$, we have the expansion
\begin{equation*}
	\Ex[\big]{e^{\i zX_n}} = \Ex[\big]{e^{\i zX}} \exp\biggl(\sum_{k=1}^\infty \bigl(\varkappa_n^{(k)} - \varkappa^{(k)}\bigr) \frac{(\i z)^k}{k!}\biggr)
\end{equation*}
which, by Fourier inversion, implies that the distribution functions satisfy
\begin{equation*}
	F_{X_n}(x) = F_X(x) + \sum_{k=1}^\infty B_k\bigl(\varkappa_n^{(1)}-\varkappa^{(1)}, \dotsc, \varkappa_n^{(k)}-\varkappa^{(k)}\bigr) \frac{(-1)^k}{k!}F_X^{(k)}(x) ,
\end{equation*}
where $B_k$ denotes the $k$-th Bell polynomial.
This shows that the distribution function $F_{X_n}$ can be approximated by $F_X$ provided that the series can be treated as an error.
In particular, if $X_n$ is a sum of $n$ i.i.d.~well-behaved random variables, then $X$ is chosen to be Gaussian (according to the central limit theorem) and the expansion can be controlled in terms of negative powers of $\sqrt{n}$. This is made rigorous by the so-called Berry-Esseen expansion, see \cite[Theorem XVI.4.1]{F71ProbabilityTheory}: It holds uniformly for all $R>0$ as $k\to\infty$,
\begin{equation*}
	\Pr[\big]{\Gamma_k \le R}
	= \Pr[\Big]{\frac{\sum_{n=1}^k Z_n - k}{\sqrt{k}} \le \frac{R-k}{\sqrt{k}}}
	= \Phi_0\Bigl(\frac{k-R}{\sqrt{k}}\Bigr) + \frac{1}{3\sqrt{k}} \Phi_0'''\Bigl(\frac{k-R}{\sqrt{k}}\Bigr) + O\Bigl(\frac{1}{k}\Bigr).
\end{equation*}
Together with the bounds from \cref{lem:GinibreBounds}, this implies the assumptions from \cref{lem:A} with $\Phi =\Phi_0$ (the error function) and $\Upsilon = \Phi_0'''$ for any fixed $\epsilon< \frac{1}{10}$.
Indeed, $\Phi_0$ satisfies the required smoothness and integrability conditions and letting
$\eta_{k,R} = 3\sqrt{k} e^{-R} \frac{R^k}{k!} \ind_{|k-R| > R^{\frac{1}{2}+\epsilon}} + O(R^{-1}) \ind_{|k-R| \le R^{\frac{1}{2}+\epsilon}} $, we have
\[
{\textstyle \sum_{k=1}^\infty \eta_{R,k} } = O\bigl( R^{\epsilon-\frac{1}{2}}\bigr) \qquad \text{as }R\to\infty.
\]
Note that in principle, this Edgeworth expansion would allow us to compute higher order terms in \eqref{Eexp} and thus to obtain further correction terms in the deviation probabilities from \cref{thm:ModPhiConclusions}.
To summarize, by \cref{lem:A}, we obtain the following result.

\begin{proposition} 
 For any $0<\epsilon<1/10$, the radii $(\Gamma_k)_{k\in\bN}$ of the Ginibre ensemble satisfy \cref{NA} with $\vartheta=1$, speed $\Sigma_R = \sqrt{R}$,
	\begin{equation*}
		\Phi_0(x) = \frac{1}{\sqrt{2\pi }} \int_x^\infty e^{-\frac{y^2}{2}} \dif y
		\qquad\text{and}\qquad
		\Psi_0(x) = \frac{2-5x^2}{6\sqrt{2\pi}}e^{-\frac{x^2}{2}}.
	\end{equation*}
\end{proposition}

Hence, by our main results, we conclude that the counting statistics of both the finite and infinite Ginibre ensemble converges in the mod-phi sense of \cref{thm:ModPhiGeneral}. In particular, the precise deviations from \cref{thm:GinibrePreciseAsymptotics} follow from \cref{thm:ModPhiConclusions} and the functional central limit theorems presented in \cref{thm:GinibreFCLTFDD} and \cref{thm:GinibreFCLTSkorokhod} follow directly from \cref{thm:FCLTPlanar}.

As for the results of \cref{sec:FiniteGinibre}, the only difference comes from the finite-size effects occurring at the edge.
To obtain \cref{thm:FiniteGinibreModPhi}, \cref{thm:FiniteGinibreFCLTFDD} and \cref{thm:FiniteGinibreFCLTSkorokhod}, it is simply worth observing that according to \eqref{Ncond}--\eqref{interval} and \eqref{var}--\eqref{eq:VarXi}, we have
$\Var{\Xi_R^{(N)}}= \sqrt{R}\Lambda''(0) + O(1)$ as $R\to\infty$ where
\begin{equation*}
\Lambda''(0) = \begin{cases}
{\displaystyle \int_\bR\Phi_0 (x)(1-\Phi_0 (x)) \dif x = \frac{1}{\sqrt{\pi}}}
&i)\ \text{if } N=\infty \text{ or } N<\infty \text{ and } R(N) = \gamma N \text{ with }\gamma\in(0,1),\\
{\displaystyle \int_{-\infty}^{a^+} \Phi_0 (x)(1-\Phi_0 (x)) \dif x }
&ii)\ \text{if } N<\infty \text{ and } R(N) = N\bigl(1-\frac{a^+}{\sqrt{N}}+\frac{(a^+)^2}{2N}\bigr) \text{ with } a^+ \in\bR, \\
0
&iii)\ \text{if } N<\infty \text{ and } (R(N) - N) \gg \sqrt{N \log N} \text{ as } N\to\infty.
\end{cases}
\end{equation*}

Note that in case $ii)$, at the edge, the condition on $ R(N) = \gamma N$ comes from solving for $\gamma>0$ the equation \eqref{Ncond} which reads
\begin{equation*}
		\frac{N-R(N)}{\Sigma_{R(N)}}
		= (\gamma^{-1/2}-\gamma^{1/2}) \sqrt{N}
		= \bigl( (1-\gamma) + \tfrac{(1-\gamma )^2}{2} + O((1-\gamma )^3)\bigr) \sqrt{N} = a^+ +O(N^{-1/2}).
\end{equation*}
Hence, the solution has the asymptotic $\gamma = 1-\frac{a^+}{\sqrt{N}}+\frac{(a^+)^2}{2N} + O(N^{-1})$ as $N\to\infty$.
This completes the proofs of our results for counting statistics of both the infinite and finite Ginibre ensembles. In the next section, we turn to show that the Ginibre $\alpha$-type processes satisfy \cref{NA} for general $\alpha\in\bN_0$.

\subsection{Proofs of the results in \texorpdfstring{\cref{sec:Ginibre}}{Section \ref*{sec:Ginibre}}}
\label{sec:GinibreTypeProof}

Recall that for the Ginibre $\alpha$-type ensemble associated with higher-Landau levels, the random variables $(\Gamma_k^{(\LL)})_{k\in\bN}$ which represent the radii of the points have densities for $\LL\in\bN_0$ and $k\in\bN$,
\begin{equation} \label{LLLdensity}
	L_\LL^{(k-\LL-1)}(x)^2 x^{k-\LL-1} e^{-x} \ind_{x\ge0}.
\end{equation}
For $\alpha\ge1$, we cannot apply the Edgeworth expansion described in the previous section since the random variables $\Gamma_k^{(\LL)}$ are not infinitely divisible.
Nevertheless, it follows from \eqref{LLLdensity} that the Laplace transforms of $\Gamma_k^{(\LL)}$ are explicit, and we can use Fourier analysis methods to obtain the expansion \eqref{Eexp}.

Shirai showed in \cite[Proposition 4.1]{S15GinibreType} that for all $z\in\bC$ with $\Re z<1$,
\begin{equation}\label{eq:GinibreTypeMGF}
	\Ex[\big]{e^{z\Gamma_{k-\LL}^{(\LL)}}} = (1-z)^{-k} \sum_{\ell=0}^{\infty} \binom{\LL}{\ell} \binom{k-\LL-1}{\ell} z^{2\ell}.
\end{equation}

Our strategy consists in computing the asymptotics of $\Ex[\big]{e^{\i x\Gamma_{k-\LL}^{(\LL)}}}$ as $k\to\infty$ up to an error that converges to 0 in $L^1(\bR)$ in order to deduce \eqref{Eexp} by using Fourier's inversion formula. This argument is inspired from the proof of the Berry-Esseen expansion for sums of i.i.d.\ random variables, \cite[Chap.~XVI]{F71ProbabilityTheory}.
The next lemma is motivated by the fact that it holds locally uniformly as $k\to\infty$,
\[
\sqrt{k!} L_\LL^{(k)}(k+x\sqrt{k}) \to (-1)^\alpha H_\alpha(x) ,
\]
where $H_\alpha$, $\alpha\in\bN_0$, denote the (orthonormal) Hermite polynomials; see e.g.\ \cite{C78ZerosLaguerre}\footnote{Note that our normalizations are different from that of \cite{C78ZerosLaguerre} and the locally uniform convergence follows directly from the convergence of the zeros of these polynomials.}.
Let us recall that the random variables $Z_\alpha$, $\alpha\in\bN_0$ have densities $h_\alpha^2(x) = H_\alpha^2(x) e^{-\frac{x^2}{2}}/\sqrt{2\pi}$ for $x\in\bR$ and probability tail function $\Phi_\alpha$.

\begin{lemma}\label{lem:GinibreTypeBE}
 It holds uniformly for all $x\in\bR$, as $k\to\infty$,
	\begin{equation} \label{GinibreTypeBE}
		\Pr[\bigg]{\frac{\Gamma_{k-\LL}^{(\LL)}-k}{\sqrt{k}} \le x}
		= \Pr[\big]{Z_\LL \le x} + \frac{1}{\sqrt{3k}}\Phi_\alpha'''(x) + O_\LL\Bigl(\frac{1}{k}\Bigr).
	\end{equation}
\end{lemma}
\begin{proof}
Recall that the classical Laguerre polynomials satisfy
$L_\LL(x) = \sum_{\ell=0}^\LL \binom{\LL}{\ell} \frac{(-x)^\ell}{\ell!}$ and that the Laplace transform of the random variable $Z_\alpha$ is explicitly given by
	\begin{equation} \label{cfHermite}
		\Ex[\big]{e^{z Z_\LL}} = L_\alpha\bigl(-z^2\bigr) e^{\frac{z^2}{2}} , \qquad z\in\bC.
	\end{equation}
This follows from instance from \cite[(4.9)]{S15GinibreType}.
Observe that for any $\alpha\in\bN_0$, it holds uniformly for all $\ell \in\bN_0$,
\[
\binom{k-\LL-1}{\ell} = \frac{k^\ell}{\ell!}\Bigl(1+O_\LL\Bigl(\frac{1}{k}\Bigr)\Bigr).
\]
By \eqref{eq:GinibreTypeMGF}, this implies that it holds uniformly for all $x\in\bR$, as $k\to+\infty$,
\[
\Ex[\big]{e^{\i x\Gamma_{k-\LL}^{(\LL)}}} = (1-\i x)^{-k} L_\LL(\sqrt{k}x^2) \Bigl(1+O_\LL\Bigl(\frac{1}{k}\Bigr)\Bigr).
\]
This shows that the characteristic function of the normalized $\Gamma_{k-\LL}^{(\LL)}$ random variables is given by
\begin{equation} \label{asympcf1}
\Ex[\bigg]{\exp\biggl(\i x\frac{\Gamma_{k-\LL}^{(\LL)}-k}{\sqrt{k}}\biggr)}
= \Bigl(1-\frac{\i x}{\sqrt{k}}\Bigr)^{-k} e^{-\i x\sqrt{k}} L_\alpha(x^2) \Bigl(1+O_\LL\Bigl(\frac{1}{k}\Bigr)\Bigr).
\end{equation}
Let us now compute the asymptotics of the first two factors on the RHS of \eqref{asympcf1}.
Recall that it holds $\abs{e^z - e^w} \le \abs{z-w} e^{\max\set{\Re(z),\Re(w)}}$ for all $z,w\in\bC$.
Using this bound, we obtain
\[
\abs[\bigg]{ \Bigl(1-\frac{\i x}{\sqrt{k}}\Bigr)^{-k} e^{-\i x\sqrt{k}} - e^{- \frac{x^2}{2} - \frac{\i x^3}{3\sqrt{k}}} }
\le k \abs[\bigg]{\log\Bigl(1-\frac{\i x}{\sqrt{k}}\Bigr) + \frac{\i x}{\sqrt{k}} + \frac{1}{2}\Bigl(\frac{\i x}{\sqrt{k}}\Bigr)^2 + \frac{1}{3}\Bigl(\frac{\i x}{\sqrt{k}}\Bigr)^3} e^{-\min\set{k \log|1-\frac{\i x}{\sqrt{k}}| , \frac{x^2}{2}} } ,
\]
where we used the principle branch of $\log(\cdot)$.
Observe that
\[
k \log\abs[\Big]{1-\frac{\i x}{\sqrt{k}}} = \frac{k}{2} \log\Bigl(1+\frac{x^2}{k}\Bigr) \le \frac{x^2}{2}
\]
and $\abs{\log(1-ix) + \sum_{\ell=1}^4 \frac{(ix)^\ell}{\ell}} \le C x^4$ for all $x\in\bR$, so that it holds for all $k>0$ and $x\in\bR$,
\[
\abs[\bigg]{\Bigl(1-\frac{\i x}{\sqrt{k}}\Bigr)^{-k} e^{-\i x\sqrt{k}} - e^{- \frac{x^2}{2} - \frac{\i x^3}{3\sqrt{k}}} }
\le \frac{C x^4}{k} \Bigl(1+\frac{x^2}{k} \Bigr)^{-\frac k2}.
\]
By applying the same bounds, we also have
\[
\abs[\bigg]{ e^{ - \frac{\i x^3}{3\sqrt{k}}} -1 + \frac{\i x^3}{3\sqrt{k}} } \le C \frac{x^6}{k} \sqrt{1+\frac{x^6}{9k} }.
\]
By the triangle inequality, this shows that for all $k>0$ and $x\in\bR$,
\[
\abs[\bigg]{ \Bigl(1-\frac{\i x}{\sqrt{k}}\Bigr)^{-k} e^{-\i x\sqrt{k}} - e^{- \frac{x^2}{2}}\biggl(1 - \frac{\i x^3}{3\sqrt{k}} \biggr) }
\le \frac{C x^4}{k} \Bigl(1+\frac{x^2}{k}\Bigr)^{-\frac k2} + C \frac{x^6}{k} \sqrt{1+\frac{x^6}{9k} } e^{-\frac{x^2}{2}}.
\]
As the function $ k\mapsto (1+\frac{x^2}{k} )^{-\frac k2} $ is non-increasing, we can choose $k= 2\alpha+6$ on the RHS of the previous estimate.
According to \eqref{asympcf1}, this shows that as $k\to+\infty$,
\begin{equation*}
\abs[\bigg]{ \Ex[\bigg]{\exp\biggl(\i x\frac{\Gamma_{k-\LL}^{(\LL)}-k}{\sqrt{k}}\biggr)} - e^{- \frac{x^2}{2}}\biggl(1 - \frac{\i x^3}{3\sqrt{k}} \biggr)
L_\LL(x^2) \biggl(1+O_\LL\Bigl(\frac{1}{k}\Bigr)\biggr) }
\le \frac{Cx^4 | L_\LL(x^2)| }{k} \biggl( (1+x^2 )^{-\alpha-3} + x^2 (1+x^3) e^{-\frac{x^2}{2}} \biggr)
\end{equation*}
where the RHS is in $L^1(\bR)$.
Since the function $x\mapsto \Ex{e^{\i x \Gamma_{k-\LL}^{(\LL)}}} $ is in $L^1(\bR)$, by the Fourier inversion formula, the probability density function of the normalized random variable $ \Gamma_{k-\LL}^{(\LL)}$ is given by
\[
f_{\alpha,k}(x) = \frac{1}{2\pi} \int_{\bR} \Ex[\bigg]{\exp\biggl(\i \xi \frac{ \Gamma_{k-\LL}^{(\LL)}-k}{\sqrt{k}}\biggr)} e^{-\i \xi x} \dif\xi , \qquad x\in\bR.
\]

Then, our previous bound implies that uniformly for all $x\in\bR$,
\[
\bigg| f_{\alpha,k}(x) - \Bigl(1+O_\LL\Bigl(\frac{1}{k}\Bigr)\Bigr) \frac{1}{2\pi} \int_{\bR} e^{- \frac{\xi^2}{2}}\biggl(1 - \frac{\i \xi^3}{3\sqrt{k}} \biggr)L_\LL(\xi^2) e^{-\i \xi x} \dif\xi \bigg| \le \frac{C_\alpha}{k}.
\]
By \eqref{cfHermite}, since $h^2_{\alpha}$ is the probability density function of $Z_\alpha$ and the functions $e^{- \frac{\xi^2}{2}}\bigl(1 + |\xi|^n\bigr)L_\LL(\xi^2)$ are integrable on $\bR$ for any $n\in\bN$, we also have
\[
\diff[n]{}{x}h^2_{\alpha}(x) = \frac{1}{2\pi} \int_{\bR} e^{- \frac{\xi^2}{2}} (-\i \xi)^nL_\LL(\xi^2) e^{-\i \xi x} \dif\xi , \qquad x\in\bR.
\]
Hence, we conclude that
\begin{equation} \label{asympcf2}
\abs[\Big]{ f_{\alpha,k}(x) - h^2_{\alpha}(x) - \frac{1}{3\sqrt{k}}\bigl(h_\LL^2(x)\bigr)''' } \le \frac{C_\alpha}{k} ,
\end{equation}
for a larger constant $C_\alpha$.

Let us now finish the proof. Let $A_k =\set{x\in\bR \given \abs{x} \le \sqrt{2\log k} }$ and recall that
$\int_{\bR\setminus A_k} e^{-\frac{x^2}{2}} \dif x \le \frac{1/k}{\sqrt{(\log k)/2}} $.
Then, if $k$ is sufficiently large (depending on $\alpha$),
\[
 \int_{\bR\setminus A_k} h^2_{\alpha}(x) + \abs[\big]{ \bigl(h_\LL^2(x)\bigr)''' } \dif x \le \frac 1k.
\]
This tail bound with \eqref{asympcf2} implies that for any $x\in\bR$,
\[
\int_{-\infty}^x \abs{ f_{\alpha,k}(t) - h^2_{\alpha}(t) - \frac{1}{3\sqrt{k}}\bigl(h_\LL^2(t)\bigr)''' } \dif t\le \frac{C_\alpha |A_k|+1}{k}.
\]
Since $\Pr[\big]{Z_\LL \le x} = 1-\Phi_\alpha(x) = \int_{-\infty}^x h^2_{\alpha}(t) \dif t$,
this completes the proof with the required uniformity.
\end{proof}

From \cref{lem:GinibreTypeBE}, we can (almost) verify that the random variables
$(\Gamma_k^{(\LL)})_{k\in\bN}$ satisfy the assumptions from \cref{lem:A}.
Indeed, the tail functions $\Phi_\alpha$ satisfy the required smoothness and integrability conditions for any $\alpha\in\bN$, the obstacle being essentially that the errors $O_\alpha(1/k)$ are not summable.
We could actually push the expansion \eqref{GinibreTypeBE} one order further so as to have summable errors but it suffices instead to produce tail-bounds for the random variables $\Gamma_{k}^{(\LL)}$.
Moreover, these bounds will also be crucial in \cref{sec:JLMProof} to prove the JLM asymptotics from \cref{thm:ldp}.

The proof of \cref{lem:GinibreTypeBounds} proceeds by comparing the probability density functions of $\Gamma_{k}^{(\LL)}$ with that of standard gamma random variables.

\begin{lemma}\label{lem:GinibreTypeBounds}
For any $\alpha\in\bN_0$, it holds for all $k\ge R>2\alpha+2$ that
\[ \begin{aligned}
\Pr[\big]{\Gamma_k^{(\LL)} \le R}
\le 3^{2\alpha+1} R^{\alpha+1} \Pr[\big]{\Gamma_{k-2\alpha-2} \le R}.
\end{aligned}\]
and for all $1\le k\le R$ that
\[
\Pr[\big]{\Gamma_k^{(\LL)} \ge R} \le 4^\alpha k^{\alpha} \Pr[\big]{\Gamma_{k+\alpha} \ge R}.
\]
\end{lemma}
\begin{proof}
Recall the expression \eqref{LLLdensity} for the probability density function of the random variable $\Gamma_{k}^{(\LL)}$ where the generalized Laguerre polynomial is given by
\[
	L_\LL^{(k)}(x) = \sqrt{\frac{\alpha!}{\Gamma(k+\alpha+1)}}\sum_{\ell=0}^\alpha \binom{\alpha+k}{\alpha-\ell} \frac{(-x)^\ell}{\ell!} = \sqrt{\frac{\Gamma(k+\alpha+1)}{\alpha!}} \sum_{\ell=0}^\alpha \binom{\alpha}{\ell} \frac{(-x)^\ell}{\Gamma(k+\ell+1)}.
\]
For the first claim, we can use the bound valid for all $|x|\le R$,
\[
	\abs[\big]{ L_\LL^{(k)}(x) }
	\le \sqrt{\Gamma(k+\alpha+1)}	 \sum_{\ell=0}^\alpha \binom{\alpha}{\ell} \frac{R^\ell}{\Gamma(k+1) k^\ell}
	\le \frac{\sqrt{\Gamma(k+\alpha+1)}}{\Gamma(k+1)} \Bigl(1+ \frac Rk \Bigr)^{\alpha}
\]
By \eqref{LLLdensity}, this immediately implies that for $k>2(\alpha+1)$,
\[ \begin{aligned}
\Pr{\Gamma_k^{(\LL)} \le R}
&\le \Bigl(1 + \frac{R}{k-\alpha-1} \Bigr)^{2\alpha} \frac{\Gamma(k)}{\Gamma(k-\alpha)^2}
 \int_0^R x^{k-\LL-1} e^{-x} \dif x
\le 3^{2\alpha} k^\alpha \Pr[\big]{\Gamma_{k-\alpha-1} \le R},
\end{aligned}\]
where we have used that the probability density function of $\Gamma_{k}$ is $x\mapsto\frac{x^{k-1} e^{-x}}{\Gamma(k)}$ on $\bR_+$.
Now, by \cref{lem:GinibreBounds}, it holds for any $\gamma>0$,
\[
\Pr[\big]{\Gamma_{k+\gamma +1} \le R} \le \frac{3R^{\gamma+1}}{\sqrt{k+\gamma+1}} k^{-\gamma} \Pr[\big]{\Gamma_k \le R}.
\]
By combining these bounds, we conclude that
\[
\Pr{\Gamma_k^{(\LL)} \le R} \le 3^{2\alpha+1} R^{\alpha+\frac{1}{2}} \Pr[\big]{\Gamma_{k-2\alpha-2} \le R}.
\]

For the second claim, we can similarly bound for all $x \ge 0$ and $k\ge -\alpha$,
\[
\abs[\big]{ L_\LL^{(k)}(x) } \le \frac{1}{\sqrt{\alpha! \Gamma(k+\alpha+1)}} \sum_{\ell=0}^\alpha \binom{\alpha}{\ell} (k+\alpha)^{\alpha-\ell}x^\ell = \frac{(k+\alpha+x)^\alpha}{\sqrt{\alpha! \Gamma(k+\alpha+1)}}
\]
If $k\ge 1$, by \eqref{LLLdensity}, this shows that for any $1\le k\le R$ that
\[
\Pr{\Gamma_k^{(\LL)} \ge R} \le \frac{1}{\alpha!\Gamma(k)} \int_R^\infty (k-1+x)^{2\alpha} x^{k-\LL-1} e^{-x} \dif x
 \le 2^{2\alpha} k^{\alpha} \int_R^\infty \frac{x^{k+\alpha-1} e^{-x}}{\Gamma(k+\alpha)} \dif x
= 2^{2\alpha} k^{\alpha} \Pr{\Gamma_{k+\alpha} \ge R},
\]
where we used that $\frac{\Gamma(k+\alpha)}{\Gamma(k)} \le \alpha! k^\alpha$.
This completes the proof.
\end{proof}

By combining \cref{lem:A}, \cref{lem:GinibreTypeBE} and \cref{lem:GinibreTypeBounds}, we immediately obtain the following asymptotics.

\begin{proposition}\label{lem:GinibreTypeAssumptionA}
 For any $\alpha\in\bN_0$ and $0<\epsilon<1/10$, the radii $(\Gamma_k^{(\LL)})_{k\in\bN}$ of the Ginibre $\alpha$-type ensemble satisfy \cref{NA} with $\vartheta=1$, speed $\Sigma_R = \sqrt{R}$,
	\begin{equation*}
		\Phi_\LL (x) = \int_x^\infty h_\LL^2(t) \dif t
		\qquad\text{and}\qquad
		\Psi_\LL(x) = \frac{x^2}{2}\Phi_\LL'(x) + \frac{1}{3}\Phi_\LL'''(x) = -\frac{x^2}{2}h_\LL^2(x)-\frac{1}{3}\bigl(h_\LL^2(x)\bigr)''.
	\end{equation*}
\end{proposition}

Hence, according to our main results (\cref{thm:ModPhiGeneral} and \cref{thm:FCLTPlanar}), we have proven the mod-phi convergence from \cref{thm:GinibreTypeModPhi} and the functional central limit theorems from \cref{thm:FCLTPlanar}.
To complete this section, it remains to prove \cref{thm:arealaw} about the asymptotics of entanglement entropies.
The proof is also based on \cref{NA} together with Riemann sum approximations. Before, we a useful lemma relating the tails of $\Phi_\LL$ and $\Psi_\LL$ with the Gaussian tail.
\begin{lemma}\label{lem:TailsPsiPhiPrimePhi}
	For all $\LL\ge0$ and all $\epsilon>0$ there exist constants $c,C>0$ such that for all $\abs{x}\le R^\epsilon$ it holds
	\begin{align*}
		\max\set{\abs{\Psi_\LL(x)},\abs{\Psi'_\LL(x)},\abs{\Phi'_\LL(x)}}
		&\le C R^{(2\LL+2)\epsilon} e^{-\frac{x^2}{2}},\\
		\Phi_\LL(x)\bigl(1-\Phi_\LL(x)\bigr)
		&\ge c R^{-\epsilon} e^{-\frac{x^2}{2}}.
	\end{align*}
\end{lemma}
\begin{proof}
	Notice that all $\Psi_\LL$, $\Psi'_\LL$ and $\Phi_\LL'$ are of the form $x\mapsto p_\LL(x) e^{-\frac{x^2}{2}}$ for a polynomial $p_\LL$ of degree at most $2\LL+2$. Hence, for all $\abs{x} \le R^\epsilon$ it holds
	\begin{equation*}
		\max\set{\abs{\Psi_\LL(x)},\abs{\Psi'_\LL(x)},\abs{\Phi'_\LL(x)}}
		\le C_\LL R^{(2\LL+2)\epsilon} e^{-\frac{x^2}{2}}
	\end{equation*}
	with a constant $C_\LL>0$.
	
	We now prove the lower bound for $\Phi_\LL(1-\Phi_\LL)$. 
	Recall that $\Phi_\LL(x) = \int_x^\infty h_\LL(y)^2 \dif y$. Notice that for any $L>0$ there exists $\delta_L>0$ such that
	\begin{equation*}
		\inf_{\abs{x} \le L} \Phi_\LL(x) \bigl(1-\Phi_\LL(x)\bigr) \ge \delta_L.
	\end{equation*}
	Choose $L$ large enough such that the Hermite polynomial satisfies $H_\LL(x)^2 \ge \sqrt{2\pi}$ for all $\abs{x}\ge L$. The Gaussian tail bound $\frac{\abs{x}}{1+x^2}e^{-\frac{x^2}{2}} \le \int_x^\infty e^{-\frac{y^2}{2}} \dif y$ for $x\in\bR$ together with $\Phi_\LL(0)=\frac{1}{2}$ implies for all $L\le x\le R^\epsilon$ that
	\begin{equation*}
		\Phi_\LL(x)\bigl(1-\Phi_\LL(x)\bigr)
		\ge \frac{1}{2} \int_x^\infty h_\LL(y)^2 \dif y
		\ge \frac{1}{2} \int_x^\infty e^{-\frac{y^2}{2}} \dif y
		\ge \frac{x}{2(1+x^2)}e^{-\frac{x^2}{2}}
		\ge \frac{1}{4R^\epsilon} e^{-\frac{x^2}{2}}.
	\end{equation*}
	By symmetry, the same bound holds for $-R^\epsilon \le x\le -L$. By combining the bounds for $\abs{x}\le L$ and $L\le \abs{x}\le R^\epsilon$, there exists a constant $c_\LL>0$ such that for all $\abs{x}\le R^\epsilon$ it holds
	\begin{equation*}
		\Phi_\LL(x)\bigl(1-\Phi_\LL(x)\bigr)
		\ge c_\LL R^{-\epsilon} e^{-\frac{x^2}{2}}.
	\end{equation*}
\end{proof}

\begin{proof}[Proof of {\cref{thm:arealaw}}]
	Notice that by definition we have
	\begin{equation*}
		S_\LL^\beta(D_r)
		= \sum_{k\in\bN} f_\beta\bigl(\Pr{\Gamma_k^{(\LL)} \le R}\bigr)
	\end{equation*}
	with $r^2=R$ and recall that $f_\beta(x) = \frac{\log(x^\beta+(1-x)^\beta)}{1-\beta}$ if $\beta\ne1$ and $f_1(x) = -x\log(x)-(1-x)\log(1-x)$ for $x\in\intcc{0,1}$. Notice that $f_\beta(0) = f_\beta(1) = 0$. Within this proof we extend $f_\beta$ by continuity as $f_\beta(x)=0$ for all $x\in\bR\setminus\intcc{0,1}$. Let us first show that only terms for $k$ near $R$ contribute significantly to the above sum. It holds for all $\beta>0$ that $f_\beta(x) = f_\beta(1-x)$. Moreover, $f_\beta$ is $\beta\wedge1$-Hölder continuous for $\beta\ne1$ and $1^-$-Hölder continuous for $\beta=1$. In the following, we use the slightly weaker condition that $f_\beta$ is $\Delta_\beta$-Hölder continuous for $\Delta_\beta\in\intoo{\frac{1}{2},\beta\wedge1}$ for all $\beta>0$. In particular, this shows for all $x\in\intcc{0,1}$ that
	\begin{equation*}
		\abs{f_\beta(x)} \le C_\beta \min\set{x^{\Delta_\beta},(1-x)^{\Delta_\beta}}.
	\end{equation*}
	Fix $\epsilon < \min\set{\Delta_\beta-\frac{1}{2},\frac{1}{8\LL+14}}$. By combining \cref{lem:GinibreTypeBounds} with \cref{lem:GinibreBounds} and using that $\log(1-x) \le -x-\frac{x^2}{2}$ for all $x\in\intco{0,1}$, we have for any $\ell>0$ that
	\begin{align*}
		\Pr[\big]{\Gamma_{R+\ell}^{(\LL)} \le R}
		&\le C_\LL R^{\LL+1} \Pr[\big]{\Gamma_{R+\ell-2\LL-2} \le R}\\
		&\le C_\LL (R+\ell)^{2\LL+2} e^{\ell +(R+\ell)\log(1-\frac{\ell}{R+\ell})}\\
		&\le C_\LL (R+\ell)^{2\LL+2} e^{-\frac{\ell^2}{2(R+\ell)}}.
	\end{align*}
	Thus, we obtain as $R\to\infty$ that
	\begin{equation*}
		\sum_{k\ge R+R^{\frac{1}{2}+\epsilon}} f_\beta\bigl(\Pr[\big]{\Gamma_k^{(\LL)} \le R}\bigr)
		\le C_{\LL,\beta} \sum_{\ell=R^{\frac{1}{2}+\epsilon}}^\infty (R+\ell)^{(2\LL+2)\Delta_\beta} e^{-\frac{\Delta_\beta \ell^2}{2(R+\ell)}}
		\to 0.
	\end{equation*}
	Similarly, we obtain by \cref{lem:GinibreTypeBounds} and \cref{lem:GinibreBounds} for any $1\le k\le R-R^{\frac{1}{2}+\epsilon}$ that
	\begin{equation*}
		\Pr[\big]{\Gamma_k^{(\LL)} > R}
		\le C_\LL k^{\LL} \Pr[\big]{\Gamma_{k+\LL} > R}
		\le C_\LL R^{\LL} e^{k-R} \Bigl(\frac{R}{k}\Bigr)^k
		\le C_\LL R^{\LL} e^{-R^{\frac{1}{2}+\epsilon}} \Bigl(\frac{R}{R-R^{\frac{1}{2}+\epsilon}}\Bigr)^{R-R^{\frac{1}{2}+\epsilon}}.
	\end{equation*}
	Since $\log(1-x) \ge -x-\frac{4}{5}x^2$ for $0<x<\frac{1}{2}$, we find for $R$ large enough that
	\begin{equation*}
		e^{-R^{\frac{1}{2}+\epsilon}} \Bigl(\frac{R}{R-R^{\frac{1}{2}+\epsilon}}\Bigr)^{R-R^{\frac{1}{2}+\epsilon}}
		= e^{-R^{\frac{1}{2}+\epsilon} - (R-R^{\frac{1}{2}})\log(1-R^{\epsilon-\frac{1}{2}})}
		\le e^{-\frac{1}{5}R^{2\epsilon}-R^{3\epsilon-\frac{1}{2}}}
		\le e^{-\frac{1}{5}R^{2\epsilon}}
	\end{equation*}
	which implies for $R\to\infty$ that
	\begin{equation*}
		\sum_{k\le R-R^{\frac{1}{2}+\epsilon}} f_\beta\bigl(\Pr[\big]{\Gamma_k^{(\LL)} \le R}\bigr)
		\le C_\beta \sum_{k\le R-R^{\frac{1}{2}+\epsilon}} \Pr[\big]{\Gamma_k^{(\LL)} > R}^{\Delta_\beta}
		\le C_{\LL,\beta} R^{\LL \Delta_\beta+1} e^{-\frac{1}{5}\Delta_\beta R^{2\epsilon}}
		\to 0.
	\end{equation*}
	Hence, we proved for $R\to\infty$ that
	\begin{equation}\label{eq:EntanglementTruncated}
		S_\LL^\beta(D_r)
		= \sum_{\abs{k-R}\le R^{\frac{1}{2}+\epsilon}} f_\beta\bigl(\Pr[\big]{\Gamma_k^{(\LL)} \le R}\bigr) + o_{\LL,\beta}(1).
	\end{equation}

	Next, we show that we can replace $\Pr{\Gamma_k^{(\LL)} \le R}$ in the right-hand side of \eqref{eq:EntanglementTruncated} by its asymptotic expansion \eqref{cond} with $\Phi_\LL$ and $\Psi_\LL$ given in \cref{lem:GinibreTypeAssumptionA}. We use that the radii $(\Gamma_k^{(\LL)})_{k\in\bN}$ of the Ginibre $\LL$-type ensemble satisfy \eqref{cond} with the explicit error $\theta_{R,k} = \frac{C_\LL}{R}$ for all $k\in\intcc{R-R^{\frac{1}{2}+\epsilon},R+R^{\frac{1}{2}+\epsilon}}$ as proven in \cref{lem:GinibreTypeBE}. Since $f_\beta$ is bounded for all $\beta>0$, the Hölder continuity implies for all $k\in\intcc{R-R^{\frac{1}{2}+\epsilon},R+R^{\frac{1}{2}+\epsilon}}$ that
	\begin{equation}\label{eq:EntanglementBetaOneHalf}
		f_\beta\bigl(\Pr[\big]{\Gamma_k^{(\LL)} \le R}\bigr)
		= f_\beta\biggl(\Phi_\LL\Bigl(\frac{k-R}{\sqrt{R}}\Bigr) + \frac{1}{\sqrt{R}}\Psi_\LL\Bigl(\frac{k-R}{\sqrt{R}}\Bigr)\biggr) + O\bigl(R^{-\Delta_\beta}\bigr).
	\end{equation}
	Since $\epsilon < \Delta_\beta - \frac{1}{2}$ (here we use that $\beta>\frac{1}{2}$), the error in the above equation is summable. This yields
	\begin{equation}\label{eq:EntanglementApprox}
		S_\LL^\beta(D_r)
		= \sum_{\abs{k-R}\le R^{\frac{1}{2}+\epsilon}} f_\beta\biggl(\Phi_\LL\Bigl(\frac{k-R}{\sqrt{R}}\Bigr) + \frac{1}{\sqrt{R}}\Psi_\LL\Bigl(\frac{k-R}{\sqrt{R}}\Bigr)\biggr) + o_{\LL,\beta}(1).
	\end{equation}

	In the following, we further expand the right-hand side of \eqref{eq:EntanglementApprox} and approximate it by its leading order term. Since $f_\beta$ is smooth on $\intoo{0,1}$ for all $\beta>0$, a Taylor expansion shows that
	\begin{align*}
		f_\beta\biggl(\Phi_\LL\Bigl(\frac{k-R}{\sqrt{R}}\Bigr) + \frac{1}{\sqrt{R}}\Psi_\LL\Bigl(\frac{k-R}{\sqrt{R}}\Bigr)\biggr)
		&= f_\beta\Bigl(\Phi_\LL\Bigl(\frac{k-R}{\sqrt{R}}\Bigr)\Bigr)
		+ \frac{1}{\sqrt{R}}\Psi_\LL\Bigl(\frac{k-R}{\sqrt{R}}\Bigr) f'_\beta\Bigl(\Phi_\LL\Bigl(\frac{k-R}{\sqrt{R}}\Bigr)\Bigr)\\
		&\qquad + \frac{1}{R}\Psi_\LL\Bigl(\frac{k-R}{\sqrt{R}}\Bigr)^2 f''_\beta\bigl(x_{k,R}\bigr)
	\end{align*}
	with $x_{k,R}\in\intoo{0,1}$. By \cref{lem:TailsPsiPhiPrimePhi} and because $(2\alpha+3)\epsilon<\frac{1}{2}$, we have in particular that $\abs{\Psi_\LL(x)} \le \frac{\sqrt{R}}{2} \Phi_\LL(x)(1-\Phi_\LL(x))$ for all $\abs{x}\le R^\epsilon$ and all $R$ large enough. This shows that $\frac{\Phi_\LL(x)}{2}\le x_{k,R}\le 1-\frac{1-\Phi_\LL(x)}{2}$. A direct computation yields for all $\beta>0$ and all $x\in\intoo{0,1}$ that $\abs{f''_\beta(x)} \le C_\beta (x^{\Delta_\beta-2} + (1-x)^{\Delta_\beta-2})$. Hence, \cref{lem:TailsPsiPhiPrimePhi} implies for all $k\in\intcc{R-R^{\frac{1}{2}+\epsilon},R+R^{\frac{1}{2}+\epsilon}}$ that
	\begin{align*}
		\abs[\bigg]{\frac{1}{R}\Psi_\LL\Bigl(\frac{k-R}{\sqrt{R}}\Bigr)^2 f''_\beta\bigl(x_{k,R}\bigr)}
		&\le \frac{C_{\LL,\beta}}{R} \abs[\Big]{\Psi_\LL\Bigl(\frac{k-R}{\sqrt{R}}\Bigr)}^2 \biggl(1 + \Phi_\LL\Bigl(\frac{k-R}{\sqrt{R}}\Bigr)^{\Delta_\beta-2} + \Bigl(1-\Phi_\LL\Bigl(\frac{k-R}{\sqrt{R}}\Bigr)\Bigr)^{\Delta_\beta-2}\biggr)\\
		&\le C_{\LL,\beta} R^{-1+(4\LL+6)\epsilon} \exp\Bigl(-\Delta_\beta\frac{(k-R)^2}{2R}\Bigr).
	\end{align*}
	Because we chose $-1+(4\alpha+6)\epsilon<-\frac{1}{2}-\epsilon$, this proves for all $k\in\intcc{R-R^{\frac{1}{2}+\epsilon},R+R^{\frac{1}{2}+\epsilon}}$ that
	\begin{align*}
		\MoveEqLeft f_\beta\biggl(\Phi_\LL\Bigl(\frac{k-R}{\sqrt{R}}\Bigr) + \frac{1}{\sqrt{R}}\Psi_\LL\Bigl(\frac{k-R}{\sqrt{R}}\Bigr)\biggr)\\
		&= f_\beta\Bigl(\Phi_\LL\Bigl(\frac{k-R}{\sqrt{R}}\Bigr)\Bigr)
		+ \frac{1}{\sqrt{R}}\Psi_\LL\Bigl(\frac{k-R}{\sqrt{R}}\Bigr) f'_\beta\Bigl(\Phi_\LL\Bigl(\frac{k-R}{\sqrt{R}}\Bigr)\Bigr) + o_{\LL,\beta}\bigl(R^{-\frac{1}{2}-\epsilon}\bigr)
	\end{align*}
	and therefore \eqref{eq:EntanglementApprox} implies
	\begin{equation}\label{eq:EntanglementApprox2}
		S_\LL^\beta(D_r)
		= \sum_{\abs{k-R}\le R^{\frac{1}{2}+\epsilon}} f_\beta\Bigl(\Phi_\LL\Bigl(\frac{k-R}{\sqrt{R}}\Bigr)\Bigr) + \frac{1}{\sqrt{R}}\Psi_\LL\Bigl(\frac{k-R}{\sqrt{R}}\Bigr) f'_\beta\Bigl(\Phi_\LL\Bigl(\frac{k-R}{\sqrt{R}}\Bigr)\Bigr) + o_{\LL,\beta}(1).
	\end{equation}

	We now want to apply the Euler-Maclaurin formula \eqref{eq:EulerMaclaurinReduced} to calculate the sum. To do so, let us first consider the error term. A direct computation shows for all $\beta>0$ and all $x\in\intoo{0,1}$ that $f'_\beta(x) \le C_\beta(x^{\Delta_\beta-1}+(1-x)^{\Delta_\beta-1})$.
	Using again \cref{lem:TailsPsiPhiPrimePhi}, we obtain in a similar way as before for all $\abs{x}\le R^\epsilon$ that
	\begin{equation*}
		\abs[\Big]{\frac{1}{\sqrt{R}}\Bigl(\Psi'_\LL(x) f'_\beta\bigl(\Phi(x)\bigr) + \Psi_\LL(x) f''_\beta\bigl(\Phi_\LL(x)\bigr) \Phi'_\LL(x)\Bigr)}
		\le C_{\LL,\beta} R^{-\frac{1}{2}+(2\LL+3)\epsilon} e^{-\Delta_\beta\frac{x^2}{2}}.
	\end{equation*}
	By our choice of $\epsilon$, Euler-Maclaurin's formula \eqref{eq:EulerMaclaurinReduced} implies as $R\to\infty$ that
	\begin{equation*}
		\sum_{\abs{k-R}\le R^{\frac{1}{2}+\epsilon}} \frac{1}{\sqrt{R}}\Psi_\LL\Bigl(\frac{k-R}{\sqrt{R}}\Bigr) f'_\beta\Bigl(\Phi_\LL\Bigl(\frac{k-R}{\sqrt{R}}\Bigr)\Bigr)
		= \int_{-R^\epsilon}^{R^\epsilon} \Psi_\LL(u) f'_\beta\bigl(\Phi_\LL(u)\bigr) \dif u + o_{\LL,\beta}(1).
	\end{equation*}
	Because $x\mapsto f'(\Phi(x))$ is odd and $x\mapsto \Psi(x) = \frac{x^2}{2}\Phi_\LL'(x) + \frac{1}{3}\Phi_\LL'''(x) $ is even, it holds
	\begin{equation*}
		\int_{-R^\epsilon}^{R^\epsilon} \Psi_\LL(u) f'_\beta\bigl(\Phi_\LL(u)\bigr) \dif u
		= 0.
	\end{equation*}
	Combining this with \eqref{eq:EntanglementApprox2} yields
	\begin{equation}\label{eq:EntanglementApprox3}
		S_\LL^\beta(D_r)
		= \sum_{\abs{k-R}\le R^{\frac{1}{2}+\epsilon}} f_\beta\Bigl(\Phi_\LL\Bigl(\frac{k-R}{\sqrt{R}}\Bigr)\Bigr) + o_{\LL,\beta}(1).
	\end{equation}

	Let us finally apply the Euler-Maclaurin formula \eqref{eq:EulerMaclaurin} to the remaining term in \eqref{eq:EntanglementApprox3}. We find
	\begin{equation}\label{eq:EntanglementLeadingOrder}
	\begin{split}
		\sum_{\abs{k-R}\le R^{\frac{1}{2}+\epsilon}} f_\beta\Bigl(\Phi_\LL\Bigl(\frac{k-R}{\sqrt{R}}\Bigr)\Bigr)
		&= \sqrt{R} \int_{-R^\epsilon}^{R^\epsilon} f_\beta\bigl(\Phi_\LL(u)\bigr) \dif u + \frac{f_\beta(\Phi_\LL(R^\epsilon)) + f_\beta(\Phi_\LL(-R^\epsilon))}{2}\\
		&\qquad + \int_{-R^\epsilon}^{R^\epsilon} \Bigl(\fracpart{\sqrt{R}u}-\frac{1}{2}\Bigr) \Phi'_\LL(u) f'_\beta\bigl(\Phi_\LL(u)\bigr) \dif u.
	\end{split}
	\end{equation}
	For the first term we combine the Hölder continuity of $f_\beta$ with the tail bounds
	\begin{align*}
		\Phi_\LL(x) &= \int_x^\infty H_\LL(y)^2 \frac{e^{-\frac{y^2}{2}}}{\sqrt{2\pi}} \dif y
		\le C_\LL \int_x^\infty y^{2\LL} e^{-\frac{y^2}{2}} \dif y
		\le C_\LL x^{2\LL-1} e^{-\frac{x^2}{2}},\\
		1-\Phi_\LL(x) &= \int_{-\infty}^{-x} H_\LL(y)^2 \frac{e^{-\frac{y^2}{2}}}{\sqrt{2\pi}} \dif y
		\le C_\LL \int_{-\infty}^{-x} \abs{y}^{2\LL} e^{-\frac{y^2}{2}} \dif y
		\le C_\LL x^{2\LL-1} e^{-\frac{x^2}{2}}
	\end{align*}
	for $x>0$ to obtain as $R\to\infty$ that
	\begin{equation*}
		\sqrt{R} \int_{-R^\epsilon}^{R^\epsilon} f_\beta\bigl(\Phi_\LL(u)\bigr) \dif u
		= \sqrt{R} \int_{-\infty}^{\infty} f_\beta\bigl(\Phi_\LL(u)\bigr) \dif u + o_{\LL,\beta}(1).
	\end{equation*}
	The second term in \eqref{eq:EntanglementLeadingOrder} clearly converges to zero as $R\to\infty$. The last term converges to zero by a generalized Riemann-Lebesgue lemma (apply e.g.\ \cite[Theorem 1]{CFM16RiemannLebesgueLemma} with $g_R(x) = (\fracpart{\sqrt{R}x}-\frac{1}{2})\ind_{\abs{x}\le R^\epsilon}$).
\end{proof}
\begin{remark}
	Our argument uses that $\beta>\frac{1}{2}$ only to sum the error in \eqref{eq:EntanglementBetaOneHalf}. For $\beta\le \frac{1}{2}$ the error can be improved by continuing the expansion in \cref{lem:GinibreTypeBE} up to higher order:
	\begin{equation*}
		\Pr[\bigg]{\frac{\Gamma_{k-\LL}^{(\LL)}-k}{\sqrt{k}} \le x}
		= \Pr[\big]{Z_\LL \le x} + \frac{1}{\sqrt{3k}}\Phi_\alpha'''(x) + \frac{1}{k}\Psi_{\LL,2}(x) + \frac{1}{k^\frac{3}{2}}\Psi_{\LL,3}(x) + \dotsb
	\end{equation*}
	This would allow us to prove
	\begin{equation*}
		S_\LL^\beta(D_r)
		= \sqrt{R} \int_{-\infty}^\infty f_\beta\bigl(\Phi_\LL(u)\bigr) \dif u + o_{\LL,\beta}(1)
	\end{equation*}
	for all $\beta>0$.
\end{remark}

\subsection{Proofs of the results in \texorpdfstring{\cref{sec:hyperbolic}}{Section \ref*{sec:hyperbolic}}}
\label{sec:HyperbolicProof}

In this section, we analyse the hyperbolic ensembles $\zeta_\rho$, $\rho>0$, introduced in \cref{sec:hyperbolic}.
Let us recall that the intensity of the point process $\zeta_\rho$ corresponds to the volume form of the Poincar\'e disk with curvature $-\frac{2\pi}{\rho}$ and by \eqref{GammaHyper}, the radii $(\Gamma_k^{(\DD)})_{k\in\bN}$ of the points of $\zeta_\rho$ are $\DD \Betaprimedist(k,\DD)$-distributed.
In this case, we obtain the following asymptotics.

\begin{lemma}\label{lem:HyperbolicExpansion}
	For any $\DD>0$, it holds for all $k\in\bN$, as $R\to\infty$,
\begin{equation} \label{hyper0}
\Pr[\big]{\Gamma_k^{(\DD)} \le R}
= \Phi_\DD\Bigl(\frac{k}{R}\Bigr) + \frac{1}{R} \Psi_\DD\Bigr(\frac{k}{R}\Bigr) +O_\rho\Bigl( k^{-2} \wedge R^{-(2\wedge \rho)}\Bigr) ,
	\end{equation}
	where the probability tail function $\Phi_\DD$ and the function $\Psi_\DD$ are as in \cref{thm:HyperbolicModPhi}.
\end{lemma}

\begin{proof}
	By \eqref{GammaHyper}, we can rewrite for $k\in\bN$ and $R>0$,
	\begin{equation} \label{hyper1}
		\Pr[\big]{\Gamma_k^{(\DD)} \le R}
		= \Pr[\Big]{\Betaprimedist(k,\DD) \le \frac{R}{\DD}}
		= 1 - \frac{\Gamma(k+\DD)}{\Gamma(k)\Gamma(\DD)} \int_\frac{R}{\DD}^\infty
		\Bigl(\frac{x}{1+x}\Bigr)^k \frac{\dif x}{x(1+x)^\rho}.
	\end{equation}
Taylor expansions show that it holds for all $x\ge 1$
	\begin{align*}
		\Bigl(\frac{x}{1+x}\Bigr)^k
		&= \exp\biggl(-k\log\Bigl(1+\frac{1}{x}\Bigr)\biggr)
		= \exp\biggl(-\frac{k}{x} + \frac{k}{2x^2} + O\Bigl(\frac{k}{x^3}\Bigr)\biggr)\\
		&= \exp\Bigl(-\frac{k}{x}\Bigr) \biggl(1+\frac{k}{2x^2} + O\Bigl( \frac{k^2(x^{-1}+k^{-1})}{x^3}\Bigr)\biggr),
	\end{align*}
where the error is independent of $k$ and
\[
(1+x)^{-\DD} = x^{-\DD} \bigl(1 - \DD x^{-1} + O_{\rho}\bigl(x^{-2}\bigr)\bigr).
 \]
By \eqref{hyper1} and using that $\frac{k}{x} \le 1+\frac{k^2}{x^2}$, this implies that
\begin{equation*}
	\begin{aligned}
\Pr[\big]{\Gamma_k^{(\DD)} \le R}
		&= 1 - \frac{\Gamma(k+\DD)}{\Gamma(k)\Gamma(\DD)}
 \int_\frac{R}{\DD}^\infty x^{-\DD-1} e^{-\frac{k}{x}} \biggl( 1 - \frac{\DD}{x} + \frac{k}{2x^2} + O_\rho\Bigl( \frac{1}{x^2}+\frac{k^2}{x^4}\Bigr) \biggr) \dif x \\
		&= 1 - \frac{\Gamma(k+\DD)}{k^\DD\Gamma(k)\Gamma(\DD)} \int_0^{\frac{\DD k}{R}} u^{\DD-1} e^{-u} \biggl( 1 - \frac{\DD u}{k} + \frac{u^2}{2k} + O_\rho\Bigl(\frac{u^2(1+u^2)}{k^2} \Bigr) \biggr)\dif u ,
	\end{aligned}
\end{equation*}
where we made the change of variable $u = k/x$.
Since for any $\beta \ge 0$, it holds for all $x>0$
 \begin{equation} \label{hyper2}
 \int_0^x u^{\DD-1+\beta} e^{-u} \dif u \le C_\rho \min\set{x^{\rho+\beta}, 1}
 \end{equation}
and $\frac{\Gamma(k+\DD)}{k^\DD \Gamma(k)\Gamma(\DD)} \le C_\rho$ for all $k\in\bN$, this shows that
\begin{equation} \label{hyper3}
\Pr[\big]{\Gamma_k^{(\DD)} \le R} = 1- \frac{\Gamma(k+\DD)}{k^\DD\Gamma(k)\Gamma(\DD)} \int_0^{\frac{\DD k}{R}} u^{\DD-1} e^{-u} \biggl( 1 - \frac{\DD u}{k} + \frac{u^2}{2k} \biggr)\dif u + O_\rho\Bigl( k^{-2} \wedge R^{-2}\Bigr).
\end{equation}

It remains to plug in the asymptotics of the combinatorial factor in \eqref{hyper3}:
\begin{equation*}
		\frac{\Gamma(k+\DD)}{k^\DD\Gamma(k)}
		= \biggl( 1+\frac{\DD(\DD-1)}{2k} + O_\rho\Bigl(\frac{1}{k^2}\Bigr) \biggr)
		\qquad \text{as } k\to\infty.
\end{equation*}
Using \eqref{hyper2} again, we conclude that for any $k\in\bN$ and $R>0$,
\[
\Pr[\big]{\Gamma_k^{(\DD)} \le R} = 1- \frac{1 }{\Gamma(\DD)} \int_0^{\frac{\DD k}{R}} u^{\DD-1} e^{-u} \biggl( 1+\frac{\DD(\DD-1)}{2k} + \frac{u^2}{2k} - \frac{\DD u}{k} \biggr)\dif u + O_\rho\Bigl( k^{-2} \wedge R^{-(2\wedge \rho)}\Bigr).
\]
If $Y_\rho \overset{\dif }{=} \Gammadist(\DD,\DD) $ is a random variable, 	$\Phi_\DD$ denotes its probability tail function, and we let for any $x\in\bR$
\[ \begin{aligned}
\Psi_\rho (x)
& = - \frac{1 }{\Gamma(\DD)} \frac{ \ind_{x> 0}}{2x} \int_0^{\rho x} u^{\DD-1} e^{-u} \bigl( u^2 -2\rho u + \DD(\DD-1) \bigr) \dif u \\
&= \ind_{x> 0} \frac{\rho^{\rho} }{2\Gamma(\DD)} \bigl( \rho x +\rho-1 \bigr) x^{\rho-1} e^{-\rho x},
\end{aligned}\]
then we obtain the asymptotic expansion \eqref{hyper0}.
\end{proof}

For any $\rho>0$,
\cref{lem:HyperbolicExpansion} shows that the radii $(\Gamma_k^{(\DD)})_{k\in\bN}$ of the hyperbolic ensemble $\zeta_\rho$ satisfy \eqref{cond} with $\vartheta=0$, $\Sigma_R = R$ for $R>0$ and $\theta_{k,R} = O_\rho\bigl( k^{-2} \wedge R^{-(2\wedge \rho)}\bigr)$.
In particular, we verify that $\sum_{k\in\bN} \theta_{R,k}= O_\rho\bigl( R^{- 1\wedge(\rho/2)} \bigr)$
and $\Phi_\DD(x) = \Pr{Y_\rho\le x}$, where $Y_\DD$ is a gamma-distributed random variable with shape $\DD$ and rate $\DD$ so that $Y_\rho$ is absolutely continuous and positive with $\Ex{Y_\rho}<\infty$.
Hence, since the function $\Psi_\rho \in W^{1,1}(\bR)$ if $\rho\ge 1$, \cref{NA} are satisfied in this case, and we can apply the results from \cref{sec:MainResults}.

Nonetheless, let us observe that for $\rho<1$, we can decompose
$\Psi_\rho (x) = \Psi_{AC}(x) + \Psi_{M}(x)$
with $\Psi_{AC}(x) = \ind_{x> 0} \frac{\rho^{\rho+1}}{2\Gamma(\rho)} x^\rho e^{-\rho x} \in W^{1,1}(\bR)$ and $\Psi_{M}(x) = \ind_{x> 0} \frac{\rho^{\rho}(\rho-1)}{2\Gamma(\rho)} x^{\rho-1} e^{-\rho x} \in L^1(\bR_+)$ being increasing on $(0,\infty)$.
So, by \cref{rk:BV}, we can still apply our main results when $\rho<1$.

Hence, by applying \cref{thm:ModPhiGeneral} and \cref{thm:FCLTHyperbolic}, this concludes the proofs of \cref{thm:HyperbolicModPhi} and \cref{thm:HyperbolicFCLT}.



\section{Proof of \texorpdfstring{\cref{thm:ldp}}{Theorem \ref{thm:ldp}}}
\label{sec:JLMProof}

In this section, we work under the \cref{ass:LD} and define for $R, k>0$,
\begin{equation}\label{eq:DefPk}
	\PP(k) = \max\set{k,R}^\beta e^{k-R} \Bigl(\frac{R}{k}\Bigr)^k ,
\end{equation}
for a constant $\beta>0$.
Using \eqref{Xialpha} and the independence of the random variables $\Gamma_{k}^{(\LL)}$, it formally holds
\begin{equation} \label{pTheta}
\Pr[\big]{\Xi_R^{(\LL)} \ge \Theta_R} = \sum_{I \subseteq \bN, |I| \ge \Theta_R} \prod_{k\in I}	\Pr{\Gamma_k^{(\LL)} \le R} \prod_{k\notin I}	\Pr{\Gamma_k^{(\LL)} > R}.
\end{equation}
We claim that the large deviations estimates from \cref{thm:ldp} only depend on the tails of the random variables, so that we can use the function $\PP(k)$ to control the probabilities on the RHS of \eqref{pTheta}.
Namely, by \cref{lem:GinibreBounds} and \cref{lem:GinibreTypeBounds}, we see that for any $\LL>0$, there exists a constant $\beta>0$ such that
\begin{equation} \label{p}
\begin{aligned}
	\Pr{\Gamma_k^{(\LL)} \le R} \le \PP(k) \qquad\text{for all }k\ge R, \\
	\Pr{\Gamma_k^{(\LL)} > R} \le \PP(k) \qquad\text{for all }k\le R.
\end{aligned}
\end{equation}
Moreover, the function $\PP(k)$ is monotone\footnote{Since the random variables $\bigl(\Gamma_k^{(\LL)}\bigr)_{k\in\bN}$ are not stochastically ordered, it is very convenient to replace $\Pr{\Gamma_k^{(\LL)} \le R} $ by $\PP(k)$ in the expansion \eqref{pTheta}.}: There exists a constant $c_\beta>0$ such that $k\mapsto \PP(k)$ is non-decreasing for $k\le R$ and non-increasing for $k\ge R+c_\beta$.
We also make crucial use of the following bound:
 Using that $\log(1-t) \le -t - \frac{t^2}{2}$ for all $t\in\intco{0,1}$, we obtain
 \begin{equation}\label{eq:JLMSizePOutsideWindow}
 \begin{aligned}
	\PP(R+x\Theta_R)
	& = (R+x\Theta_R)^\beta \exp\biggl(x\Theta_R + (R+x\Theta_R)\log\Bigl(1-\frac{x\Theta_R}{R+x\Theta_R}\Bigr)\biggr) \\
	&\le (R+x\Theta_R)^\beta \exp\biggl(- \frac{x^2 \Theta_R^2}{2R} \wedge \frac{x\Theta_R}{2}\biggr).
	\end{aligned}
\end{equation}
Under the \cref{ass:LD}, the RHS of \eqref{eq:JLMSizePOutsideWindow} converges to zero faster than any power of $(R+\Theta_R)$ as $R\to\infty$.

We explained in \cref{sec:Ginibre} that the function $J_\gamma$ which governs the large deviations is coming from the exponential tail of the random variables $\Gamma_k^{(\LL)}$, see \cref{thm:gammaJ}.
It also appears as follows:

\begin{proposition}\label{lem:JLMDeviationPrinciples}
Under the \cref{ass:LD}, for all $C>0$, it holds uniformly for all $x\in\intcc{0,C}$
	\begin{equation} \label{cvgJ}
		- \lim_{R\to\infty} \frac{1}{v_R} \log \PP(R+x\Theta_R) = J_\gamma(x).
	\end{equation}
\end{proposition}
\begin{proof}
With $J_2(x)= (1+x)\log(1+x)-x$, observe that for all $x>0$,
	\begin{equation*}
		\log\PP(R+x\Theta_R) = \beta\log(R+x\Theta_R) - R J_2\Bigl(\frac{x\Theta_R}{R}\Bigr).
	\end{equation*}
Under the \cref{ass:LD}, $\frac{\log(R+C\Theta_R)}{v_R} \to 0$ as $R\to\infty$, so it suffices to show that $\frac{R}{v_R} J_2(\frac{x\Theta_R}{R}) \to J_\gamma(x)$ uniformly for all $x\in\intcc{0,C}$.
In case $\frac{\Theta_R}{R} \to 1$, as $v_R =\Theta_R$ and $J_2$ is continuous on $\intco{0,\infty}$, the claim is straightforward.

Consider now the case $\frac{\Theta_R}{R} \to 0$ and $v_R = \frac{\Theta_R^2}{R}$. A Taylor expansion shows that $J_2(t) = -\frac{t^2}{2} + O(t^3)$ as $t\to0$ which yields as $R\to\infty$,
	\begin{equation*}
		\frac{R}{v_R} f\Bigl(\frac{x\Theta_R}{R}\Bigr) = -\frac{x^2}{2} + O\Bigl(C^3\frac{\Theta_R}{R}\Bigr).
	\end{equation*}

	In case $\frac{\Theta_R}{R}\to\infty$ and $v_R = \Theta_R \log\Theta_R$, using the asymptotics of $J_2$, we have
	\begin{equation*}
		\frac{R}{v_R} J_2\Bigl(\frac{x\Theta_R}{R}\Bigr)
		= \frac{x}{\log \Theta_R} \log\Bigl(\frac{x\Theta_R}{R}\Bigr) + O\Bigl(\frac{C}{\log \Theta_R}\Bigr)
		= x\Bigl(1- \frac{2}{\gamma}\Bigr)+ o(1)
	\end{equation*}
	uniformly for all $x\in\intcc{0,C}$.
\end{proof}

\begin{corollary}\label{lem:JLMLeadingOrderTerm}
	 Under the \cref{ass:LD}, it holds for all $c\ge 0$ and $x>0$ that
	\begin{equation*}
		\lim_{R\to\infty} \frac{1}{\Theta_R v_R} \log \Biggl(\prod_{k=R+c}^{R+x\Theta_R} \PP(k) \Biggr)
		= -\int_0^x J_\gamma(t) \dif t.
	\end{equation*}
\end{corollary}

\begin{proof}
This follows directly from \cref{lem:JLMDeviationPrinciples} by a Riemann sum approximation. Since the function $\PP$ depends on $R$, it is important to note that the limit \eqref{cvgJ} is uniform for all $x\in[0,C]$.
\end{proof}

Let us now prove the upper bound in \cref{thm:ldp}.

\begin{proposition}\label{prop:JLMUB}
	It holds for all $x>0$ that
	\begin{equation*}
		 \limsup_{R\to\infty} \frac{-1}{\Theta_R v_R} \log \Pr{\Xi_R \ge x\Theta_R}
		\le \int_0^x J^\gamma(t) \dif t.
	\end{equation*}
\end{proposition}
\begin{proof}
First, observe that
	\begin{equation*}
		\set[\big]{\Xi_R \ge x\Theta_R}
		\supseteq \set[\big]{\Gamma_k^{(\LL)} \le R \text{ for all } k \le R+x\Theta_R } \cap \set[\big]{ \Gamma_k^{(\LL)} > R \text{ for all } k > R+x\Theta_R }.
	\end{equation*}
	By independence of the random variables $\Gamma_k^{(\LL)}$, this implies that
	\begin{equation}\label{lb}
		\Pr[\big]{\Xi_R \ge x\Theta_R}
		 \ge \prod_{k=1}^{R+x\Theta_R} \Pr[\big]{\Gamma_k^{(\LL)} \le R} \prod_{k=R+x\Theta_R+1}^\infty \Pr[\big]{\Gamma_k^{(\LL)} > R}.
	\end{equation}
	We show that the leading contribution to the right-hand side is given by $\prod_{k=R}^{R+x\Theta_R} \Pr{\Gamma_k \le R}$.

By \eqref{p}, we have
	\begin{equation*}
		\prod_{k=R+x\Theta_R+1}^\infty \Pr{\Gamma_k^{(\LL)}>R}
		\ge \prod_{k=R+x\Theta_R+1}^\infty \bigl(1-\PP(k)\bigr).
	\end{equation*}
	Since $\PP$ is decreasing for all $k\ge R+x\Theta_R$ if $R$ is large enough and $\PP(R+x\Theta_R) \to 0$ as $R\to\infty$, we can apply the bound $\abs{\log(1-y)} \le 2y$ for $0\le y\le 0.7$ to $y=\PP(k)$ to obtain
	\begin{equation*}
		\abs[\Bigg]{\log\Biggl( \prod_{k=R+x\Theta_R+1}^\infty \Pr[\big]{\Gamma_k^{(\LL)} > R}\Biggr)}
		\le 2 \sum_{k>R+x\Theta_R} \PP(k)
		\le 2 \Theta_R \int_x^\infty \PP(R+t \Theta_R) \dif t.
	\end{equation*}
Using the bound \eqref{eq:JLMSizePOutsideWindow}, this shows that
\begin{equation} \label{est1}
\abs[\Bigg]{\log\Biggl( \prod_{k=R+x\Theta_R+1}^\infty \Pr[\big]{\Gamma_k^{(\LL)} > R} \Biggr)}
\le 2 \Theta_R \int_x^\infty (R+t\Theta_R)^\beta \exp\biggl(- \frac{t^2 \Theta_R^2}{2R} \wedge \frac{t\Theta_R}{2}\biggr) \dif t
\end{equation}
and the right-hand side converges to 0 because of our condition $\frac{\Theta_R}{\sqrt{R\log R}} \to \infty$ for $R\to\infty$.

	Let us now treat the first product in \eqref{lb}.
	Note that by \eqref{p},
	\begin{equation*}
		\prod_{k=1}^{R-\Theta_R} \Pr{\Gamma_k^{(\LL)}\le R}
		\ge \prod_{k=1}^{R-\Theta_R} (1-\PP(k)).
	\end{equation*}
Since, $\PP(k)$ is increasing for $k\le R$ the bound $\abs{\log(1-y)} \le 2y$ for $0\le y\le 0.7$ applied to $y=\PP(k)$ yields
	\begin{equation} \label{estlog}
		\abs[\Bigg]{\log\Biggl( \prod_{k=1}^{R-\Theta_R} \Pr[\big]{\Gamma_k^{(\LL)} \le R}\Biggr)}
		\le 2 \sum_{k=1}^{R-\Theta_R} \PP(k)
		\le 2 R \PP(R-\widetilde{\Theta}_R).
	\end{equation}
A similar computation as in \eqref{eq:JLMSizePOutsideWindow} shows that $\PP(R-\Theta_R) \to 0$ converges to 0 faster than any power of~$R$.
Hence, the right-hand side of \eqref{estlog} converges to 0 and it follows from \eqref{lb} and \eqref{est1} that
\begin{equation} \label{est2}
	\liminf_{R\to\infty} \log \Pr[\big]{\Xi_R \ge x\Theta_R} =
	 	\liminf_{R\to\infty} \log\Biggl( \prod_{k=R-\Theta_R}^{R+x\Theta_R} \Pr[\big]{\Gamma_k^{(\LL)} \le R} 	\Biggr)
\end{equation}

	By \cref{lem:GinibreTypeBE} (choosing $x= \frac{R-k}{\sqrt{R}}$ and using the uniformity), for any small $\epsilon>0$, it holds for $k\in\intcc{R-\widetilde{\Theta}_R,R}\cap\bN$ that
\[
\Pr{\Gamma_k^{(\LL)} \le R} \ge \Pr[\Big]{Z_\LL \le \frac{R-k-\alpha}{\sqrt{k+\alpha}}} -\epsilon.
\]
Using that the function $k\mapsto \frac{R-k-\alpha}{\sqrt{k+\alpha}}$ is decreasing, this shows that
$\Pr{\Gamma_k^{(\LL)} \le R} \ge \Pr[\big]{Z_\LL \le \frac{-\alpha}{\sqrt{R+\alpha}}} - \epsilon$, so that for $\epsilon>0$ small enough (depending on $\alpha$)
	\[ \begin{aligned}
		\liminf_{R\to\infty} \frac{1}{\Theta_R v_R} \log\Biggl( \prod_{k=R-\Theta_R}^R \Pr[\big]{\Gamma_k^{(\LL)} \le R}\Biggr)
		& \ge \liminf_{R\to\infty} \frac{1}{\Theta_R v_R} \sum_{k=R-\Theta_R}^R \log\bigl( \Pr[\big]{Z_\LL \le -\alpha} - \epsilon \bigr) \\
		&\ge c_\alpha \liminf_{R\to\infty} \frac{\Theta_R}{\Theta_R v_R}
		\ge c_\alpha \liminf_{R\to\infty} \frac{1}{v_R}= 0.
	\end{aligned}\]
	By combining the previous estimates with \eqref{est2}, we conclude that
	\begin{align*}
		\liminf_{R\to\infty} \frac{1}{\Theta_R v_R} \log \Pr{\Xi_R \ge x\Theta_R}
		& \ge \liminf_{R\to\infty} \frac{1}{\Theta_R v_R} \log\Biggl( \prod_{k=R}^{R+x\Theta_R} \Pr{\Gamma_k \le R}\Biggr).
	\end{align*}
	By \cref{lem:JLMLeadingOrderTerm}, this completes the proof.
\end{proof}

Let us finish the proof of \cref{thm:ldp} by proving the complementary lower bound.

\begin{proposition} \label{prop:JLMLB}
	It holds for all $x>0$ that
	\[
		\liminf_{R\to\infty} \frac{-1}{\Theta_R v_R} \log \Pr{\Xi_R \ge x\Theta_R}
		\ge \int_0^x J_\gamma(y) \dif y.
	\]
\end{proposition}
\begin{proof}
Here, we use the convention that for any $m>0$ and $k\in\bN$, $\binom{m}{k} = \ind_{k\le \floor{m}} \binom{\floor{m}}{k}$.
We also use the shorthand notation $\theta_R = R+x \Theta_R$. Recall that by \eqref{def:Xi}, the random variable $\Xi_R$ is centred for any $R>0$. In particular, it has finite expectation such that $\Pr{\Xi_R = \infty} = 0$, and we have the inclusion for any $x>0$,
	\[
		\set[\big]{\Xi_R \ge x\Theta_R} \subseteq \bigcup_{I\in S_0} \set[\big]{ \Gamma_k^{(\LL)} \le R \text{ for all } k \in I } \cup \set[\big]{\Xi_R=\infty},
	\]
	where $S_0 = \set{ I \subseteq \bN \given \theta_R \le \abs{I} < \infty }$. Since the random variables $\Gamma_k^{(\LL)}$ are independent, by a union bound, this shows that
	\begin{equation} \label{estimate0}
		\Pr[\big]{\Xi_R \ge x\Theta_R} \le \sum_{I\in S_0} \prod_{k\in I} \Pr[\big]{\Gamma_k^{(\LL)} \le R}
		\le \sum_{I\in S_0} \prod_{\substack{k\in I\\k\ge R}} \Pr[\big]{\Gamma_k^{(\LL)} \le R}.
	\end{equation}

	Recall \eqref{eq:DefPk} and \eqref{p} and that the function $\PP$ is decreasing for $k\ge R+c_\beta$. For any interval $I \in S_0$, if we let
	\[
		m = \sup I,\qquad
		r_1 = | \intco{1,R+c_\beta} \cap I^c |, \qquad
		r_2 = | \intco{R+c_\beta, \theta_R} \cap I^c| ,
		\quad\text{and}\quad
		r_3 = | I \cap \intco{\theta_R,m} | ,
	\]
	then we can bound
	\[
 		\prod_{k\in I, k\ge R} \hspace{-0.2cm} \Pr{\Gamma_k \le R}
		\le \Biggl(\prod_{k\in I \cap \intco{R+c_\beta, \theta_R}} \hspace{-0.7cm} \PP(k)\Biggr) \PP(\theta_R)^{r_3} \PP(m)
		\le \Biggl( \prod_{k\in \intco{R+c_\beta, \theta_R}} \PP(k) \bigg) \PP(\theta_R)^{(r_3-r_2)} \PP(m),
	\]
	where we used that $\prod_{k\in\intco{R+c_\beta, \theta_R} \cap I^c} \PP(k) \ge \PP(\theta_R)^{r_2}$ and that $r_1+r_2 \le r_3$ (this constraint follows from the fact that $\abs{I} \ge \theta_R$ for $I\in S_0$) to get the second bound.
	We assume that $x\in[0,c]$.
	Since there are at most $\binom{R+c_\beta}{r_1} \binom{c\Theta_R}{r_2} \binom{m}{r_3}$ many ways to choose an interval $I\in S_0$ with given parameters $(r_1,r_2,r_3,m)$, according to \eqref{estimate0}, this shows that
	\begin{equation}\label{estimate1}
		\Pr{\Xi_R \ge x\Theta_R}
		\le \Biggl( \prod_{k=R+c_\beta}^{\theta_R} \PP(k) \Biggr) \sum_{\substack{r_1+r_2 \le r_3 \le m \\ r_1+r_2\le \theta_R \le m}} \binom{R+c_\beta}{r_1} \binom{c\Theta_R}{r_2} \binom{m}{r_3} \PP(\theta_R)^{(r_3-r_2)} \PP(m).
	\end{equation}
	Let us emphasize that this bound is very crude, but \ it suffices to control large deviation probabilities.
	First note that if $R$ is sufficiently large (depending only on $\beta$), Stirling's formula implies that $\PP(m) \le m^{\beta+1} e^{-R} \frac{R^m}{m!} $ for all $m\ge R$. Using this together with $(a+b)!\ge a!b!$ for all $a,b\in\bN$ and $(a+b)^\beta \le (2ab)^\beta$ for all $a,b\ge1$ and all $\beta \ge 0$ we obtain
	\begin{equation} \label{estimate2}
		\sum_{m = r_3}^\infty \binom{m}{r_3} \PP(m)
		\le e^{-R} \sum_{k=0}^\infty \frac{R^{r_3+k} }{r_3!k!} (r_3+k)^{\beta+1}
		\le 2^{\beta+1}\frac{R^{r_3}r_3^{\beta+1}}{r_3!}e^{-R} \sum_{k=0}^\infty \frac{R^k k^{\beta+1}}{k!}
		\le C_\beta\frac{R^{r_3+\beta+1}r_3^{\beta+1}}{r_3!}
	\end{equation}
	for some constant $C_\beta>0$.
Second, since $r_2 \le c\Theta_R$ for $x\in[0,c]$ and $r_1 +r_2 \le \theta_R$, we also have
	\begin{equation*}
		\sum_{r_3=r_2+r_1}^\infty \frac{R^{r_3} r_3^{\beta+1}}{r_3!} \PP(\theta_R)^{(r_3-r_2)}
		\le C_\beta \theta_R^{\beta+1} \frac{\PP(\theta_R)^{r_1}}{r_1!} R^{r_1+c\Theta} \biggl( 1+\sum_{k=1}^\infty \frac{\bigl(R \PP(\theta_R)\bigr)^{k} k^{\beta+1}}{k!} \biggr).
	\end{equation*}

	Using the condition \eqref{eq:JLMSizePOutsideWindow}, we check that $\sum_{k=1}^\infty \frac{R^k \PP(\theta_R)^{k} k^{\beta+1}}{k!} \to0$ as $R\to\infty$,
this shows that
		\begin{equation}\label{estimate3}
		\sum_{r_3=r_2+r_1}^\infty \frac{R^{r_3+\beta+1} r_3^{\beta+1}}{r_3!} \PP(\theta_R)^{(r_3-r_2)}
		\le C_\beta \theta_R^{2\beta+2} R^{c\Theta_R} \frac{R^{r_1}\PP(\theta_R)^{r_1}}{r_1!}
	\end{equation}

	By combining the estimates \eqref{estimate1}--\eqref{estimate3}, we have shown when $R$ is sufficiently large that
	\begin{align*}
		\Pr{\Xi_R \ge x \Theta_R}
		& \le C_\beta R^{c\Theta_R} \theta_R^{2\beta+2} \biggl( \prod_{k=R+c_\beta}^{\theta_R} \PP(k) \biggr) \sum_{r_1,r_2\ge 0} \binom{R+c_\beta}{r_1} \frac{\bigl(R \PP(\theta_R)\bigr)^{r_1}}{r_1!} \binom{x\Theta_R}{r_2} \\
		& \le C_\beta (2R)^{c\Theta_R} \theta_R^{2\beta+2} \biggl( \prod_{k=R+c_\beta}^{\theta_R} \PP(k) \biggr) \sum_{r_1\ge 0} \binom{R+c_\beta}{r_1} \frac{\bigl(R \PP(\theta_R)\bigr)^{r_1}}{r_1!} \\
		&\le 2C_\beta (2R)^{c\Theta_R} \theta_R^{2\beta+2} \biggl(\prod_{k=R+c_\beta}^{\theta_R} \PP(k) \biggr),
	\end{align*}
	where we used that $\sum_{r_1\ge 0} \binom{R+c_\beta}{r_1} \frac{(R \PP(\theta_R))^{r_1}}{r_1!} \to 1$ as $R\to\infty$ because of the condition \eqref{eq:JLMSizePOutsideWindow}.
	Since we assume that $\frac{\Theta_R}{\sqrt{R \log R}} \to \infty$, by definition of $v_R$, we obtain
	\[
		\limsup_{R\to\infty} \frac{\log(2C_\beta (2R)^{x\Theta_R} (\theta_R)^{2\beta+2})}{\Theta_R v_R}
		= \limsup_{R\to\infty} \frac{\log R}{v_R} = 0.
	\]

By \cref{lem:JLMLeadingOrderTerm}, we conclude that for any $x\ge 0$
\[
\limsup_{R\to\infty} \frac{1}{\Theta_R v_R} \log \Pr{\Xi_R \ge x\Theta_R}
\le \limsup_{R\to\infty} \frac{1}{\Theta_R v_R} \log \biggl(\prod_{k=R+c_\beta}^{\theta_R} \PP(k) \biggr)
\le -\int_0^x J_\gamma(y) \dif y.
\]
This completes the proof of \cref{prop:JLMLB} and of \cref{thm:ldp} (see \cref{prop:JLMUB}).
\end{proof}


\appendix

\section{Proof of \texorpdfstring{\cref{lem:ModulusPoints}}{Lemma \ref{lem:ModulusPoints}}}
\label{sec:ModulusPoints}

A point process $\xi$ on $\bC$ is a random Borel measure which satisfies $\xi(A) \in \bN_0$ for any bounded Borel set $A \subseteq \bC$. The distribution of $\xi$ is uniquely characterized by its finite-dimensional distributions $(\xi(A_1), \dotsc, \xi(A_n))$ for disjoint bounded sets $A_1, \dotsc, A_n$, $n\in\bN$. For an introduction to the basic theory of point processes we refer to \cite[Chapter 9]{DV08IntroductionPointProcessesII}.
Let $\mu$ be a reference Radon measure on $\bC$ and let $\mathcal{Z} = \supp(\xi)$ denote the atoms of $\xi$.
If they exist, we can define the intensity functions $(f_n)_{n\in\bN}$ of $\xi$ with respect to $\mu$ via
\begin{equation} \label{corrfct}
	\Ex[\bigg]{\sum_{\set{z_1, \dotsc, z_n} \subseteq \mathcal{Z}} g(z_1,\dotsc,z_n)}
	= \frac{1}{n!} \int_{\bC^n} g(z_1,\dotsc,z_n) f_n(z_1,\dotsc,z_n) \dif\mu(z_1) \dotsm \dif\mu(z_n),
\end{equation}
for any (symmetric) function $g\colon\bC^n \to \bR_+$ with compact support and $n\in\bN$. In general, the intensity functions need not uniquely characterize a point process.
However, if $(A_k)_{k\in\bN}$ is an increasing sequence of sets exhausting $\bC$ such that
\begin{equation} \label{cnk}
c_{n,k} = \int_{A_k^n} f_n(z_1,\dotsc,z_n) \dif\mu(z_1)\dotsm \dif\mu(z_n) =
n! \Ex[\big]{\# \set[\big]{ \set{ z_1, \dotsc, z_n} \subseteq \mathcal{Z}\cap A_k}}
\end{equation}
satisfy $\limsup_{n\to\infty} (c_{n,k}/n!)^{1/n} < \infty$ for all $k\in\bN$, then by \cite[Proposition 5.4.VII]{DV03IntroductionPointProcessesI}, the joint intensities uniquely characterize the point process $\xi$.

\begin{remark} \label{rk:lawuniqueness}
Recall that for a determinantal process with a kernel $K$ with respect to $\mu$, the joint intensities are given by $f_n(z_1,\dotsc,z_n) = \det_{n\times n} K(z_i,z_j)$ for $z_i\in\bC$.
In particular, if the kernel $K$ is locally trace class, we verify that
$c_{n,k} \le \tr(KA_k)^n <\infty$ for all $n\in\bN$ and any sequence of bounded sets $(A_k)_{k\in\bN}$. Hence, if $K$ is locally trace class, it uniquely characterizes the law of the determinantal process $\xi$.
Moreover, it is not required that the kernel $K$ is Hermitian symmetric.
\end{remark}

\begin{proof}[Proof of \cref{lem:ModulusPoints}]
Let $\mu$ be a rotation-invariant Radon measure on $\bC$ and denote $\mathcal{Y} = \set[\big]{ (k,j) \given j\in\set{1, \dotsc, \ell_k}, k\in\bZ }$.
Let us consider the family of projection kernels (with respect to $\mu$),
\[
K_S(z,w)=\sum_{(k,j)\in S } \rho_{k,j}\bigl(\abs{z}\bigr) z^k \rho_{k,j}\bigl(\abs{w}\bigr)\bar{w}^k , \qquad S\subseteq \mathcal{Y},
\]
where $\rho_{k,j} \colon \bR_+ \to \bR$ satisfy $\int_{\bC} \rho_{k,j}^2(\abs{z}) \abs{z}^{2k} \dif \mu(z) =1$ for all $(k,j)\in \mathcal{Y}$.
Our goal is to characterize the law of the radii of the atoms of the determinantal process $\xi$ with correlation kernel $K_{\mathcal{Y}}$.
The intensity of this point process is $\nu( z) = K_\mathcal{Y}(z,z)$ is a rotation-invariant function, and we assume that $\nu \in L^1_{\operatorname{loc}}(\bC, \mu)$---this is equivalent to the local trace-class condition.
The proof of \cref{lem:ModulusPoints} is inspired from that of Kostlan's \cref{thm:Kostlan}.
Since $\mu$ is rotation-invariant, by the disintegration theorem, we can write $\dif\mu(z) = \dif\widehat{\mu}(r) \frac{\dif \theta}{2\pi}$ for $z= r e^{\i \theta}$, where $\frac{\dif \theta}{2\pi}$ denotes the uniform measure on $\bT = [0,2\pi]$.
Then $\widehat{\mu}$ is a Radon measure and
$\dif\lambda_u(r) = \rho_u^2(r) r^{2k} \dif\widehat{\mu}(r)$ are probability measures on $\bR_+$ for all $u=(k,j)\in \mathcal{Y}$.
Moreover, we will use that $\sum_{u\in \mathcal{Y}} \frac{\dif \lambda_u}{\dif \widehat{\mu}} = \widehat\nu$ with $\widehat\nu \in L^1_{\operatorname{loc}}(\bR_+, \widehat\mu)$ is given by
$\dif\nu(z) = \dif\widehat{\nu}(r) \frac{\dif \theta}{2\pi}$ (according to the disintegration theorem).

We let $\set{Z_u}_{u\in \mathcal{Y}}$ be the collection of atoms of $\xi$. Then, $\set{|Z_u|}_{u\in \mathcal{Y}}$ form a point process on $\bR_+$ and it is rather straightforward to verify that its intensity functions (with respect to the measure $\widehat{\mu}$ defined above) are given by
\begin{equation} \label{fn}
f_n(r_1,\dots, r_n) = \int_{\bT^n} \det_{n \times n} \bigl[ K_\mathcal{Y}(r_i e^{\i \theta_i} , r_j e^{\i \theta_j}) \bigr] \frac{\dif\theta_1}{2\pi}\cdots \frac{\dif\theta_n}{2\pi} ,
\qquad r_1,\dots, r_n \in \bR_+ , \qquad n\in\bN.
\end{equation}
This follows from rotation-invariance and the decomposition of $\mu$.
Note that $f_n \colon \bR_+^n \mapsto \bR_+$ are locally integrable symmetric functions.
In order to prove \cref{lem:ModulusPoints}, it suffices to show that for any $n\in\bN$,
\begin{equation} \label{pcf}
f_n = \sum_{\substack{S_n \subseteq \mathcal{Y}\\|S_n| = n}} \sum_{\sigma\in\cS(S_n)} \prod_{u \in S_n} \frac{\dif \lambda_{\sigma(u)}}{\dif \widehat{\mu}},
\end{equation}
where $\cS(S_n)$ is the group of permutations of the set $S_n \subseteq \mathcal{Y}$.
Indeed, by \eqref{corrfct}, we verify that the RHS of \eqref{pcf} are exactly the intensity functions of the point process $ \set{X_u}_{u \in \mathcal{Y}}$, where $X_u$ are independent (positive) random variables with law $\lambda_u$.
Note that by assumptions, \eqref{pcf} defines locally integrable functions. Indeed, we have
$0 \le f_n(r_1,\dots, r_n) \le \prod_{j=1}^n \sum_{u\in \mathcal{Y}} \frac{\dif \lambda_u}{\dif \widehat{\mu}}(r_j) = \widehat\nu(r_1)\cdots\widehat\nu(r_n)$ for any $r_1,\dots, r_n \in \bR_+$ and $n\in\bN$.
Moreover, using the notation \eqref{cnk}, this bound shows that
$c_{n,k} \le \widehat\nu(A_k)^n$ for all $n\in\bN$ and any sequence of bounded sets $(A_k)_{k\in\bN}$.
Hence, by the above discussion, the point process $ \set{X_u}_{u\in \mathcal{Y}}$ is uniquely characterized by its intensity functions.

Fix $n\in\bN$. By the Cauchy-Binet formula, it holds for any $z_1, \dots, z_n\in\bC$,
\begin{equation} \label{cbf}
\det_{n \times n} \bigl[ K_\mathcal{Y}(z_i , z_j) \bigr]
= \sum_{\substack{S_n \subseteq \mathcal{Y}\\|S_n| = n}}
\det_{n\times n}\big[ K_{S_n}(z_i,z_j) \big].
\end{equation}

Note that since the kernels $ K_{S_n}$ are all positive definite, the sum on the RHS of \eqref{cbf} converges uniformly as a locally integrable function.
Moreover, by combining \eqref{fn} and \eqref{cbf}, it follows from Fubini's theorem that for any $ r_1,\dots, r_n \in \bR_+$,
\begin{equation} \label{fn2}
f_n(r_1,\dots, r_n) = \sum_{\substack{S_n \subseteq Y\\|S_n| = n}} \int_{\bT^n} \det_{n\times n}\big[ K_{S_n}(r_i e^{\i \theta_i} , r_j e^{\i \theta_j}) \big] \frac{\dif\theta_1}{2\pi}\cdots \frac{\dif\theta_n}{2\pi}.
\end{equation}

Then with $S_n = \set{u_1, \dotsc, u_n}$ and $u_j =(k_j, \cdot)$ for $j\in\set{1, \dotsc, n}$, by expanding the previous determinant, it holds
 \begin{align} \notag
 \det_{n\times n}\big[ K_{S_n}(r_i e^{\i \theta_i} , r_j e^{\i \theta_j}) \big]
& = \det_{n\times n} \bigl[ \rho_{u_j}(r_i) (r_i e^{\i \theta_i})^{k_j} \bigr] \overline{\det_{n\times n} \bigl[ \rho_{u_j}(r_i) (r_i e^{\i \theta_i})^{k_j} \bigr] } \\
& \label{Lexp} = \sum_{\sigma, \tau\in\cS_n} \prod_{j \in \set{1, \dotsc, n}} \rho_{u_{\sigma(j)}}(r_k) \rho_{u_{\tau(j)}}(r_k) (r_j e^{\i \theta_j})^{k_{\sigma(j)}} \overline{(r_j e^{\i \theta_j})^{k_{\tau(j)}}},
\end{align}
where $\cS_n = \cS(\set{1, \dotsc, n})$.
Using the orthogonality condition
$\int_{\bT} e^{\i k \theta} \overline{\textstyle e^{\i j \theta}} \frac{\dif\theta}{2\pi} = \delta_{kj}$ for $k,j \in \bZ$, by Fubini's theorem, only the terms on the RHS of \eqref{Lexp} for which $\sigma =\tau$ contribute to the integrals \eqref{fn2}. This shows that for any $ r_1,\dots, r_n \in \bR_+$,
\begin{equation*}
f_n(r_1,\dots, r_n) = \sum_{\substack{S_n \subseteq \mathcal{Y}\\|S_n| = n}} \sum_{\sigma \in\cS_n} \prod_{j \in \set{1, \dotsc, n}}
\rho_{u_{\sigma(j)}}^2(r_k) r_j^{2k_{\sigma(j)}}.
\end{equation*}
By the definitions of the probability measures $(\lambda_u)_{u\in \mathcal{Y}}$, we obtain \eqref{pcf} which completes the proof.
\end{proof}

\section{Explicit covariance computation for the Ginibre ensemble}
\label{sec:CovarianceGinibre}

Interestingly, there exists an explicit formula for the covariance of the infinite Ginibre ensemble.
\begin{proposition}\label{lem:GinibreCovariance}
	Consider the infinite Ginibre ensemble $\xi$. For $0\le s\le t$, it holds
	\begin{equation*}
		\Cov[\big]{\xi(D_s)}{\xi(D_t)}
		= e^{-s^2-t^2}\bigl( s^2I_0(2st) + stI_1(2st) \bigr) - \bigl(t^2-s^2\bigr)K(s^2,t^2),
	\end{equation*}
	where $K(s,t) = \int_0^s e^{-t-x} I_0(2\sqrt{tx}) \dif x$ and $I_\nu$ denotes the modified Bessel function of kind $\nu\in \bN_0$. In particular, as $r\to\infty$,
	\begin{equation*}
		\Var[\big]{\xi(D_r)}
		= r^2e^{-2r^2}\bigl( I_0(2r^2)+I_1(2r^2) \bigr)
		= \frac{r}{\sqrt{\pi}} + O(\frac{1}{r}).
	\end{equation*}
\end{proposition}
\begin{proof}
	The proof follows from the identity \eqref{eq:GammaPoisson} and a careful treatment of sums. Denote $S=s^2$ and $T=t^2$. By reordering the sum one obtains
	\begin{align*}
		\Cov[\big]{\xi (D_s)}{\xi (D_t)}
		&= \sum_{k=1}^\infty \Pr[\big]{\Gamma_k \le S}\Pr[\big]{\Gamma_k > T}\\
 		&= \sum_{k=1}^\infty \sum_{i=k}^\infty \sum_{j=0}^{k-1} \frac{e^{-S}S^i}{i!}\frac{e^{-T}T^j}{j!}\\
 		&= \sum_{j=0}^\infty \sum_{i=j+1}^\infty (i-j) \frac{e^{-S}S^i}{i!}\frac{e^{-T}T^j}{j!}\\
 		&= \sum_{j=0}^\infty S \frac{e^{-S}S^j}{j!}\frac{e^{-T}T^j}{j!} + \sum_{j=0}^\infty T \frac{e^{-S}S^{j+1}}{(j+1)!}\frac{e^{-T}T^j}{j!} - (T-S)\sum_{j=0}^\infty \sum_{i=j+1}^\infty
 		\frac{e^{-S}S^i}{i!}\frac{e^{-T}T^j}{j!}.
 	\end{align*}
 	Next, recall the definition of the modified Bessel function of first kind:
 	\begin{equation*}
		I_\nu(x) = \sum_{k=0}^\infty \frac{1}{(k+\nu)!k!}(\frac{x}{2})^{2k+\nu}
	\end{equation*}
	for $\nu\in\bN_0$. Hence, the first two summands are given by $Se^{-S-T}I_0(2\sqrt{ST})$ and $\sqrt{ST}e^{-S-T}I_1(2\sqrt{ST})$, respectively. For the last summand, notice that
 	\begin{align*}
 		\MoveEqLeft (T-S)\sum_{j=0}^\infty \sum_{i=j+1}^\infty \frac{e^{-S}S^i}{i!}\frac{e^{-T}T^j}{j!}
 		= (T-S)\sum_{j=0}^\infty \frac{e^{-T}T^j}{j!} \int_0^S e^{-x}\frac{x^j}{j!}\dif x\\
 		&= (T-S) \int_0^S e^{-T-x} I_0(2\sqrt{Tx}) \dif x
 		= (T-S)K(S,T).
	\end{align*}
	Putting all this together the claimed formula for the covariance follows. To obtain the expansion for the variance use the classical expansion of the Bessel functions
	\begin{equation}\label{eq:BesselAsymptotic}
		I_\nu(r) = \frac{e^r}{\sqrt{2\pi r}}\Bigl(1+O_\nu\Bigl(\frac{1}{r}\Bigr)\Bigr),
	\end{equation}
	as $r\to\infty$ for $\nu\in\bN_0$.
\end{proof}

The explicit formula of the covariance allows us to present an alternative proof for the functional central limit theorems \cref{thm:GinibreFCLTFDD} and \cref{thm:GinibreFCLTSkorokhod} by calculating the covariance asymptotic by hand and applying \cref{lem:FCLTFDD} and \cref{thm:FCLTSkorokhod} directly.
\begin{lemma}
	For $0<s\le t$, as $r\to\infty$, it holds that
	\begin{equation*}
		\Cov[\big]{\xi(D_{sr})}{\xi(D_{tr})}
		= \frac{tr}{\sqrt{\pi}}\ind_{s=t}\bigl(1+o(1)\bigr).
	\end{equation*}
	For $-\infty<s\le t<\infty$, as $r\to\infty$, it holds that
	\begin{equation*}
		\Cov[\big]{\xi(D_{r+s})}{\xi(D_{r+t})}
		= \Bigl(\frac{r}{\sqrt{\pi}} e^{-(t-s)^2} - 2r(t-s)\Pr[\big]{\Normaldist_{0,1} > \sqrt{2}(t-s)}\Bigr) \bigl(1+o(1)\bigr).
	\end{equation*}
\end{lemma}
\begin{proof}
	We evaluate the different summands in the explicit formula for the covariance in \cref{lem:GinibreCovariance} separately Consider first the macroscopic regime $(\xi(D_{tr}))_{t\in\bR_+}$. If $s=t$ we have seen already that
	\begin{equation*}
		\Var[\big]{\xi(D_{tr})} = \frac{tr}{\sqrt{\pi}}\bigl(1+o(1)\bigr).
	\end{equation*}
	In case $s<t$, the asymptotic of the Bessel functions \eqref{eq:BesselAsymptotic} implies that
	\begin{equation*}
		e^{-r^2(s^2+t^2)} \bigl(r^2s^2I_0(2r^2st) + r^2stI_1(2r^2st)\bigr)
		= e^{-r^2(t-s)^2}\Bigl(\frac{s^2}{\sqrt{4\pi st}} + \frac{1}{\sqrt{4\pi st}}\Bigr)\bigl(1+o(1)\bigr).
	\end{equation*}
	For the other summand notice $K(x,y)\in \intcc{0,1}$ for all $x,y\in \bR_+$ and apply $I_0(x) \le \frac{e^x}{\sqrt{x}}$ for all $x\in \bR_+$ so that
	\begin{align*}
		\MoveEqLeft K\bigl(r^2s^2,r^2t^2\bigr)
		= \int_0^{r^2s^2} e^{-r^2t^2-x} I_0\bigl(2rt\sqrt{x}\bigr) \dif x
		\le \int_0^{r^2s^2} \frac{1}{2rt\sqrt{x}} e^{-(rt-\sqrt{x})^2} \dif x\\
		&\le e^{-r^2(t-s)^2} \int_0^{r^2s^2} \frac{1}{2rt\sqrt{x}} \dif x
		= e^{-r^2(t-s)^2} \frac{s}{t}.
	\end{align*}
	Hence, for $s<t$ the covariance converges to zero exponentially fast and \cref{lem:FCLTFDD} applies as claimed.

	Now, consider the microscopic regime $(\xi(D_{r+t}))_{t\in\bR}$. Fix $s\le t$ and assume that $r$ is large enough such that $r+s>0$ and $r+t>0$. For the first summand we apply again the asymptotic \eqref{eq:BesselAsymptotic} to obtain
	\begin{align*}
		\MoveEqLeft e^{-(r+s)^2-(r+t)^2} \Bigl( (r+s)^2 I_0\bigl(2(r+s)(r+t)\bigr) + (r+s)(r+t)I_1\bigl(2(r+s)(r+t)\bigr) \Bigr)\\
		&= \frac{r}{\sqrt{\pi}} e^{-(t-s)^2} \bigl(1+o(1)\bigr).
	\end{align*}
	For the second summand, the asymptotic of $K$ can be found in \cite[12.1.5 formula 8]{L62IntegralsOfBesselFunctions} and is given by
	\begin{equation*}
		K\bigl(x^2,y^2\bigr) = -\frac{1}{2}e^{-(y-x)^2}\Bigl( e^{-2xy}I_0(2xy) - 2e^{(y-x)^2}\Pr[\big]{\Normaldist_{0,1}\ge \sqrt{2}(y-x)} + O\Bigl(\frac{1}{xy}\Bigr) \Bigr),
	\end{equation*}
	as $xy \to \infty $ with $\frac{y}{x}\ge 1$ and $\frac{(y-x)^2}{xy} = o(1)$. Applying this yields
	\begin{equation*}
		\bigl((r+t)^2-(r+s)^2\bigr) K\bigl((r+s)^2,(r+t)^2\bigr)
		= 2r(t-s)\Pr[\big]{\Normaldist_{0,1} \ge \sqrt{2}(t-s)} \bigl(1+o(1)\bigr).
	\end{equation*}
	The claimed asymptotic for the covariance of $(\xi (D_{n+\sqrt{n}t}))_{t\in \bR }$ can be concluded by combining the above asymptotic results.
\end{proof}


\bibliography{./DPP_Literature}

\end{document}